\documentclass[12pt,a4paper]{article}
\usepackage[margin=1.05in,includeheadfoot,headheight=15pt]{geometry}
\usepackage{microtype}
\usepackage{latexsym}
\usepackage{needspace}
\begin{document}
\pagestyle{headings}

\begin{titlepage}

\vspace*{.2cm}
{\large\bf
\begin{center}
UNIVERSITATEA DE STAT DIN MOLDOVA\\
State University of Moldova\\
\end{center}
}
\vspace*{.85cm}  
{\large
\begin{center}
Doctoral Thesis
\end{center}
\begin{flushright}
                            As a manuscript\hbox{\hspace{2.0cm}}\\
                                 CZU  519.2 \hbox{\hspace{2.0cm}}
\end{flushright}
}
  \vspace*{0.6cm}  
{
\large
\begin{center}
Hlibiciuc Alexandru Anatolie
							 \\[0.6cm] 
%\end{center}

\Large \bf
%\begin{center}
The Characteristics of Finite Projections\\
of Certain Stochastic Processes and Queues\\
\end{center}
}
\vspace{.85 cm}
{\large
\noindent
  01.01.05 - ``Probability Theory \hfil\newline
  \hbox{\hspace{1.5cm}} and Mathematical Statistics''

\vspace{.85 cm}
\noindent
Doctoral thesis in physical and mathematical sciences

\vspace{0.7 cm}
  \begin{flushright} Academic supervisor, 
  \hbox{\hspace{2.0cm}}\\  Associate Professor Dr A. Leahu \hbox{\hspace{2.0cm}}
 \end{flushright} 
}

{\large
\vfill
\begin{center}
CHI\c{S}IN\u{A}U, 1998\\[0.3cm]
Corrected English edition, 2026.
\end{center}
}
\end{titlepage}

\newpage
\setcounter{page}{2}
\section*{Preface to the corrected English edition}
\addcontentsline{toc}{section}{Preface to the corrected English edition}
\thispagestyle{myheadings}
\markright{}

This volume is a corrected English translation of my doctoral thesis, completed in Romanian at the State University of Moldova in 1998 and defended on 15 January 1999. Its purpose is to make the work accessible to a wider readership while preserving its original structure, notation and mathematical approach as far as possible.

The preparation of this edition included comparison with the Romanian original, revision of the English text and mathematical checking. Errors identified in statements, proofs and formulas have been corrected. Some corrections require additional hypotheses or restrictions on the original claims. The guiding principle has been to make the smallest changes necessary for mathematical correctness and readability.

The thesis is presented in its historical setting. Corrections and later additions belong to the present edition and should not be attributed retrospectively to the 1998 text. The original Romanian version remains the reference for what was stated and proved at that time.

In retrospect, I consider ``recurrent projection'' a more descriptive term for the construction called ``finite projection'' in this thesis. The construction conditions each regeneration cycle on having finite duration, producing a process that is recurrent with respect to the chosen regeneration structure. Here, ``finite'' refers to the duration of each cycle, not to the state space or the mean cycle length. I have retained the original terminology in this edition for continuity with the thesis and the earlier publications.

I remain deeply grateful to my supervisor, Professor Alexei Leahu, for giving me the opportunity to prepare this thesis and for his support in organizing its defence. He and the late Academician Petru Soltan, who chaired the examining committee, were also my professors at university. I remember their teaching with gratitude and value the part they played in my mathematical education. Professor Victor Chepoi also had a lasting influence on my mathematical development, for which I am sincerely grateful.

I remember with particular respect and gratitude the late Marius Iosifescu and Ioan Cuculescu, members of the Romanian Academy and distinguished figures in Romanian probability. I thank Marius Iosifescu and Vladimir Dragalin for their favourable reviews, and Ioan Cuculescu and Arcadie Barbarosie for serving as the official examiners of my thesis. Their support at this important stage of my life meant a great deal to me.

Returning to this thesis to prepare its corrected English edition has brought back memories of my teachers and of those who helped me bring this work to completion and defence. Looking back after so many years, I feel even more deeply how much I owe them.

\begin{flushright}
Alexandru Hlibiciuc\\
2026
\end{flushright}

\newpage
\thispagestyle{myheadings}  
\markright{}
\tableofcontents
\newpage
%%%%%%%%%%%%%%%%%%%%%%%%%%%%%%%%%%%%%%%%%%%%%%%%%%%%%%%%%%%%%%%%%%%%%%%%
%%%%%%%%%%%%%%%%%%%%%%%%%%%%%%%%%%%%%%%%%%%%%%%%%%%%%%%%%%%%%%%%%%%%%%%%
%definitii 
%
\newtheorem{def1}{ Definition }[section]
\newtheorem{prop}{Proposition}[section]
\newtheorem{rem}{Note}[section]
\newtheorem{lema}{Lemma}[section]
\newtheorem{teor}{ Theorem }[section]
  \section*{ Introduction }  \setcounter{equation}{0}
  \addcontentsline{toc}{section}{ Introduction }
\thispagestyle{myheadings}  
\markright{\hspace{1cm} \it Introduction}

% RO-INTRO-01
Probabilistic methods, especially those of the theory of stochastic processes, are successfully used in various fields of practical activity for constructing mathematical models of phenomena of a random nature. In the applied domains of probability theory, such as queueing theory, risk theory, etc., methods from renewal theory and the theory of regenerative processes are widely used \cite{part1_1},\cite{risc_1}.

The present work contains a study of a conditioning method for regenerative processes, referred to by its authors as the {\it finite projection} \cite{part1_4}. The aim of this method is to simplify the study of transient regenerative processes and unstable queueing systems by replacing them with recurrent models and, when the projected mean regeneration cycle is finite, stable models with equilibrium distributions. The notion of finite projection is based on a natural method of conditioning regenerative processes on certain events of probability zero through a limiting construction.

Thus, two main aspects emerge, which pertain to the research topics presented in this work: the first concerns the approximation of regenerative processes, particularly queueing processes, and the second addresses the conditioning of stochastic processes, particularly regenerative processes and queueing systems.

The finite projection method compares certain transient processes with recurrent counterparts. The finite-time bounds below specify the regimes and observations for which this comparison is effective; the long-time behavior of the two laws can differ.

% RO-INTRO-02
The problems of approximating some stochastic processes (systems) with others are well known in the literature and encompass several fundamental research directions. Among those closely related to the topics presented in this work are studies on the behavior of queueing systems under heavy traffic conditions \cite{Borovkov_2}, \cite{Asmussen_1}, \cite{Asmussen_2}, \cite{Asmussen_3}.

It is precisely under heavy traffic conditions that the approximation of queueing systems described by regenerative processes by their finite projections makes sense. Unlike most cited works, in the present study, we do not focus on the limiting distributions of the random characteristics of queueing systems under heavy traffic conditions. Instead, the main objective is the determination and study of the finite projections of certain classes of queueing processes, with a particular interest in the quality of approximating queueing systems by finite projections under heavy traffic conditions.

% RO-INTRO-03
Another aspect of the vast domain of research concerning the approximation of models, adjacent to the topic of this work, pertains to the continuity of stochastic systems, which has aroused great interest among researchers in recent decades. The main reason is the practical importance of the continuity of random models \cite{Kalashnikov_1}, due to the property of continuous systems to maintain their qualitative properties, and with small errors, their quantitative properties, under small parameter changes. The basic results concerning the continuity of stochastic systems are presented in the monographs \cite{Kalashnikov_1}, \cite{Rachev}. A key step in examining continuity problems of systems is the choice of a metric defined on the space of examined stochastic models.

% RO-INTRO-04
Building on the idea of system continuity, the first part of the work undertakes a study of the continuity of the finite projection of regenerative processes. Our main objective is to highlight the optimality of the finite projection in the sense of the introduced metrics, which essentially represent modifications of the total variation distance and the Hellinger distance between probabilities. With this, the quality of finite projections in approximating initial processes over finite and infinite time intervals is examined. While the continuity of processes and stochastic systems over both finite and infinite time intervals \cite{Kalashnikov_1}, \cite{Karr} has been studied in the works of many researchers, the specificity of the problems examined in this work allows, in addition to demonstrating the continuity of finite projection, the determination of asymptotically sharp finite-time errors under the stated observability and parameter assumptions.

% RO-INTRO-05: conditioning and the scope of the thesis
Another aspect of this research is the conditioning of stochastic processes,
a subject with a long history. Kolmogorov's measure-theoretic formulation of
conditional probability is based on the Radon--Nikodym theorem
\cite{Kolmogorov_base}. This rigorous definition does not, by itself,
provide a universal constructive rule for conditioning on events of
probability zero. Different limiting procedures that appear natural can
lead to different conditional laws; see, for example,
\cite{Bertoin_Doney}, \cite{Bertoin_Doney2}, \cite{Bertoin_1}.
Similar questions arise for finite projections of regenerative processes.
For a transient irreducible time-homogeneous Markov chain, choosing a
regeneration state gives a finite projection under which that state is
recurrent. On an infinite state space this gives a family indexed by the
states, although different states need not give distinct projections; see
Section~\ref{proiect_arb}. R\'enyi proposed an alternative axiomatic
framework for conditional probability in
\cite{Renyi_1}, \cite{Renyi_2}, \cite{Renyi_3}. This framework extends
the scope of the usual probability-space formulation. The relation between
the two frameworks is not examined here: we treat conditioning as the
application of a projection operator that transforms one probability
measure into another.

% RO-INTRO-06: historical literature review, retained in its 1998 setting
A short introduction cannot cover the entire literature on conditioning
stochastic processes, so the references given here are selective and may
not be representative. A principal concern of these studies is to identify
properties and relationships shared by the original and conditioned
processes. Conditioning Brownian motion has received particularly extensive
attention. We mention the work of Doob \cite{Doob_1}, McKean
\cite{McKean_1}, and Williams \cite{Williams_1}, together with the later
works \cite{Banuelos_1}, \cite{Bass_1}, \cite{Cranston_1},
\cite{Cranston_2}, \cite{Day_1}, \cite{Griffin_1}, \cite{Liao_1},
\cite{Pinsky_1}, and \cite{Zhao_1}. These results also motivated analogous
studies of random walks; see, for example,
\cite{Bertoin_Doney}, \cite{Bertoin_Doney2}, \cite{Chaumont_1},
\cite{Doney_1}, \cite{Keener_1}, \cite{Poleshchuk_1}, and
\cite{Poleshchuk_2}. During the decade preceding this thesis, research on
the existence and properties of quasi-stationary distributions of Markov
chains became increasingly active, although the general problem remained
unresolved. This line of research was developed in the work of Seneta and
Vere-Jones \cite{Vere-Jones_1}, \cite{Seneta_1}. Among the later
publications available at the time, we mention
\cite{Kesten_1}, \cite{Pollett_1}, and \cite{Kesten_2}. These works
establish connections between quasi-stationarity, the lifetime of a Markov
chain, and conditioning.

% RO-INTRO-07: h-transforms, duality, and Martin boundaries
In a number of cases, conditioning a Markov process produces an
$h$-transform. Conditioning can also simplify the determination of Martin
boundaries. It has revealed relationships between dual processes and
queueing systems; see, for example, Asmussen \cite{Asmussen_4}, Fakinos
\cite{Fakinos_1}, Leahu \cite{l_pnm7}, and Tavar\'e \cite{Tavare_1}.
Since finite projections of regenerative processes are obtained by
conditioning, it is natural to investigate their connections with duality,
Martin boundaries, reversibility, and, for Markov processes, the existence
of invariant measures. In particular, finite projections of Markov chains
are related to $h$-transforms through limits along suitable sequences of
regeneration states. In
Section~\ref{spat_proiect}, these are transforms associated with Martin
boundary kernels, with possible killing; the assertion does not concern
all arbitrary $h$-transforms. The results of the thesis extend or develop
earlier work in this direction, including
\cite{Fakinos_1}, \cite{l_pnm7}, \cite{part1_4}, and \cite{Veech_1}.

% RO-INTRO-08: the general definition
The first chapter gives a new formulation of finite projection for
regenerative processes in Section~\ref{par_1_2}. It retains the substance
of the definition in \cite{part1_4} and permits a general study of the
properties of finite projections.

% RO-INTRO-09: scope and results of independent interest
The remaining chapters describe finite projections for specific classes
of stochastic processes: birth-and-death processes, Markov chains, and
semi-Markov processes, including skip-free processes. Several results are
also of independent interest. These include the determination of a
birth-and-death process from prescribed distributions of the times spent
in the integer sets $[0,n]$ and $[n+1,\infty)$, and sharp bounds for the
relative error in approximating continued fractions by their convergents.

% RO-INTRO-10: regenerative cycles, including an infinite final cycle
The exposition begins with a definition of a regenerative process that
allows infinite regeneration cycles, as in \cite{part1_7} and
\cite{part1_4}. This definition is useful because models of queues,
inventories, and system reliability are often described by regenerative
processes \cite{part1_1}, \cite{l_frc1}. Such a process is constructed
from independent, identically distributed cycle data, apart from a
possibly different initial cycle. The cycles describe the system over
successive time intervals and are concatenated until an infinite cycle
occurs. The regeneration times $S_n$ that
separate these intervals form an embedded renewal process
$S=\{S_n\}$ \cite{c2}. In the classical definition, these times are
finite \cite{part1_3}. A wider range of models, including supercritical
queues with traffic intensity greater than one under the assumptions
considered here, may regenerate only finitely many times and then have an
infinite final cycle. Definitions allowing this possibility were given in
\cite{part1_7} and \cite{part1_4}. Section~\ref{par_1_1} provides
examples of such processes, which we call transient regenerative
processes. Finite projection is of interest for these processes. We call
regenerative processes recurrent when their regeneration cycles, including
the initial delay when present, are almost surely finite.

% RO-INTRO-11: purpose of Section 1.2
Section~\ref{par_1_2} defines the finite projection of a regenerative
process. The procedure replaces a transient regenerative process by a
recurrent one, its finite projection, with the aim of approximating the
original process.

% RO-INTRO-12: approximation, renewal characteristics, and optimality
Section~\ref{par_1_3} gives a detailed analysis of this approximation.
It studies distances between the original and projected probability
measures and asymptotically sharp convergence rates in the specified
regimes. It also compares characteristics of the embedded renewal
processes (Theorem~\ref{t_car_dif}). The section develops the
best-approximation interpretation of finite projection relative to the
distances and comparison classes considered there
(Propositions~\ref{p1_r6} and~\ref{p1_r3}), including classes
larger than the probability laws of recurrent regenerative processes.

% RO-INTRO-13: queueing applications
Queueing processes often admit a regenerative description, so finite
projection applies to them. Section~\ref{par_1_4} gives a general
definition for $GI/G/n/\infty$ queues, describes the structure of the
finite projection of an $M/G/1$ queue, and proves its stability.

% RO-INTRO-14: all three sections and both queueing examples of Chapter 2
The second chapter consists of three sections on finite projections of
birth-and-death processes. Section~\ref{par_2_1} establishes auxiliary
results on the time spent in a specified set of states and represents the
Laplace--Stieltjes transforms of its distribution by continued fractions.
These results are then used to determine the structure of finite
projections. A recurrent finite projection need not be positive recurrent:
Section~\ref{par_2_2} gives a null-recurrent example using continued-fraction
expansions of Laplace--Stieltjes transforms. The section also examines
sufficient conditions for positive recurrence and gives necessary and
sufficient conditions for a finite projection to coincide with the
$0$-dual of the original process. The $0$-dual interchanges the birth and
death rates at positive states and retains the birth rate at state zero,
where no death rate is defined. Fakinos \cite{Fakinos_1} introduced a
related notion of duality for queues; for birth-and-death queueing models
the same rate interchange applies away from zero, with the boundary rate
retained in the $0$-dual. Section~\ref{par_2_3} studies finite
projections of queues described by birth-and-death processes and their
characteristics, including positive recurrence and the corresponding
stability conditions. For the finite projection of the $M/M/1$ queue,
it obtains the equilibrium distribution, Laplace transforms of the
state probabilities, and the distribution of the virtual waiting time
under the FIFO discipline. A method of inverting the resulting Laplace
transforms is illustrated. For the $M/M/n$ queue, the chapter describes
the finite projection, its equilibrium distribution, and the distribution
of the first hitting time of the state set $[0,n-1]$.

% RO-INTRO-15: the first two sections of Chapter 3
The third chapter studies finite projections of time-homogeneous Markov
chains on a discrete state space. Its first section describes how finite
projection transforms a transient Markov chain, whose transition matrix
may be substochastic, into a Markov chain with a single recurrent class $E_0^*$
and possibly transient states in $E^*\setminus E_0^*$. The construction
requires a regeneration state with a positive return probability.
The next section examines relations between the characteristics of the
original and projected chains. It derives generating functions of the
projected transition probabilities and, when the original chain is
irreducible, invariant measures of the projected chain.

% RO-INTRO-16: coincidence, invariant measures, and the Martin boundary
For a transient irreducible Markov chain, a finite projection can be
associated with each state. Section~\ref{proiect_arb} gives examples
using different regeneration states and a criterion for two such finite
projections to coincide. It also presents a criterion, expressed in terms
of finite projections, for the existence of invariant measures of the
original transient irreducible chain. This criterion resembles Harris's
criterion in form, although its underlying ideas are different. The
necessity of Harris's condition was proved by Veech \cite{Veech_1}.
Section~\ref{spat_proiect} continues the study of the space of finite
projections and its relation to the Martin boundary of the original
chain.

% RO-INTRO-17: equivalence and reverse-time conditional laws
The next section introduces $\tau$-equivalence of Markov probability
measures relative to a stopping time $\tau$. This generalizes the notion
of equivalence of Markov chains used in \cite{Kerov_O} and reduces to
that notion when $\tau=\infty$. Up to the specified stopping times,
finite projections can be viewed as $\tau$-equivalent Markov chains.
This connects conditioning by finite projection with equivalence of the
conditional laws of the past, given the present state, in reverse time.

% RO-INTRO-18: embedded queueing chains and semi-Markov processes
The final two sections of the third chapter examine duality relations
for the $M/G_E/1$ and $GI/M/1$ queues and their finite projections,
and finite projections of semi-Markov processes. Determining a semi-Markov finite projection reduces to determining the finite projection of its embedded Markov chain; the holding-time laws conditional on successive states remain unchanged.

% RO-INTRO-19: skip-free semi-Markov processes and relative errors
The final chapter is devoted to skip-free semi-Markov processes. It
extends several results for birth-and-death processes, particularly those
concerning the time spent in a given set of states. It also gives a
probabilistic interpretation of continued-fraction convergents that makes
it possible to bound the relative error of approximation by these
convergents.

% RO-INTRO-20: dissemination of the historical results
The results presented in the thesis were reported and discussed at
national and international scientific meetings. Some were published in
the works cited here; see, for example, \cite{l_2:k1} and
\cite{part1_5}.

%se  indic\u{a} formulele pentru transformatele Laplace ale
%  probabilit\u{a}\c{t}ilor de aflare \^{i}n starea dat\u{a} \^{i}n
%  momentul fixat de timp  folosind 

\newpage
  \section{ Regenerative processes and finite projections}
\markright{1\quad Regenerative processes and finite projections}
 \setcounter{equation}{0}
  \subsection{Regenerative processes}\label{par_1_1}
%%
%
% RO-1.1-01: historical opening
Since the publication of the work \cite{reg_r1} by W.L. Smith,
regenerative processes have played an important role in the theory and
applications of stochastic processes. A later survey of the fundamental
results and their generalizations is given by K. Sigman and R.W. Wolff
\cite{reg_r2}.

% RO-1.1-02: classical cycles and the queueing example
In the classical definition, a regenerative process consists of independent,
identically distributed cycles \cite{part1_3}, apart from a possibly
different initial cycle. A standard example is the $GI/G/1$ queue with
traffic intensity $\rho<1$. Each regeneration cycle consists of a busy
period followed by the adjacent idle period; successive cycles begin when
customers arrive to find the system empty. When $\rho>1$, a busy period
can be infinite with positive probability. After an infinite cycle, no
further cycle exists. Thus, when finite and infinite cycles both have
positive probability, the existence of the next observed cycle depends
on whether the preceding cycle is finite; the successive observed cycles
are not independent. Regeneration means that, whenever a cycle ends, the
next cycle starts afresh with the same cycle law. To include infinite
cycles, \cite{part1_4} introduced the
notion of a stopped regenerative process. A similar definition in discrete
time was used by V.V. Kalashnikov \cite{part1_7}. Following the terminology
for Markov chains, we use the term \emph{transient regenerative process}
in place of stopped regenerative process.

% RO-1.1-03: probability space
We work on a canonical measurable space $(\Omega,{\cal F})$ carrying the
regenerative path and its regeneration times, with its uncompleted cylinder
$\sigma$-algebra when a new law is constructed. The initial probability
measure is $P$; either law may be completed separately after construction.

% RO-1.1-04: renewal epochs, delay, nondegeneracy and transience
Recall that, according to \cite{c2}, a sequence of random variables
$S=\{S_n\}_{n\ge0}$ is called a \emph{renewal process} if
$S_0=Y_0$ and $S_n=S_{n-1}+Y_n$, $n\ge1$. The random variables
$Y_0,Y_1,Y_2,\ldots$ take values in $[0,\infty]$ and are independent;
$Y_1,Y_2,\ldots$ are identically distributed, whereas $Y_0$ may have a
different distribution. We assume
$P\{Y_i\in(0,\infty)\}>0$, $i\ge1$. The process is called
\emph{ordinary} if $Y_0=0$ almost surely (a.s.) under $P$, and
\emph{delayed} otherwise. It is called transient (or stopped) if
$P\{S_n<\infty\ \hbox{for all }n\ge0\}<1$.

% RO-1.1-05: regeneration, including the full cycle data
Let $X=\{X(t):t\in T\}$ be a stochastic process on
$(\Omega,{\cal F},P)$, where $T=\{0,1,2,\ldots\}$ or $T=[0,\infty)$.
We call $X$ a \emph{regenerative process} if there exists a possibly
delayed renewal process
$S=\{S_n\}_{n\ge0}=\{Y_0+\cdots+Y_n\}_{n\ge0}$, with values in
$T\cup\{\infty\}$, having the following regeneration property.
Conditional on $S_n<\infty$, the future process
$\{X(S_n+t):t\in T\}$ is independent of the history before $S_n$,
including $S_n$, and has the same distribution as
$\{X(S_0+t):t\in T\}$ conditional on $S_0<\infty$.
This property is understood jointly with the subsequent regeneration
epochs measured from $S_n$: the cycle paths and their durations regenerate
together. We assume $P\{S_0<\infty\}>0$, so the conditional laws are
defined; the renewal assumptions then give $P\{S_n<\infty\}>0$ for
every finite $n$. In particular,
$$
P\{X(S_0+t_1)\in B_1,\ldots,X(S_0+t_k)\in B_k\mid S_0<\infty\}=
$$
$$
=P\{X(S_n+t_1)\in B_1,\ldots,X(S_n+t_k)\in B_k\mid S_n<\infty\}.
\label{p1_r1}
$$
Here $0\le t_1<\cdots<t_k<\infty$ are in $T$, and $B_i$,
$i=1,\ldots,k$, are Borel sets. The times $S_n$, $n\ge0$, are called
\emph{regeneration times}, and $S$ is the renewal process embedded in $X$.

% RO-1.1-06: the working sigma-algebra and defective interval laws
Henceforth $X$ is regenerative and
${\cal F}=\sigma(X(t),S_n:t\in T,n\ge0)$. When the regeneration times
are path functionals, this is simply the path $\sigma$-algebra. We denote
the distribution functions of the renewal intervals by $F_0$ and $F$:
$$
F_0(t)=P\{Y_0\le t\}=P\{S_0\le t\},\qquad
F(t)=P\{Y_n\le t\},\quad n\ge1.
$$
Here $F_0(\infty)$ and $F(\infty)$ denote the limits as finite
$t\to\infty$, namely $P\{Y_0<\infty\}$ and $P\{Y_1<\infty\}$.
We call $X$ transient if its embedded renewal process is transient,
that is, if $F_0(\infty)<1$ or $F(\infty)<1$, and recurrent otherwise.
These terms refer to the chosen embedded renewal process.

Here are two examples of transient regenerative processes.
\begin{description}
% RO-1.1-07: Markov-chain example and its absorbing-state extension
\item[(a)] Let $X=\{X_n:n\ge0\}$ be a time-homogeneous Markov chain
with countable state space $E=\{0,1,\ldots\}$. If $X$ is transient and
irreducible, its regeneration times can be taken to be the successive
visits to a fixed state $l\in E$:
$$
S_0=\inf\{k>0:X_k=l\},\qquad
S_n=\inf\{k>S_{n-1}:X_k=l\},\quad n\ge1,
$$
where $\inf\emptyset=\infty$. For every finite subset $A\subset E$,
$$
\lim_{n\to\infty}P\{X_n\in A\}=0.
$$
The same construction applies when the state space consists of absorbing
states and a transient communicating class containing $l$, provided $l$
can be reached from the initial distribution and has a positive
probability of return. An infinite final regeneration cycle can then
arise either through escape from every finite subset of the transient
class or through absorption in an absorbing state. Here and below we
follow the terminology of \cite{part1_1}.
% RO-1.1-08: reflected random walk and its queueing interpretation
\item[(b)] Let $Z_1,Z_2,\ldots$ be independent and identically
distributed (i.i.d.) random variables, and let $W_0$ be an independent,
finite nonnegative random variable. Define
$$
W_n=\max(0,W_{n-1}+Z_n),\qquad n\ge1.
$$
The process $\{W_n:n\ge0\}$ is a \emph{reflected random walk}.
In a FIFO single-server queue, $W_n$ is the waiting time of the $n$th
customer before service when
$$
Z_n=Y_{n-1}-(T_n-T_{n-1}),
$$
where $Y_n$ is the service time and $T_n$ the arrival time of the $n$th
customer. The service times $Y_0,Y_1,\ldots$ are i.i.d., and the
interarrival times $T_1-T_0,T_2-T_1,\ldots$ are i.i.d. and independent
of the entire service-time sequence. If the mean ${\bf E}Z_1$ is well
defined and positive and $P\{Z_1<0\}>0$, then $\{W_n:n\ge0\}$ is
a transient regenerative process, with regeneration at successive visits
to zero.
\end{description}

% RO-1.1-09: the measurable-transformation proposition
The preservation of the regenerative property under measurable
transformations also holds for transient regenerative processes.
\begin{prop}
Let $X=\{X(t):t\ge0\}$ be a transient regenerative process with
embedded renewal process $S$, taking values in $(E,{\cal E})$.
If $\varphi:E\to E_0$ is measurable from $(E,{\cal E})$ to
$(E_0,{\cal E}_0)$, then $\{\varphi(X(t)):t\ge0\}$ is regenerative
with the same embedded renewal process $S$, and is transient relative
to $S$.
\end{prop}

% RO-1.1-10: proof and the constant-transformation qualification
The result follows by applying the regeneration property to inverse
images under $\varphi$. The transformed process can also be recurrent
with respect to a different embedded renewal process. For example, a
constant transformation ($|E_0|=1$) gives a constant process, which
admits deterministic recurrent regeneration times.

%  \begin{rem} \label{p1_r2}
%  Ciclurile de regenerare ale procesului de regenerare $X(t)$ sunt
%  2-dependente, adic\u{a} v.a. $\{X(t): S_{k}\le t< S_{k+1} \}$ pot
%  depinde numai de v.a. $\{X(t): S_{k-1}\le t< S_k \}$ \c{s}i nu
%  depind de v.a. $\{X(t): S_{m-1}\le t< S_{m} \}$ pentru $m<k$.
%  \end{rem}
%
%#######################################################################  
%
  %
  %
  \subsection{Finite projections of regenerative processes}
\label{par_1_2}

Example (a) shows that a transient regenerative process need not have a
limiting equilibrium distribution. In the transient irreducible case of that example,
the limiting probabilities of individual states are all zero and hence
do not form a probability distribution. This possible loss of mass
motivated the introduction of finite projections of transient
regenerative processes~\cite{part1_4}.

The idea is to construct a new probability measure $P^*$ under which
$X$ is a recurrent regenerative process, with the same embedded renewal
structure $S$, by giving each cycle its original law conditional on
its length being finite. The initial cycle may have a different law
from the subsequent cycles.

For $n\ge1$, put
$$
r_0=P\{S_0<\infty\}=P\{Y_0<\infty\}>0,\;
r=P\{S_n<\infty\mid S_{n-1}<\infty\}=P\{Y_n<\infty\}>0.
$$
Write $p_n=P\{S_n<\infty\}=r_0r^n$, $n\ge0$.

We use the canonical measurable space of Section~\ref{par_1_1},
with measurable concatenation of cycle paths.
The Ionescu--Tulcea theorem~\cite{part1_6} gives the following result.
\begin{prop}\label{p1_r4}
There is a unique probability measure $P^*$ on $(\Omega,{\cal F})$
under which $X$ is regenerative with recurrent embedded renewal
process $S$, and the initial and subsequent cycles have their original
distributions conditional on their lengths being finite.
Moreover, $S_n\to\infty$, $P^*$-almost surely.
For every $n\ge0$, $0\le t_1<\cdots<t_k<\infty$ in $T$ and Borel sets
$B_1,\ldots,B_k$,
$$
P^*\{X(t_1)\in B_1,\ldots,X(t_k)\in B_k,
       S_{n-1}\le t_1<\cdots<t_k<S_n\}=
$$
$$
=P\{X(t_1)\in B_1,\ldots,X(t_k)\in B_k,
       S_{n-1}\le t_1<\cdots<t_k<S_n\mid S_n<\infty\},
$$
where $S_{-1}=0$.
\end{prop}

{\bf Proof.} A cycle consists of its duration and the path on its
half-open interval: $[0,S_0)$ for the initial cycle and
$[S_{i-1},S_i)$ for cycle $i\ge1$. The value at a finite right endpoint is excluded from that cycle. A zero-length cycle has an empty path.
Denote these random cycles, with their paths measured from their
starting times, by $C_0,C_1,\ldots$. Let $Q_0$ be the conditional
distribution of $C_0$ given $Y_0<\infty$, and let $Q$ be the common
distribution of a subsequent cycle conditional on its start and its
length being finite.

The joint regeneration property implies that, conditional on
$S_n<\infty$, the first $n+1$ cycles have distribution
$Q_0\otimes Q^{\otimes n}$. By the Ionescu--Tulcea theorem, these finite-cycle distributions extend
uniquely to a probability measure on the infinite sequence of cycles.
Concatenating the cycles gives the measure $P^*$ on $(\Omega,{\cal F})$.
The subsequent durations are i.i.d. with distribution $F(t)/r$.
Since this distribution assigns positive
probability to $(0,\infty)$, $S_n\to\infty$ a.s., so the path is
defined at every finite time.

For every event $A$ determined by the first $n+1$ cycles, including
their durations, the construction gives
$$
P^*(A)=P(A\mid S_n<\infty).
$$
The event displayed in the proposition is such an event.
Conversely, regeneration and the specified conditional cycle
distributions determine the law of every finite sequence of cycles.
These cycles determine the path and all regeneration times, and hence
the law on ${\cal F}$. This proves uniqueness.\ $\Box$

\begin{def1}\label{p1_r5}
The probability measure $P^*$ is called the \emph{finite projection}
of $P$ with respect to the renewal process $S$.
\end{def1}

Random variables and stochastic processes on $(\Omega,{\cal F})$ are
called \emph{initial} when considered under $P$, and \emph{projected}
when considered under $P^*$. We use unstarred notation for the initial
variables and their characteristics (e.g., $X,S_n,X(t),F(t)$), and a
superscript asterisk for their projected counterparts (e.g.,
$X^*,S_n^*,X^*(t),F^*(t)$). The random variables themselves are the
same measurable functions; their probability laws differ.

Conditioning the cycle lengths on finiteness gives, for finite $t\ge0$,
$$
F_0^*(t)=F_0(t)/r_0,\qquad F^*(t)=F(t)/r,\qquad
p_nP^*\{S_n\le t\}=P\{S_n\le t\},\quad n\ge0.
$$

For each fixed finite $t\in T$, let
${\cal F}_t=\sigma(X(u):u\in T,0\le u\le t)$ and put
$P_n\{\cdot\}=P\{\cdot\mid S_n<\infty\}$, $n\ge0$. Then
\begin{equation}\label{lim}
P^*\{A\}=\lim_{n\to\infty}P\{A\mid S_n<\infty\},
\qquad A\in{\cal F}_t.
\end{equation}

To prove this, first note that the conditional identity in the proof of
Proposition~\ref{p1_r4} gives, for $A\in{\cal F}_t$,
$$
P_n\{A\cap\{S_n>t\}\}=P^*\{A\cap\{S_n>t\}\}.
$$
On $\{S_n>t\}$ the path through time $t$ is wholly contained in the
first $n+1$ cycles. The equality therefore holds first for
cylinder events and then for every $A\in{\cal F}_t$ by the monotone
class theorem. Splitting $A$ according to $S_n>t$ or $S_n\le t$ yields
\begin{equation}\label{p1_r14}
\varepsilon_n=|P_n\{A\}-P^*\{A\}|
=|P_n\{A\cap\{S_n\le t\}\}
 -P^*\{A\cap\{S_n\le t\}\}|.
\end{equation}
Since $P_n\{S_n\le t\}=P^*\{S_n\le t\}$, both terms inside the last
absolute value lie between zero and $P^*\{S_n\le t\}$. Hence, since $S_n\to\infty$, $P^*$-almost surely, as established in
Proposition~\ref{p1_r4},
$$
\varepsilon_n\le P^*\{S_n\le t\}\longrightarrow0,
$$
which proves~(\ref{lim}).

\subsection{Finite projections as approximations of the original probabilities}
\label{par_1_3}

Finite projections make it possible to compare transient stochastic
models with recurrent ones. In this section, we study the accuracy of
this approximation by comparing the original and projected probabilities
and the basic characteristics of their embedded renewal processes.

\subsubsection{\sloppy The distance between probabilities $P$ and $P^*$ on the $\sigma$-algebra ${\cal F}_t$ \break I. Preliminary Results}

To begin, we will estimate the difference between the finite projection $P^*$ and the initial measure $P$ on the $\sigma$-algebra ${\cal F}_t,\ \ t>0$, using equality (\ref{p1_r14}).

\begin{prop} \label{p1_r8}
Let $P^*$ be the finite projection of the probability $P$ with respect to the renewal process $\{S_n\}_{n\ge 0}$. Then, for every $t>0$,
$$
\sup_{A\in {\cal F}_t} \vert P\{A\} -P^*\{A\} \vert
\le \inf_{n\ge 0}\left ( 1-p_n + P^*\{S_n \le t\} \right ). 
$$
\end{prop}

{\bf Proof.} Let $A\in {\cal F}_t$ and $n\ge 0$. Clearly,
\begin{equation}
\vert P\{A\} -P^*\{A\} \vert \le \vert P\{A\} -P_n\{A\} \vert +
\vert P_n\{A\} -P^*\{A\} \vert \label{p1_r7}
\end{equation}
By (\ref{p1_r14}),
$$
\vert P_n\{A\}-P^*\{A\} \vert \le \max \left ( P_n\{A\cap
\{S_n\le t\}\}; P^*\{A\cap \{S_n\le t\}\} \right ) \le P^*\{S_n\le t\},
$$
Also,
$$
\vert P\{A\}-P_n\{A\}\vert
=\left \vert P\{A \cap \{S_n<\infty\}\} 
+ P\{A \cap \{S_n=\infty\}\} - \frac{ P\{A \cap
\{S_n<\infty\}\}} {p_n} \right \vert =
$$
$$
=\vert P\{A \cap \{S_n<\infty\}\}(1-p_n^{-1})+
P\{A \cap \{S_n=\infty\}\} \vert \le
$$
$$
\le \max \left (P\{A \cap \{S_n<\infty\}\}(p_n^{-1}-1); P\{A
\cap \{S_n=\infty\}\} \right ) \le 1-p_n
$$
Combining these estimates with (\ref{p1_r7}), taking the supremum over $A\in {\cal F}_t$, and then taking the infimum over integers $n\ge 0$ proves the assertion. $\Box$ \bigskip

The approximation error of the probability $P$ by its finite projection $P^*$ on the $\sigma$-algebra 
${\cal F}_t$ depends on the values of $r_0,\ r,\ t$, as well as on the properties of the renewal process $\{S_n\}$.

{\bf Remark.} This is a preliminary bound, but its proof suggests a
method for comparing the original and projected probabilities through
an identity such as~(\ref{p1_r14}). For fixed projected renewal laws,
central limit estimates applied to Proposition~\ref{p1_r8} give
convergence-rate bounds as the finite-cycle probabilities $r_0,r$
approach one.

We now present some consequences of Proposition \ref{p1_r8} for ordinary (non-delayed) renewal processes, i.e., renewal processes for which $P\{S_0=0\}=1$. Here $p_n=r^n$. We use results concerning convergence rates in the Central Limit Theorem from \cite{part1_8} and \cite{part1_9}.

\begin{prop} \label{c_r1}
Let $E Y^*_1=\mu$,
$\sigma^2=E (Y^*_1-\mu)^2>0$, and
$\alpha_j=E\vert Y^*_1-\mu\vert^j$ for $j=3,k$, with
$\alpha_k<\infty$ for some $k\ge 3$. Then, for every $t>0$,
$$
\sup_{A\in {\cal F}_t} \vert P\{A\} -P^*\{A\} \vert
\le
$$
$$
\le \inf_{n\ge 1}\left ( 1-p_n + 
\Phi \left (x(n)\right )+ C(k)
\left (1+\left \vert x(n)\right\vert \right )^{-k} 
\left (\frac{\alpha_3}{\sigma^3 \sqrt{n}}+ 
\frac{\alpha_k}{\sigma^k n^{(k-2)/2}}\right ) \right ),
$$
where $x(n)=(t-\mu n)/(\sigma\sqrt{n})$, $C(k)$ is a positive constant
depending only on $k$, and $\Phi$ is the standard normal distribution function.
\end{prop}
{\bf Proof.} Under the conditions of the proposition, the following inequality holds:
$$
\left|P^*\{S_n\le t\}-\Phi(x(n))\right|
\le C(k)(1+|x(n)|)^{-k}
\left(\frac{\alpha_3}{\sigma^3\sqrt n}
+\frac{\alpha_k}{\sigma^k n^{(k-2)/2}}\right),
$$
for every integer $n\ge 1$ and $t>0$, where $C(k)$ is a positive
constant depending only on $k$; see~\cite{part1_8}. Applying
Proposition~\ref{p1_r8} with $r_0=1$ and taking the infimum over
integers $n\ge 1$ proves the assertion.~$\Box$

For $0<r<1$, whenever the expression under the infimum is at most $1$
for some integer $n\ge 1$, the infimum is attained and is therefore a minimum.

\noindent{\bf Corollary.} Under the conditions of
Proposition~\ref{c_r1} with $k=3$,
$$
\sup_{A\in{\cal F}_t}|P\{A\}-P^*\{A\}|
\le\inf_{n\ge1}\left(1-p_n+\Phi(x(n))+
\frac{A\alpha_3}{\sigma^3\sqrt n}(1+|x(n)|)^{-3}\right),
$$
where $A$ is an absolute positive constant. This follows by taking
$A=2C(3)$ in the proposition.

The normal-tail asymptotic
$$
1-\Phi(x)\sim\frac{e^{-x^2/2}}{x\sqrt{2\pi}},\qquad x\to\infty,
$$
also gives, after increasing $C(k)$,
$$
\sup_{A\in{\cal F}_t}|P\{A\}-P^*\{A\}|
\le\inf_{n\ge m(t)}\left(1-p_n+
C(k)(1+|x(n)|)^{-k}
\left(\frac{\alpha_3}{\sigma^3\sqrt n}+
\frac{\alpha_k}{\sigma^k n^{(k-2)/2}}\right)\right).
$$
Here $m(t)$ is a sufficiently large positive integer, which may depend
on the fixed renewal law and on $k$. Indeed, $x(n)\sim-\mu\sqrt n/\sigma$,
so the normal tail decays exponentially in $n$ and is eventually bounded
by the displayed polynomial remainder. A continuous relaxation may
help select candidate indices, whose integer values must then be checked.

We next use the Chernoff bound to obtain a preliminary estimate when
the projected interval law has an exponential moment.
\begin{prop}\label{cramer_bound}
Assume $P\{S_0=0\}=1$ and that $Y_1^*$ satisfies Cram\'er's condition
\begin{equation}\label{cram}
\varphi(\lambda)={\bf E}^*e^{\lambda Y_1^*}<\infty
\quad\hbox{for some }\lambda>0.
\end{equation}
Write $\mu={\bf E}^*Y_1^*>0$ and
$$
h(a)=\sup_{\lambda:\,\varphi(\lambda)<\infty}
       \{a\lambda-\ln\varphi(\lambda)\}.
$$
For fixed $t>0$, put $n_0=\lfloor t/\mu\rfloor+1$ and choose a finite
constant $C=C(t)$ with $0<C\le h(t/n_0)$. Then
$$
\sup_{A\in{\cal F}_t}|P\{A\}-P^*\{A\}|
\le\inf_{n\ge n_0}\{1-p_n+e^{-Cn}\}.
$$
If $q=\ln(1/r)$ satisfies $0<q<C$, the infimum is attained at one of
the integers
$$
\max\{n_0,\lfloor x_0\rfloor\},\qquad
\max\{n_0,\lceil x_0\rceil\},\qquad
x_0=\frac{\ln(C/q)}{C-q}.
$$
In particular, if also $q<1$ and
$n_q=\lceil2C^{-1}\ln(1/q)\rceil\ge n_0$, then
$$
\sup_{A\in{\cal F}_t}|P\{A\}-P^*\{A\}|
\le 1-\exp\!\left[-q\left(\frac{2}{C}\ln\frac1q+1\right)\right]+q^2.
$$
For a fixed projected law and fixed $t$, this is
$O((1-r)|\ln(1-r)|)$ as $r\uparrow1$.
\end{prop}

{\bf Proof.} The Chernoff inequality (see~\cite{part1_6},
\cite{part1_14}) gives, for $a<\mu$,
$$
P^*\{S_n\le na\}\le e^{-nh(a)}.
$$
It follows directly by applying Markov's inequality to
$e^{\lambda S_n}$ for $\lambda<0$ and optimizing. For $a<\mu$,
positive $\lambda$ cannot improve the exponent, since
$\ln\varphi(\lambda)\ge\lambda\mu$. The rate function is positive
on $[0,\mu)$ and nonincreasing there, with its minimum zero at $\mu$.
Thus, for $n\ge n_0$,
$$
h(t/n)\ge h(t/n_0)\ge C,\qquad
P^*\{S_n\le t\}\le e^{-Cn}.
$$
If $h(t/n_0)=\infty$, any finite positive $C$ may be used.
Proposition~\ref{p1_r8} now gives the first bound.

For $0<q<C$, consider $f(x)=1-e^{-qx}+e^{-Cx}$, $x\ge0$.
Its derivative is $f'(x)=qe^{-qx}-Ce^{-Cx}$, which is negative
before $x_0=\ln(C/q)/(C-q)$ and positive after $x_0$.
The minimum over the allowed integer indices is therefore attained
at one of these two integers.

For the last bound, use the admissible integer $n_q$. Since
$n_q<2C^{-1}\ln(1/q)+1$ and $e^{-Cn_q}\le q^2$, we obtain
$$
1-p_{n_q}+e^{-Cn_q}
\le1-\exp\!\left[-q\left(\frac2C\ln\frac1q+1\right)\right]+q^2.
$$
Finally $q\sim1-r$, and $1-e^{-u}\le u$ for $u\ge0$, proving
the stated order.\ $\Box$

For fixed $t$ and a fixed projected inter-renewal law, Proposition~\ref{c_r1}
gives, as $r\uparrow1$,
$$
\sup_{A\in{\cal F}_t}|P\{A\}-P^*\{A\}|
=O\big((1-r)^{(k+1)/(k+3)}\big).
$$
Indeed, $x(n)\sim-\mu\sqrt n/\sigma$, so the normal term decays
exponentially and the remainder is $O(n^{-(k+1)/2})$.
Since $1-p_n=1-r^n\le n(1-r)$, taking
$n=\lceil(1-r)^{-2/(k+3)}\rceil$ gives the stated bound.
Proposition~\ref{cramer_bound}, under its exponential-moment assumption,
gives the smaller upper bound $O((1-r)|\ln(1-r)|)$.
Both bounds tend to zero; here $r_0=1$, since the renewal process is ordinary.

The following subsections give stronger approximation estimates under
weaker assumptions than these preliminary bounds.

  \subsubsection{The difference between the characteristics of the embedded renewal processes over the time interval $[0,t]$}

In this section, we are interested in the difference between the renewal process $\{S_n\}$ embedded in the initial regenerative process $X$ and the renewal process embedded in the projected process $X^*$.

Let $G(t)$ be a distribution function concentrated on $[0,\infty]$,
with $G\{0\}<1$. We denote by $G^{(\star n)}(t)$ the $n$-fold
convolution of the function $G(t)$ with itself. The primary object
of study in renewal theory is the renewal equation:
\begin{equation}
Z(t)=z(t)+\int_0^tZ(t-y)G\{dy\},\qquad t\ge 0, 
\label{tir1}
\end{equation}
and the asymptotic behavior of its solution $Z$. The interval of integration is assumed to be closed, and the functions $Z(t)$ and $z(t)$ are zero on the negative half-axis.

The probabilistic interpretation and general theory of renewal equations (\ref{tir1}) can be found, for example, in \cite{part1_1}, \cite{c2}. Recall that for $z(t)=1,\ t\ge 0$, the solution of the equation (\ref{tir1}) is the renewal function:
$$
U(t)=\sum_{n=0}^\infty G^{(\star n)}(t),
$$
whose value at point $t$ equals the mean number of renewals that occur by time $t$ in an ordinary renewal process, where zero is considered a renewal moment (this result can be found in \cite{c2}).

\begin{lema} \label{eq_ren} {\normalsize \cite{c2}}. \ $U(t)<\infty$ for every $t\ge 0$, and
$$
U(t)=1+\int_0^t U(t-y)\,G\{dy\},\qquad t\ge 0.
$$
If $z$ is measurable and bounded on every finite interval, then
the function $Z$ defined by
$$
Z(t)=\int_0^t z(t-y)\, U\{dy\}, \qquad t\ge 0,
$$
is the unique measurable solution of equation (\ref{tir1}) that is
bounded on every finite interval. In particular,
$Z(0)=z(0)/(1-G\{0\})$.
\end{lema}

Using the symbol $\star$ for the convolution operation, we rewrite (\ref{tir1}) in a shortened form:
$$
Z=z+G\star Z,
$$
and for the equalities in the lemma, we have respectively:
$$
U=1+G\star U, \qquad Z=U\star z.
$$

Let $H(t)$ and $H^*(t)$ denote the renewal functions corresponding
to $F$ and $F^*$, respectively:
$$
H(t)=\sum_{n=0}^\infty F^{(\star n)}(t),\qquad
H^*(t)=\sum_{n=0}^\infty F^{*\, (\star n)}(t).
$$
Let $Z$ and $Z^*$ be the solutions of (\ref{tir1}) with the same
locally bounded measurable function $z$ and with $G=F$ and $G=F^*$,
respectively.

\begin{prop}\label{tir2}
$$
Z^*(t)-Z(t)=(1-r)(H^*\star F^*\star H\star z)(t).
$$
\end{prop}
{\bf Proof.} By definition,
$$
Z^*=z+F^*\star Z^*,\qquad Z=z+F\star Z.
$$
Subtracting and using $F=rF^*$ gives
$$
Z^*-Z=(1-r)F^*\star Z+F^*\star(Z^*-Z).
$$
Thus $Z^*-Z$ satisfies the renewal equation with forcing function
$(1-r)F^*\star Z$, which is bounded on every finite interval.
Lemma~\ref{eq_ren} gives
$Z^*-Z=(1-r)H^*\star F^*\star Z$.
Since $Z=H\star z$, the asserted identity follows. $\Box$

From Proposition~\ref{tir2} we obtain the following properties:
\begin{description}
\item[(a)] $Z^*(t)\geq Z(t)$ if $z\geq0$.
\item[(b)] For a general $z$,
\begin{eqnarray*}
|Z^*(t)-Z(t)|
&=&(1-r)\big|(H^*\star F^*\star H\star z)(t)\big|\\
&\leq&(1-r)(H^*\star F^*\star H^*\star |z|)(t),
\end{eqnarray*}
since $H\{dy\}\leq H^*\{dy\}$.
\item[(c)] Taking $z\equiv1$ gives
$$
H^*(t)-H(t)=(1-r)(H^*\star F^*\star H)(t)
\leq(1-r)(H^*\star F^*\star H^*)(t).
$$
\end{description}

For fixed $F^*$ and $z$, the factors multiplying $1-r$ in these
upper bounds depend only on $t$ and not on $r$.

Let $N(t)=\inf\{n\geq0:S_n>t\}$ be the renewal counting process.
It counts the renewal epochs in $[0,t]$; in an ordinary process
($S_0=0$ a.s.) the renewal at zero is included.
For an ordinary renewal process,
$$
H(t)={\bf E}N(t),\qquad H^*(t)={\bf E}N^*(t),\qquad t\geq0.
$$

For a delayed renewal process, the mean functions
$H_0(t)={\bf E}N(t)$ and $H_0^*(t)={\bf E}N^*(t)$ satisfy
\cite{c2}
$$
H_0=F_0+F\star H_0,\qquad H_0^*=F_0^*+F^*\star H_0^*.
$$
As in the proof of Proposition~\ref{tir2}, subtracting and using
$F=rF^*$ and $F_0=r_0F_0^*$ gives
$$
H_0^*-H_0=(1-r_0)F_0^*+(1-r)F^*\star H_0
+F^*\star(H_0^*-H_0).
$$
Applying Lemma~\ref{eq_ren}, we obtain
\begin{equation}\label{dif_1}
H_0^*(t)-H_0(t)=(1-r_0)(H^*\star F_0^*)(t)
+(1-r)(H^*\star F^*\star H_0)(t).
\end{equation}
Consequently,
$$
0\leq H_0^*(t)-H_0(t)\leq\max(1-r_0,1-r)C(t),
$$
where, for fixed $F^*$ and $F_0^*$, one may take
$$
C(t)=H_0^*(t)+(H^*\star F^*\star H_0^*)(t).
$$
Here we used $H^*\star F_0^*=H_0^*$ and $H_0\leq H_0^*$.

We now examine the difference between the variances. As in
\cite{part1_3}, put $\psi(t)={\bf E}[N(t)(N(t)+1)]$. Then
$$
\mathop{\rm Var}N(t)=\psi(t)-H_0(t)-H_0(t)^2.
$$
Using Laplace--Stieltjes transforms as in~\cite{part1_3} gives
$\psi(t)=2(F_0\star H\star H)(t)$. Hence
$$
\mathop{\rm Var}N(t)=2(F_0\star H\star H)(t)-H_0(t)-H_0(t)^2,
$$
$$
\mathop{\rm Var}N^*(t)=2(F_0^*\star H^*\star H^*)(t)
-H_0^*(t)-H_0^*(t)^2.
$$
Subtracting and using $F_0=r_0F_0^*$, we have
$$
\mathop{\rm Var}N^*(t)-\mathop{\rm Var}N(t)
=2\big(F_0^*\star[H^*\star H^*-r_0H\star H]\big)(t)
$$
$$
\hfill{}-[H_0^*(t)-H_0(t)][1+H_0^*(t)+H_0(t)].
$$
Since $H^*\star H^*-H\star H=(H^*-H)\star(H^*+H)$,
this yields
\begin{eqnarray}\label{dif_2}
\mathop{\rm Var}N^*(t)-\mathop{\rm Var}N(t)
&=&2\big(F_0^*\star(H^*-H)\star(H^*+H)\big)(t)\nonumber\\
&&+2(1-r_0)(F_0^*\star H\star H)(t)\nonumber\\
&&-[H_0^*(t)-H_0(t)][1+H_0^*(t)+H_0(t)].
\end{eqnarray}

We use the following averaging lemma. Its discrete counterpart is
related to regular N\"orlund means; see~\cite{hardy} for that theory.
\begin{lema}\label{norl}
Let $V$ be a nonzero locally finite positive measure on $[0,\infty)$,
$V([0,t])>0$ for large $t$, and assume, for every fixed $a>0$,
\begin{equation}\label{norl1}
\frac{V([0,t-a])}{V([0,t])}\longrightarrow1\qquad(t\to\infty).
\end{equation}
If $f$ is measurable, bounded on finite intervals, and tends to $L$
(finite or infinite), then
$(V*f)(t)/V([0,t])\longrightarrow L$.
\end{lema}
{\bf Proof.} For finite $L$, choose $A$ so that $|f-L|\leq\epsilon$
on $[A,\infty)$ and bound $|f-L|$ on $[0,A]$ by $C_A$. Split the
integral at $t-A$: the normalized error is at most
$\epsilon V([0,t-A])/V([0,t])+C_A(1-V([0,t-A])/V([0,t]))$.
Let $t\to\infty$ and $\epsilon\downarrow0$. For $L=+\infty$ use an
arbitrary lower bound for $f$ on $[A,\infty)$ and its finite lower
bound on $[0,A]$; apply this to $-f$ for $L=-\infty$. $\Box$

The condition (\ref{norl1}) is necessary here: the cumulative measure
$V([0,t])=e^t$ is increasing but has
$(V([0,t+1])-V([0,t]))/V([0,t])=e-1$. 

Using the preceding lemma, properties (b) and (c) after
Proposition~\ref{tir2}, identities (\ref{dif_1})--(\ref{dif_2}),
and the renewal theorems (see~\cite{c2}), we obtain the following result.

\begin{teor}\label{t_car_dif}
Assume $0<\mu={\bf E}Y_1^*<\infty$. When $r$ and $r_0$ vary, the
conditional distributions $F^*$ and $F_0^*$, and the function $z$,
are kept fixed.
\begin{description}
\item[a)]
$$
0\leq H^*(t)-H(t)\leq(1-r)(F^*\star H^*\star H^*)(t),
$$
where $(F^*\star H^*\star H^*)(t)\sim t^2/(2\mu^2)$.
\item[b)]
\begin{eqnarray*}
0\leq H_0^*(t)-H_0(t)
&\leq&(1-r_0)H_0^*(t)\\
&&{}+(1-r)(F_0^*\star F^*\star H^*\star H^*)(t),
\end{eqnarray*}
where
$$
H_0^*(t)\sim\frac{t}{\mu},\qquad
(F_0^*\star F^*\star H^*\star H^*)(t)
\sim\frac{t^2}{2\mu^2}.
$$
\item[c)] If $F^*$ is non-arithmetic and $z\geq0$ is directly
Riemann integrable, then
$$
0\leq Z^*(t)-Z(t)\leq(1-r)(F^*\star H^*\star Z^*)(t),
$$
where
\begin{equation}\label{renew_forcing_bound}
(F^*\star H^*\star Z^*)(t)
=\frac{t}{\mu^2}\int_0^\infty z(y)\,dy+o(t).
\end{equation}
\item[d)] If $(1-r_0)/(1-r)\to c\in[0,\infty)$ as $r\uparrow1$,
then, for every fixed $t$, the limit
$$
L(t)=\lim_{r\uparrow1}
\frac{\mathop{\rm Var}N^*(t)-\mathop{\rm Var}N(t)}{1-r}
$$
exists, and $L(t)\sim-t^3/(3\mu^3)$ as $t\to\infty$.
\end{description}
The inequalities hold for every $t\geq0$; all asymptotic relations
concern $t\to\infty$.
\end{teor}
{\bf Proof.} {\bf a)} Property (c) after Proposition~\ref{tir2}
gives the inequality. By the elementary renewal theorem,
$H^*(t)\sim t/\mu$. Writing $H^*(u)=u/\mu+e(u)$ gives
$\sup_{0\leq u\leq t}|e(u)|=o(t)$; Stieltjes integration by parts,
first for two factors and then for three, yields
\begin{equation}\label{renew_power}
(H^*\star H^*)(t)\sim\frac{t^2}{2\mu^2},\qquad
(H^*\star H^*\star H^*)(t)\sim\frac{t^3}{6\mu^3}.
\end{equation}
Convolution with a fixed proper distribution preserves these
asymptotics: restrict its integration to $[0,A]$ for a lower bound,
use its total mass one for an upper bound, and then let $A\to\infty$.
This proves a).

{\bf b)} Identity (\ref{dif_1}), $H\{dy\}\leq H^*\{dy\}$ and
$H_0^*=F_0^*\star H^*$ give the inequality. The asserted asymptotics
follow from a) and convolution with the proper delay law $F_0^*$.

{\bf c)} Proposition~\ref{tir2} and $H\{dy\}\leq H^*\{dy\}$
give the inequality. By the key renewal theorem
(see~\cite{part1_1}, \cite{c2}),
$Z^*(t)\to\mu^{-1}\int_0^\infty z(y)\,dy$.
The cumulative function $(F^*\star H^*)(t)\sim t/\mu$
satisfies (\ref{norl1}); Lemma~\ref{norl} therefore gives
(\ref{renew_forcing_bound}).

{\bf d)} Keep the conditional laws fixed and let $r\uparrow1$.
The renewal series are dominated on every finite interval by the
corresponding projected renewal series. Using property (c) after
Proposition~\ref{tir2} and identity~(\ref{dif_1}), we obtain,
for each fixed $t$,
$$
\frac{H^*(t)-H(t)}{1-r}
\longrightarrow(H^*\star F^*\star H^*)(t),
$$
$$
\frac{H_0^*(t)-H_0(t)}{1-r}
\longrightarrow cH_0^*(t)
+(F_0^*\star F^*\star H^*\star H^*)(t).
$$
Substitution in~(\ref{dif_2}) gives
\begin{eqnarray}\label{variance_derivative}
L(t)&=&4(F_0^*\star F^*\star H^*\star H^*\star H^*)(t)\nonumber\\
&&{}+2c(F_0^*\star H^*\star H^*)(t)\nonumber\\
&&{}-\big[cH_0^*(t)+(F_0^*\star F^*\star H^*\star H^*)(t)\big]
[1+2H_0^*(t)].
\end{eqnarray}
By~(\ref{renew_power}), convolution with proper distributions,
and $H_0^*(t)\sim t/\mu$, the terms containing $c$ are $O(t^2)$,
while the cubic terms are
$$
\frac{2t^3}{3\mu^3}-\frac{t^3}{\mu^3}
=-\frac{t^3}{3\mu^3}.
$$
Thus $L(t)\sim-t^3/(3\mu^3)$, proving d). $\Box$

The ordering of the mean renewal counts does not imply the same
ordering of their variances. Part d) concerns successive limits:
first $r\uparrow1$ for fixed $t$, and then $t\to\infty$.

\noindent{\bf Example.} Consider the single-server instant-service
model in which arrival epochs form the projected renewal process,
with initial delay distribution $F_0^*$ and interval distribution $F^*$.
Each arrival independently causes permanent failure with probability
$1-r_0$ at the first arrival and $1-r$ thereafter. An arrival is
served instantly if it does not cause failure; the failing arrival
and all later arrivals are unserved. The served count has the
defective renewal distribution, so the mean unserved count by time
$t$ is $H_0^*(t)-H_0(t)$. Thus, part b) of
Theorem~\ref{t_car_dif} bounds the expected number of unserved
arrivals by time $t$.
For $r<1$, the first failure time $T_f$ has distribution
\begin{equation}\label{failure_delay}
P\{T_f\leq t\}=(1-r_0)F_0^*(t)+(1-r)(F^*\star H_0)(t).
\end{equation}
This is a proper distribution: failure occurs at the first arrival
or at an arrival immediately following a successfully served one.
The first unserved arrival occurs at $T_f$; subsequent interarrival
times retain distribution $F^*$. Hence the unserved arrival epochs
form a delayed renewal process with initial delay distribution
(\ref{failure_delay}).

\subsubsection{Distance between the probabilities $P$ and $P^*$ on the $\sigma$-algebra ${\cal F}_t$ \break  II. \ Exact Estimates}

The bounds for $H_0^*(t)-H_0(t)$ and $Z^*(t)-Z(t)$ suggest that,
on a fixed observation interval, the difference between the original
and projected probabilities is $O(\max(1-r,1-r_0))$ as $r,r_0\to1$.
A similar suggestion comes from convergence-rate results for the time
to a rare event in a regenerative process (R\'enyi's theorem), as the
cycle probability $1-r$ of that event tends to zero; see~\cite{kalas_2}.
Renewal theorems and an argument similar to the proof of
Proposition~\ref{p1_r8} give the following result:

\begin{teor}\label{t_prob_dif}
Let $P^*$ be the finite projection of $P$ with respect to the renewal
process $\{S_n\}$, with $r_0$ and $r$ as defined above. Then, for
every $t\geq0$,
\begin{equation}\label{est_baz}\label{tv_sharp_bound}
\sup_{A\in{\cal F}_t}|P^*\{A\}-P\{A\}|
\leq 1-r_0+(1-r)H_0(t)
\leq 1-r_0+(1-r)H_0^*(t).
\end{equation}
The following consequences use the upper bound with $H_0^*$.
\begin{description}
\item[a)] If $0<\mu={\bf E}Y_1^*<\infty$, then
$$
\sup_{A\in{\cal F}_t}|P^*\{A\}-P\{A\}|
\leq 1-r_0+(1-r)(1+\varepsilon(t))\frac{t}{\mu}.
$$
\item[b)] If, in addition, $F^*$ is non-arithmetic and
$\mu_2={\bf E}(Y_1^*)^2<\infty$, then
$$
\sup_{A\in{\cal F}_t}|P^*\{A\}-P\{A\}|
\leq 1-r_0+(1-r)\left[\frac{t}{\mu}
-\frac1\mu\int_0^t(1-F_0^*(y))\,dy
+\frac{\mu_2}{2\mu^2}+\varepsilon(t)\right].
$$
\item[c)] If $1-F^*(t)=t^{-\alpha}L(t)$, where $0<\alpha<1$
and $L$ is slowly varying, then
$$
\sup_{A\in{\cal F}_t}|P^*\{A\}-P\{A\}|
\leq 1-r_0+(1-r)\left[\frac{\sin(\pi\alpha)}{\pi\alpha}
+\varepsilon(t)\right]\frac{t^\alpha}{L(t)}.
$$
\end{description}
The bounds in a)--c) hold for all sufficiently large $t$. In each
part, $\varepsilon(t)$ denotes a possibly different nonnegative,
nonincreasing function tending to zero as $t\to\infty$. These
functions depend on $F^*,F_0^*$, but not on $r,r_0$.
\end{teor}
For arithmetic $F^*$, the corresponding version of part b) uses the
lattice renewal expansion; see~\cite{c3}, relation (12.2).

{\bf Proof.} Let $A\in{\cal F}_t$. On $\{N(t)=n,S_n<\infty\}$,
the path through time $t$ is determined by the first $n+1$ finite
cycles. The definition of finite projection therefore gives
$$
P\{A\cap\{N(t)=n,S_n<\infty\}\}
=p_nP^*\{A\cap\{N(t)=n\}\}.
$$
Decomposing according to $N(t)$ and the finiteness of $S_n$, we obtain
\begin{eqnarray*}
P^*\{A\}-P\{A\}
&=&\sum_{n=0}^\infty(1-p_n)P^*\{A\cap\{N(t)=n\}\}\\
&&{}-\sum_{n=0}^\infty P\{A\cap\{N(t)=n,S_n=\infty\}\}.
\end{eqnarray*}
For $A=\Omega$, the two sums are equal, since $P^*(\Omega)=P(\Omega)=1$.
For any $A$, each sum lies between zero and this common value.
Hence the absolute value of their difference is at most that value, namely
$$
\sum_{n=0}^\infty P\{N(t)=n,S_n=\infty\}
=1-r_0+(1-r)\sum_{n=1}^\infty P\{S_{n-1}\leq t\}
=1-r_0+(1-r)H_0(t).
$$
Consequently
$$
|P^*\{A\}-P\{A\}|\leq1-r_0+(1-r)H_0(t)
\leq1-r_0+(1-r)H_0^*(t),
$$
which proves (\ref{est_baz}).

{\bf a)} The elementary renewal theorem gives
$H_0^*(t)\sim t/\mu$; see~\cite{c3}, \cite{c2}. Thus we may take
$$
\varepsilon(t)=\sup_{y\geq t}\left|\frac{\mu H_0^*(y)}{y}-1\right|,
$$
which is finite for all sufficiently large $t$ and decreases to zero.
Substitution in (\ref{est_baz}) proves a).

{\bf b)} The non-arithmetic renewal expansion gives
$H^*(t)=t/\mu+\mu_2/(2\mu^2)+e(t)$, where $e$ is bounded and
tends to zero; see~\cite{c2}. Convolving with $F_0^*$, dominated
convergence and integration by parts yield
\begin{equation}\label{delayed_second_order}
H_0^*(t)=\frac1\mu\int_0^t F_0^*(y)\,dy
+\frac{\mu_2}{2\mu^2}+o(1).
\end{equation}
No moment assumption on the delay is needed. Take
$$
\varepsilon(t)=\sup_{y\geq t}\left|H_0^*(y)
-\frac1\mu\int_0^y F_0^*(u)\,du-\frac{\mu_2}{2\mu^2}\right|.
$$
Using $\int_0^t F_0^*(y)\,dy=t-\int_0^t(1-F_0^*(y))\,dy$
in (\ref{est_baz}) proves b).

{\bf c)} The infinite-mean renewal theorem~\cite{c2} gives
$(1-F^*(t))H^*(t)\to\sin(\pi\alpha)/(\pi\alpha)$.
Convolution with the proper delay law preserves this asymptotic,
by a fixed-shift comparison. Hence
\begin{equation}\label{heavy_tail_renewal}
(1-F^*(t))H_0^*(t)\longrightarrow
\frac{\sin(\pi\alpha)}{\pi\alpha}.
\end{equation}
Taking
$$
\varepsilon(t)=\sup_{y\geq t}\left|(1-F^*(y))H_0^*(y)
-\frac{\sin(\pi\alpha)}{\pi\alpha}\right|
$$
and applying (\ref{est_baz}) proves c). $\Box$

The sharper bound in (\ref{est_baz}) involves the defective renewal
function $H_0$. Its behaviour near $r=1$ is described by the following
form of R\'enyi's exponential limit.

\begin{teor}\label{defective_renewal_theorem}
For fixed $0<r<1$,
$$
H_0(t)\longrightarrow\frac{r_0}{1-r},\qquad t\to\infty.
$$
Consequently, $1-r_0+(1-r)H_0(t)$ increases to one and never
exceeds one.

Keep the proper laws $F_0^*,F^*$ fixed, with
$0<\mu={\bf E}Y_1^*<\infty$ and $r_0>0$. Then
\begin{equation}\label{defective_renewal_exp}
\varepsilon(r):=\sup_{t\ge0}\left|
\frac{(1-r)H_0(t)}{r_0}
-\left(1-e^{-(1-r)t/\mu}\right)\right|
\longrightarrow0,\qquad r\uparrow1.
\end{equation}
Thus, uniformly for $t\ge0$,
\begin{equation}\label{tv_exp_approx}
\sup_{A\in{\cal F}_t}|P^*\{A\}-P\{A\}|
\le1-r_0e^{-(1-r)t/\mu}+r_0\varepsilon(r).
\end{equation}
In particular, for every fixed $x>0$,
$$
H_0\!\left(\frac{\mu x}{1-r}\right)
\sim\frac{r_0}{1-r}(1-e^{-x}),\qquad r\uparrow1.
$$
No moment condition on $F_0^*$ or non-arithmetic assumption is required.
\end{teor}

{\bf Proof.} The representation
$$
H_0(t)=\sum_{n=0}^\infty p_n
(F_0^*\star F^{*\,(\star n)})(t)
$$
gives the fixed-$r$ limit by dominated convergence. Let $K$ be
independent of the projected renewal process, with
$P^*\{K=n\}=(1-r)r^n$, $n\ge0$. Then
$$
\frac{(1-r)H_0(t)}{r_0}
=P^*\left\{Y_0^*+\sum_{j=1}^K Y_j^*\le t\right\}.
$$
As $r\uparrow1$, $(1-r)K$ converges in distribution to
$\mathrm{Exp}(1)$ and $K\to\infty$ in probability. The law of large
numbers and $(1-r)Y_0^*\to0$ in probability therefore give
$$
\frac{1-r}{\mu}\left(Y_0^*+\sum_{j=1}^K Y_j^*\right)
\ \mathop{\longrightarrow}^{\rm d}\ \mathrm{Exp}(1).
$$
Continuity of the limiting distribution implies uniform convergence
of the distribution functions, proving (\ref{defective_renewal_exp}).
The remaining assertions follow from (\ref{est_baz}) and
(\ref{defective_renewal_exp}). $\Box$

The approximation has uniform absolute error; no rate or
relative-error bound is asserted. For an ordinary process with
exponential projected intervals of mean $\mu$, $r_0=1$ and
$$
H_0^*(t)=1+\frac{t}{\mu},\qquad
H_0(t)=\frac{1-r e^{-(1-r)t/\mu}}{1-r}.
$$
The sharper bound in (\ref{est_baz}) is then exactly
$1-r e^{-(1-r)t/\mu}$, which is at most one for every $t\ge0$.

The next proposition shows that the upper bound in (\ref{est_baz})
is asymptotically sharp as $r\uparrow1$, followed by $t\to\infty$.
\begin{prop}\label{p_prob_lim}
Assume that the renewal times are observable, that $F^*,F_0^*$ are
fixed, and that $0<\mu<\infty$. Let $r_0\to1$ as $r\uparrow1$.
\begin{description}
\item[1)] Suppose that $F_0^*(T_0)=F^*(T_0)=1$ for some $T_0>0$.
\begin{description}
\item[a)] If $F^*$ is non-arithmetic, then
$$
-\frac{T_0}{\mu}\le
\liminf_{t\to\infty}\liminf_{r\uparrow1}
\left(\frac{\sup_{A\in{\cal F}_t}|P^*\{A\}-P\{A\}|-(1-r_0)}{1-r}
-H_0^*(t)\right)\le0.
$$
The same assertion holds for arithmetic $F^*$.
\item[b)] In either case,
$$
\lim_{t\to\infty}\limsup_{r\uparrow1}
\left|
\frac{\sup_{A\in{\cal F}_t}|P^*\{A\}-P\{A\}|-(1-r_0)}{(1-r)H_0^*(t)}
-1\right|=0.
$$
\end{description}
\item[2)] Without the boundedness assumption in part 1, if
$1-r_0=O(1-r)$, then
\begin{equation}\label{tv_sharpness}
\lim_{t\to\infty}\limsup_{r\uparrow1}
\left|
\frac{\sup_{A\in{\cal F}_t}|P^*\{A\}-P\{A\}|}{(1-r)H_0^*(t)}
-1\right|=0.
\end{equation}
\end{description}
\end{prop}
{\bf Proof.} For $0<T<t$, let $A(t,T)$ be the observable event that
there is no renewal in $(t-T,t]$. Conditioning on the number of
renewals gives
\begin{equation}\label{renew_no_recent}
P(A(t,T))-P^*(A(t,T))=
{\bf E}^*[(1-p_{N(t-T)}){\bf1}_{\{N(t)>N(t-T)\}}],
\end{equation}
where ${\bf E}^*$ denotes expectation under $P^*$. Indeed, on a
projected path with a renewal in $(t-T,t]$, the original count has
no renewal in that interval precisely when it stops before the
first such renewal. The probability of this is $1-p_{N(t-T)}$.

{\bf 1)} Under the bounded-cycle assumption, every projected path
has a renewal in $(t-T_0,t]$ for $t>T_0$. Taking $T=T_0$ in
(\ref{renew_no_recent}) therefore gives
$$
\sup_{A\in{\cal F}_t}|P^*\{A\}-P\{A\}|
\ge1-r_0+(1-r)H_0(t-T_0).
$$
For each fixed $u\ge0$,
$$
H_0(u)=\sum_{n=0}^\infty p_nP^*\{S_n\le u\}
\longrightarrow H_0^*(u).
$$
Together with (\ref{est_baz}), this yields, for $t>T_0$,
\begin{eqnarray}\label{tv_compact}
H_0^*(t-T_0)&\leq&\liminf_{r\uparrow1}
\frac{\sup_{A\in{\cal F}_t}|P^*\{A\}-P\{A\}|-(1-r_0)}{1-r}\nonumber\\
&\leq&\limsup_{r\uparrow1}
\frac{\sup_{A\in{\cal F}_t}|P^*\{A\}-P\{A\}|-(1-r_0)}{1-r}
\leq H_0^*(t).
\end{eqnarray}
For non-arithmetic $F^*$, Blackwell's renewal theorem gives
$H_0^*(t)-H_0^*(t-T_0)\to T_0/\mu$, proving a).
For arithmetic $F^*$ with span $h$, put $b=h\lfloor T_0/h\rfloor$.
Every subsequent interval is at most $b$, so for $t>T_0$ the lower
bound in (\ref{tv_compact}) also holds with $T_0$ replaced by $b$.
The arithmetic renewal theorem, integrated over the bounded initial
delay, gives $H_0^*(t)-H_0^*(t-b)\to b/\mu\le T_0/\mu$,
which proves a) in this case as well. Dividing (\ref{tv_compact}) by $H_0^*(t)$
and using $H_0^*(t)\sim t/\mu$ shows that the absolute error in b) tends to zero.

{\bf 2)} To estimate the contribution with no renewal in $(t-T,t]$,
we sum over the last renewal before or at $t-T$:
\begin{eqnarray*}
{\bf E}^*[N(t-T){\bf1}_{\{N(t)=N(t-T)\}}]
&=&\sum_{n=0}^\infty(n+1)P^*\{S_n\le t-T,S_{n+1}>t\}\\
&\leq&[1-F^*(T)]\sum_{n=0}^\infty(n+1)P^*\{S_n\le t-T\}\\
&=&[1-F^*(T)](F_0^*\star H^*\star H^*)(t-T).
\end{eqnarray*}
Here the last equality follows by expanding $H^*\star H^*$:
the coefficient of $F^{*\,(\star n)}$ is $n+1$.
Since $(1-r^n)/(1-r)\to n$ and is bounded by $n$, dominated
convergence in (\ref{renew_no_recent}), with
${\bf E}^*N(t-T)=H_0^*(t-T)<\infty$, gives
$$
\liminf_{r\uparrow1}
\frac{\sup_{A\in{\cal F}_t}|P^*\{A\}-P\{A\}|}{1-r}
\geq H_0^*(t-T)-[1-F^*(T)](F_0^*\star H^*\star H^*)(t-T).
$$
The nonnegative contribution from $1-r_0$ has been discarded in
this lower bound. Put $T=\eta t$ with $0<\eta<1$.
Finite mean gives $t[1-F^*(\eta t)]\to0$; the renewal-power
asymptotic gives $(F_0^*\star H^*\star H^*)(t)=O(t^2)$ and
$H_0^*(t)\sim t/\mu$. The normalized lower limit is at least
$1-\eta$, while (\ref{est_baz}) and $1-r_0=O(1-r)$ give the
upper limit one. Letting $\eta\downarrow0$ shows that the absolute error in part 2 tends to zero. $\Box$

Renewal times are observable, for example, when they are
return-to-state or queue-empty times. Otherwise, the lower bounds
above concern observation of both the path and its renewal count;
the upper bound (\ref{est_baz}) remains valid for observation of
the path alone.

Theorem~\ref{t_prob_dif} and Proposition~\ref{p_prob_lim} establish
the rate of convergence of $P$ to its finite projection as
$r_0,r\to1$, under the stated assumptions.

\subsubsection{The distance between $P$ and $P^*$ on
the $\sigma$-algebra ${\cal F}$}

The preceding subsection compared the laws on a fixed observation
interval. We now compare them on the full canonical $\sigma$-algebra.
Their long-time behaviour can differ: under the usual regularity and
non-arithmetic (or discrete-time aperiodicity) conditions, a finite
positive mean projected cycle length gives an equilibrium limit
\cite{part1_1}, which the transient process need not share.

Put $E_n=\{S_n<\infty\}$ and
$C=\bigcap_{n\ge0}E_n$. Under $P^*$, $C$ has probability one.
If $r<1$, then $P(C)=0$, so $P$ and $P^*$ are mutually singular and
$$
\|P-P^*\|_{\rm TV}
=2\sup_{A\in{\cal F}}|P(A)-P^*(A)|=2.
$$
If $r=1$, then $P(C)=r_0$ and $P^*=P(\cdot\mid C)$; hence
$$
\|P-P^*\|_{\rm TV}=2(1-r_0).
$$
Thus a defect only in the initial delay does not in general give
mutual singularity.

To compare the laws through successive cycles, let
${\cal F}(n)=\sigma(C_0,\ldots,C_n)$, where the cycle records are
those of Section~\ref{par_1_2}. After the first infinite cycle, all
subsequent cycle records are assigned a fixed auxiliary symbol.
These records generate ${\cal F}$, so
$$
{\cal F}=\sigma\left(\bigcup_{n\ge0}{\cal F}(n)\right).
$$
The proof of Proposition~\ref{p1_r4} gives
$$
P^*|_{{\cal F}(n)}=P(\cdot\mid E_n)|_{{\cal F}(n)}.
$$

Let ${\cal P}$ be the probability measures on $(\Omega,{\cal F})$.
Define
$$
d_S(P_1,P_2)=\sum_{n=0}^\infty\frac1{n!}
\sup_{A\in{\cal F}(n)}|P_1(A)-P_2(A)|.
$$
This is a metric: each term is a pseudometric, the sum is bounded
by $e$, and equality of all restrictions determines the full law
because the increasing $\sigma$-algebras generate ${\cal F}$.
Unlike full total variation, it weights comparisons through a finite
number of cycles.

Write
$$
{\cal PR}=\{Q\in{\cal P}:Q(C)=1\}.
$$
These laws give each $S_n$ a proper distribution; in particular,
the set contains the laws under which $S$ is a recurrent renewal process.
\begin{prop}\label{p1_r3}
The finite projection $P^*$ is a nearest element of ${\cal PR}$ to
$P$ in the metric $d_S$, and
$$
d_S(P,{\cal PR})=d_S(P,P^*)=e-r_0e^r.
$$
\end{prop}
{\bf Proof.} Every $Q\in{\cal PR}$ assigns probability zero to
$E_n^c\in{\cal F}(n)$. Consequently
$$
\sup_{A\in{\cal F}(n)}|P(A)-Q(A)|\ge1-p_n.
$$
For $P^*$ the conditional-law identity gives, for $A\in{\cal F}(n)$,
$$
P(A)-P^*(A)=P(A\cap E_n^c)-(1-p_n)P^*(A).
$$
Both terms on the right lie in $[0,1-p_n]$, so its absolute value
is at most $1-p_n$, including when $p_n=1$. The lower bound is
therefore attained at every coordinate. Summing yields
$$
d_S(P,P^*)=\sum_{n=0}^\infty\frac{1-p_n}{n!}=e-r_0e^r.
\qquad\Box
$$

\noindent{\bf Corollary.} $d_S(P,P^*)\to0$ as $r_0,r\to1$.

Replacing $n!$ by $c^n$, $c>1$, gives the analogous best-approximation
result with distance
$$
\frac1{1-c^{-1}}-\frac{r_0}{1-r/c}.
$$
The nearest probability measure in ${\cal PR}$ need not be unique.
We now introduce a strictly convex norm and restrict the comparison
class to obtain uniqueness.

Let ${\cal L}={\cal L}(\Omega,\{{\cal F}(n)\},P)$ be the real finite
signed measures locally absolutely continuous with respect to $P$:
$Q|_{{\cal F}(n)}\ll P|_{{\cal F}(n)}$ for every $n$.
It is a linear space, with
$$
(Q+R)(A)=Q(A)+R(A),\qquad (\lambda Q)(A)=\lambda Q(A).
$$
Each $Q\in{\cal L}$ determines, up to $P$-a.s. equality, the
martingale of its local Radon--Nikodym derivatives
$$
f_n=\frac{d(Q|_{{\cal F}(n)})}{d(P|_{{\cal F}(n)})},\qquad
{\bf E}(f_{n+1}\mid{\cal F}(n))=f_n,\qquad {\bf E}|f_n|<\infty.
$$
This correspondence is injective. Its image consists of the
compatible local densities that extend to finite signed measures
on ${\cal F}$; not every integrable martingale has such an extension.
In particular, a finite signed measure necessarily satisfies
$\sup_n{\bf E}|f_n|\le |Q|(\Omega)<\infty$.

For $1<q<\infty$, let ${\cal L}^q$ consist of those $Q\in{\cal L}$
for which
\begin{equation}\label{p1_r10}
\|Q\|=\left(\sum_{n=0}^\infty\frac1{n!}
                 \int_\Omega|f_n|^q\,dP\right)^{1/q}<\infty.
\end{equation}
This is a norm.

\begin{teor}\label{part1_BSN}
The space ${\cal L}^q$, equipped with the norm~(\ref{p1_r10}), is
strictly convex.
\end{teor}
{\bf Proof.} The weighted direct sum of the spaces
$L^q(P,{\cal F}(n))$, with weights $1/n!$, is an $L^q$ space on the
disjoint union of the coordinate spaces and is strictly convex;
see~\cite{part1_12}, \cite{part1_13}. The map $Q\mapsto(f_n)$
embeds ${\cal L}^q$ isometrically in this space. Hence ${\cal L}^q$
is strictly convex.
In particular,
\begin{equation}\label{part1_r11}
\|F+G\|=\|F\|+\|G\|
\end{equation}
holds, for nonzero $F,G$, exactly when $F=\lambda G$ with
$\lambda>0$. Equality also holds if either measure is zero.
The criterion applies to the single weighted $L^q$ space, so the
proportionality constant is the same in every coordinate.\ $\Box$

The space ${\cal L}^q$ need not be complete; completeness is not
required for the following proposition.

Let
$$
{\cal PR}_q={\cal PR}\cap{\cal L}^q
=\{Q\in{\cal L}^q:Q\ge0,\ Q(\Omega)=Q(C)=1\}.
$$
This is a convex set of probability measures.
\begin{prop}\label{p1_r6}
The unique nearest element of ${\cal PR}_q$ to $P$ for the norm
(\ref{p1_r10}) is the finite projection $P^*$ with respect to $S$.
\end{prop}
{\bf Proof.} Its local densities are
$$
f_n^*=p_n^{-1}{\bf1}_{E_n},\qquad
\|P^*\|^q=\sum_{n=0}^\infty\frac{p_n^{1-q}}{n!}
=r_0^{1-q}\exp(r^{1-q})<\infty.
$$
Thus $P^*\in{\cal PR}_q$. For $Q\in{\cal PR}_q$ its local density
$f_n$ is nonnegative, vanishes on $E_n^c$, and
$\int_{E_n}f_n\,dP=1$. Consequently
\begin{equation}\label{p1_r12}
\|Q-P\|^q=\sum_{n=0}^\infty\frac1{n!}
 \left(\int_{E_n}|1-f_n|^q\,dP+1-p_n\right).
\end{equation}
Minkowski's and H\"older's inequalities~\cite{part1_13} give
$$
\left(\int_{E_n}|1-f_n|^q\,dP\right)^{1/q}
\ge\left(\int_{E_n}|f_n|^q\,dP\right)^{1/q}-p_n^{1/q}
\ge p_n^{(1-q)/q}-p_n^{1/q}.
$$
Raising to the $q$th power yields
\begin{equation}\label{p1_r13}
\int_{E_n}|1-f_n|^q\,dP
\ge p_n\left(\frac{1-p_n}{p_n}\right)^q.
\end{equation}
The density $f_n^*$ attains this bound for every $n$, proving
optimality after summation in~(\ref{p1_r12}). By strict convexity
(equivalently, the equality case in Jensen's inequality), equality
at a coordinate requires $f_n=p_n^{-1}$ a.s. on $E_n$.
Thus any minimizer has all the local densities of $P^*$.
Since the filtration generates ${\cal F}$, it is $P^*$ itself.\ $\Box$

Proposition~\ref{p1_r6} justifies the term \emph{projection} and gives
an equivalent characterization of the conditional-cycle construction.
\begin{def1}
The unique nearest probability measure to $P$ in ${\cal PR}_q$,
with the norm~(\ref{p1_r10}), is called the finite projection of
$P$ with respect to $S$.
\end{def1}

\subsection{\sloppy Finite projections of queueing systems.
Finite projection of the $M/G/1$ queue}
\label{par_1_4}

A queueing model is specified on a probability space
$(\Omega,{\cal F},P)$ by its input sequences (the terminology of
\cite{part1_2}): arrival times, service requirements, batch sizes when
customers arrive or are served in groups, and any other random inputs.
The service discipline and the rules determining when service begins
must also be specified.

As in Sections~\ref{par_1_1}--\ref{par_1_2}, we use the uncompleted
$\sigma$-algebra generated by the input sequences and cycle records.
The system's characteristics are measurable functions on this space,
with laws determined by $P$ and the service rule.

To construct a finite projection, choose a transient regenerative
structure in the model. Different choices can give different
projections. For example, the queue-length process of an unstable
$M/G/1$ queue regenerates when a customer arrives to an empty system.
Under FIFO service it also regenerates at successive departures
leaving a fixed number $n$ of customers, whenever the corresponding
return structure satisfies the conditions of Section~\ref{par_1_1}.
For Markovian queueing models, returns to a fixed state give further
examples. These include $M/M/1$, $M/M/n$, $E_n/M/k$ and $M/E_m/k$;
in the Erlang cases the Markov state includes the relevant arrival
or service phases.

For a $GI/G/n/\infty$ queue, consider arrivals to an empty system as
the starts of busy periods. When these epochs define a regenerative
structure, a cycle consists of a busy period and the following idle
period. We assume the positive finite-cycle probabilities required
in Section~\ref{par_1_1}.
\begin{def1}
Let ${\cal S}$ be a $GI/G/n/\infty$ queueing system with a transient
regenerative structure $S=\{S_k\}$ given by these busy-period starts.
Its \emph{finite projection} ${\cal S}^*$ has the same input and
performance variables as measurable functions on $(\Omega,{\cal F})$,
with their laws determined by the finite projection $P^*$ of $P$
with respect to $S$.
\end{def1}

Thus determining the finite projection means determining the joint
law of the input sequences under $P^*$. All other characteristics
follow from the same measurable queueing rules. In cycle terms,
the original cycle law is replaced by its law conditional on
finiteness, as in Section~\ref{par_1_2}.

Consider an $M/G/1$ queue with Poisson arrivals of rate $a>0$ and
i.i.d. nonnegative service requirements with distribution function
$B$, independent of the arrival process. Interarrival times are
therefore i.i.d. exponential variables of rate $a$. There is one
server and unlimited waiting space. A customer finding the system
empty starts service immediately; otherwise it waits according to
the service discipline, and leaves after completing service. The
discipline is work-conserving and nonanticipating, with the same rule
restarted at each regeneration and no state carried from earlier cycles.
We describe a
regeneration cycle starting with an arrival to an empty system;
an initial delay, if present, is treated as in Section~\ref{par_1_2}.

Assume
$$
b_1=\int_{[0,\infty)}y\,dB(y)<\infty,\qquad ab_1>1.
$$
The queue is then unstable, and its busy periods can be infinite.

\begin{teor}\label{p1_r9}
Each projected regeneration cycle consists of an $M/G/1$ busy
period, generated by Poisson arrivals of rate $a^*=ar$ and an
independent i.i.d. service sequence with distribution $B^*$, followed
by an independent exponential idle period of rate $a$, where
$$
B^*(t)=r^{-1}\int_{[0,t]}e^{-a(1-r)y}\,dB(y),\qquad t\ge0,
$$
and $r$ is the unique root in $(0,1)$ of
$$
r=\beta(a-ar),\qquad
\beta(s)=\int_{[0,\infty)}e^{-sy}\,dB(y).
$$
The projected queue is regenerative with finite mean cycle length
and has an equilibrium distribution.
\end{teor}

{\bf Proof.} By regeneration, it suffices to identify one complete
cycle. We may first use nonpreemptive LIFO service: after each
completion, the last customer to have joined the queue is served.
For a fixed arrival and service record, the busy-period duration
depends only on the total workload, not on the work-conserving
service discipline.

Call the arrivals during a customer's service its children, and
include their descendants in its family. Call a customer good if
its family is finite, and bad otherwise. The busy period is finite
exactly when its initial customer is good. The Poisson input and
independent service requirements give independent potential families
for distinct children. These may be specified in advance, including
families not reached if an earlier family is infinite.

Let $r$ be the probability that a customer is good. Conditional on
its service requirement $Y=y$, the number $N$ of its children is
Poisson with mean $ay$. The customer is good exactly when all its
children are good. Hence
$$
P\{\hbox{good}\mid Y=y\}
=\sum_{n=0}^{\infty}e^{-ay}\frac{(ay)^n}{n!}r^n
=e^{-a(1-r)y},
$$
and integration gives $r=\beta(a-ar)$. Extinction by successive
generations shows that $r$ is the smallest fixed point of
$f(z)=\beta(a-az)$. Since $f(0)>0$, $f(1)=1$, $f'(1)=ab_1>1$
and $f$ is strictly convex, this is the unique root in $(0,1)$.

Conditioning the customer to be good now gives
$$
P\{Y\in dy\mid\hbox{good}\}
=r^{-1}e^{-a(1-r)y}\,dB(y)=dB^*(y).
$$
Moreover,
$$
P\{N=n\mid Y=y,\hbox{good}\}
=\frac{e^{-ay}(ay)^n r^n/n!}{e^{-a(1-r)y}}
=e^{-ary}\frac{(ary)^n}{n!}.
$$
Given $Y=y$ and $N=n$, the factor $r^n$ does not depend on the
arrival times. They therefore retain the distribution of the ordered
points of $n$ independent uniform variables on $(0,y)$. Thus the
arrivals during this service have the Poisson law with rate $ar$.
For $y=0$, there are no children; the formulas include an atom of
$B$ at zero.

Given these arrivals, conditioning all the child families to be
finite leaves them independent, each with the original family law
conditioned on finiteness. Repeating the same argument in each
family identifies the complete busy-period record with that of an
$M/G/1$ queue driven by independent Poisson arrivals of rate $ar$
and i.i.d. service requirements with distribution $B^*$.
This justifies the good/bad-customer thinning argument without
assuming independence between the labels of ancestors and descendants.
The same conditional input law applies under the original service
discipline, since the event of a finite busy period depends only on
the workload record.

After an original finite busy period, the strong Markov property
of the Poisson input gives an independent idle period with
distribution $\mathrm{Exp}(a)$. Conditioning the regeneration
cycle on finiteness therefore leaves this idle law unchanged.
The projected cycles are independent copies of the conditioned busy
period followed by this original idle period.

Finally,
$$
\beta^*(s)=r^{-1}\beta(s+a-ar),\qquad
b_1^*=-\beta^{*\prime}(0)=-r^{-1}\beta'(a-ar).
$$
Strict convexity gives
$$
ar b_1^*=f'(r)<\frac{f(1)-f(r)}{1-r}=1.
$$
The projected busy period consequently has mean
$b_1^*/(1-ar b_1^*)$; see~\cite{part1_1}. The mean cycle length is
$$
\frac{b_1^*}{1-ar b_1^*}+\frac1a<\infty.
$$
The independent exponential idle component makes the cycle law
non-arithmetic. The regenerative limit theorem therefore gives
the equilibrium distribution.\ $\Box$

Taking $B(t)=1-e^{-bt}$, $b>0$, gives $r=b/a$ for $a>b$ and the
following special case.

\noindent{\bf Corollary.} If the original $M/M/1$ queue has arrival
rate $a$ and service rate $b$, with $a>b$, its finite projection has
service rate $b^*=a$, arrival rate $a^*=b$ while busy, and independent
idle periods with distribution $1-e^{-at}$.

The projection is therefore the dual queue with the original idle
periods. Its property of best approximation is understood with respect
to the cycle filtration, metrics and comparison classes of
Section~\ref{par_1_3}.

The projected $(M/M/1)^*$ queue is examined further in the chapter
on birth-and-death processes.

\bigskip
\centerline{\large\bf Conclusions and observations}

The finite projection of a transient regenerative process is recurrent.
On finite observation intervals, Theorem~\ref{t_prob_dif} gives
approximation bounds; when renewal times are observed,
Proposition~\ref{p_prob_lim} gives asymptotically sharp bounds.
It is also a nearest probability law in the comparison classes and
metrics of Propositions~\ref{p1_r3} and~\ref{p1_r6}.
Under the stated regeneration conditions,
finite projection also applies to queueing systems with traffic
intensity greater than one.

The class of transient regenerative processes is very broad, so
explicit descriptions of their finite projections are not always
available. This also applies to queueing systems. For narrower
classes, such as the $M/G/1$ model in Theorem~\ref{p1_r9}, the
structure of the finite projection can be described explicitly.

\clearpage

\section{\sloppy Finite Projections of Birth and Death Processes} \label{par_2_}
\setcounter{equation}{0}

This section is dedicated to the finite projection of birth and death processes and can be treated as an application of some ideas presented in the first section and in \cite{part1_4}, \cite{l_pnm7}. The primary objective is to determine the finite projection of birth and death processes (BDP) and study some of its properties. In particular, we are interested in the connection between the finite projection and the dual of the BDP.

Let $X = \{X(t):0\le t<\zeta\}$ be the minimal BDP with birth rates
$\{\lambda_n\}_{n\ge0}$ and death rates $\{\mu_n\}_{n\ge1}$,
all finite and strictly positive, on ${E}=\{0,1,2,\ldots\}$.
Its lifetime $\zeta$ may be finite; no continuation after explosion
is included. All hitting and return times below refer to visits before
$\zeta$ and equal infinity if the required visit does not occur.
The law $P$ satisfies $P\{X(0)=0\}=1$.

First, we will need some auxiliary results concerning the time spent by the BDP $X$ in a given set of states.

\subsection{Time Spent by the BDP in a Given Set of States} \label{par_2_1}

Set $T_0=0$. Let $T_n$ be the random variable (r.v.) representing the hitting time of state $n$ for the BDP, i.e., $T_n = \inf\{0<t<\zeta; X(t)=n\}$, $n=1,2,\ldots$. The r.v. $W_{n}=\inf\{T_n<t<\zeta; X(t)=n-1\}-T_n$ ($W_{n}=\infty$ if such $t$ does not exist) will be called the {\sl time spent} by the BDP $X$ in the set of states $E_{n}=\{n,n+1,n+2,\ldots \}$, $n\geq 1$. It measures the passage from $n$ to $n-1$; explosion before that passage is completed gives $W_n=\infty$.

We note that $W_{1}$ coincides with the busy period of the corresponding queueing system.

Let ${w}_{n}(s)$ denote the Laplace-Stieltjes Transform (LST) of the distribution function (d.f.) of the r.v. $W_{n}$, i.e., ${w}_{n}(s) = \int^{+\infty}_{0}e^{-st}\,d{P}(W_{n}<t)$, $n=1,2,3,\ldots$.

\begin{prop}\label{mm1_r3}
The LST of the d.f. of the r.v. $W_{n}$ $(n=1,2,\ldots)$ has the following representation in the form of a continued fraction:
\begin{equation}
{w}_{n}(s) = \frac{\mu_{n}}{s+\mu_{n}+\lambda_{n}-}\ \frac{\lambda_{n} \mu_{n+1}}{s+\mu_{n+1}+\lambda_{n+1}-}\ \frac{\lambda_{n+1}\mu_{n+2}}{s+\mu_{n+2}+\lambda_{n+2}- \ldots }
\label{pnm1}
\end{equation}
where the continued fraction (\ref{pnm1}) converges to a holomorphic function in the domain ${\cal D} = \{ s \in {C} \mid \mathrm{Re}(s)>0 \}$, and the convergence is uniform on any compact subset of ${\cal D}$.
\end{prop}

{\bf Proof.} Assume that ``catastrophes'' occur independently of the behavior of the process $X$, forming a Poisson stream of events with parameter $s>0$. Then, the LST ${w}_{n}(s)$ can be interpreted as the probability that no ``catastrophes'' occur during the time interval $W_{n}$ (see \cite{l_pnm5} for the ``catastrophe method'').

Let $t_{0}$ be an arbitrary moment when $X(t_{0})=n$, and $\tau$ be the time spent by the BDP $X$ in state $n$ starting from $t_{0}$. As we know for BDPs:
$$
{P}(X(t_{0}+\tau)=j,\tau<t \mid X(t_{0})=n ) = \left\{
\begin{array} {ll}
        \frac{\lambda_{n}}{\mu_{n}+\lambda_{n}}\left( 1-e^{-(\lambda_{n}+
  \mu_{n})t} \right), & \hbox{$j=n+1$} \\
	\frac{\mu_{n}}{\mu_{n}+\lambda_{n}}\left( 1-e^{-(\lambda_{n}+
  \mu_{n})t} \right), & \hbox{$j=n-1$} \\
  0 , & \hbox{$\vert j-n\vert \not=1$}.
\end{array}  \right.
$$

Conditioning on the first jump and using the strong Markov property gives
$$
w_n(s)=\frac{\mu_n+\lambda_n w_{n+1}(s)w_n(s)}
{s+\lambda_n+\mu_n},\qquad s>0.
$$

Thus,
\begin{equation}
{w}_{n}(s) = \frac{\mu_{n}}{\mu_{n}+\lambda_{n}+s-\lambda_{n} {w}_{n+1}(s)},
\label{pnm2}
\end{equation}
and by iterating,
$$
{w}_{n}(s) = \frac{\mu_{n}}{\mu_{n}+\lambda_{n}+s-\frac{\lambda_{n}\mu_{n+1}}{\mu_{n+1}+\lambda_{n+1}+s-\lambda_{n+1}{w}_{n+2}(s)}} = \cdots =
$$
$$
={\frac{\mu_n}{s+\mu_n+\lambda_n-}}\
 {\frac{\lambda_{n}\mu_{n+1}}{s+\mu_{n+1}+\lambda_{n+1}-}}\
 {\frac{\lambda_{n+1}\mu_{n+2}}{s+\mu_{n+2}+\lambda_{n+2}- \ldots}}\  ,
$$
Stop the process upon reaching level $N+1$. The transform of hitting
$n-1$ before $N+1$, starting from $n$, is the fraction truncated at $N$.
Its difference from $w_n(s)$ is bounded, uniformly for
$\mathrm{Re}(s)\geq0$, by the probability of reaching $N+1$ before
eventually hitting $n-1$. This probability tends to zero, since every
finite passage visits finitely many states. At $s=0$, the transforms
count finite passages only.
Under the stated rate assumptions, (\ref{pnm1}) is a continued fraction that converges, by Stieltjes' theorem \cite{l_pnm3}, to a holomorphic function ${w}_{n}(s)$ in the domain ${\cal D} = \{s \in {C} \mid \mathrm{Re}(s)>0 \}$, and this convergence is uniform on any compact subset of ${\cal D}$. $\Box$

As we know, if $X(t)$ is a transient BDP, then $\rho_{n} = {P}(W_{n}< +\infty) = {w}_{n}(0) < 1$. Thus, from (\ref{pnm1}), we have:

{\bf Corollary.}
{\it If the BDP $X(t)$ described by the parameters
  $\{\lambda _{n}\}_{n=0}^{+\infty}$,
 $\{\mu_{n}\}_{n=1}^{+\infty}$ is transient, then
$$\rho_{n}={P}(W_{n}<+\infty)={w}_{n}(0)=$$
\begin{equation}
      = {\frac{\mu_n}{\mu_n+\lambda_n-}}\
       {\frac{\lambda_{n}\mu_{n+1}}{\mu_{n+1}+\lambda_{n+1}-}}\
       {\frac{\lambda_{n+1}\mu_{n+2}}{\mu_{n+2}+\lambda_{n+2}- \ldots}}
\label{pnm3}
\end{equation}
 $n=1,2,\ldots .$
}
  
\begin{samepage}
From equality (\ref{pnm2}) and the properties of Laplace-Stieltjes Transforms, we have for $n \geq 1$
\begin{description}
\item[(a)] \qquad $\lim\limits_{s\to\infty} s w_n(s)=\mu_n$, 
\item[(b)] \qquad $\lim\limits_{s\to\infty} \left ( 
  \mu_n-\mu_n w_n(s)-s w_n(s)\right )/ w_n(s)=\lambda_n$.
\end{description}
\par
\end{samepage}

Point (a) is evident, and (b) follows from (\ref{pnm2}) and the equality $\lim\limits_{s\to\infty} w_{n+1}(s)=0$.

From (a) and (b), it follows in particular that the representation (\ref{pnm1}) of the LST $w_n(s)$ is unique.

Let $W_{\{0,\ldots,n\}}=\inf\{t> T_n: X(t)=n+1\}-T_n$ be the time spent by the BDP $X$ in the set of states $\{0,\ldots,n\}$. From equality (\ref{pnm2}) and relations (a)-(b), it follows that:
\begin{prop}
The Markov probability $P$ is uniquely determined by the distributions of the r.v.'s $W_{\{0,\ldots,n\}}$ and $W_{n+1}$, for any fixed $n \geq 0$.
\end{prop}

{\bf Proof.} The transform of $W_{n+1}$ determines successively the
rates at $n+1,n+2,\ldots$ by (a), (b) and~(\ref{pnm2}). The analogous
first-step recursion and rate-recovery limits for upward passages,
with birth and death rates interchanged, determine the rates at
$n,n-1,\ldots,0$ from the distribution of $W_{\{0,\ldots,n\}}$.
At state $0$, the upward passage time is exponential with parameter
$\lambda_0$. The rates and the specified initial state determine
the law of the minimal process. $\Box$

Recall that the initial state of the process $X$ is zero. For an arbitrary initial distribution, the same two passage laws, starting from $n$ and $n+1$ respectively, together with the distribution of $X(0)$, determine $P$.

For $j<i$, let $W_{i,j}$ be the length of the time interval starting at the first moment when $X$ reaches state $i$ and ending at the first moment when the process $X$ enters state $j$. More precisely:
$W_{i,j}= \inf\{t>T_i; X(t)=j\}-T_i$ ($W_{i,j}=\infty$ if such $t$ does not exist).
Evidently, $W_{n}=W_{n,n-1}$.
\begin{prop} \label{mmn_r3}
The LST of the d.f. of the r.v. $W_{i,j}$ is given by:
\begin{equation}
{w}_{i,j}(s)=\prod_{k=j+1}^{i}{w}_{k}(s), \quad i=1,2,\ldots,\; j<i,
\label{pnm4}
\end{equation}
where ${w}_{k}(s)$ is determined in (\ref{pnm1}) for any $k=1,2,\ldots$.
\end{prop}

{\bf Proof.} To reach $j<i$, the process must successively reach $i-1,i-2,\ldots,j$. Applying the strong Markov property at these hitting times gives $w_{i,j}=w_iw_{i-1}\cdots w_{j+1}$. $\Box$

Let the r.v. $B_n$ represent the total time spent by $X$ in state $n$
from $T_n$ until the downward passage or explosion, whichever occurs first.

\begin{prop} The d.f. of the r.v. $B_{n}$ is exponential with parameter $\mu_{n}+(1-\rho _{n+1})\lambda _{n}$, where $\rho _{n+1}={ P}( W _{n+1}<+\infty )={w} _{n+1}(0)$ is calculated from (\ref{pnm3}), $n=1,2,\ldots$.
\end{prop}

{\bf Proof.} Conditioning on the first departure from $n$, the LST of
$B_n$ satisfies
$$
b_n(s)=\frac{\mu_n+\lambda_n(1-\rho_{n+1})+
\lambda_n\rho_{n+1}b_n(s)}{s+\lambda_n+\mu_n}.
$$
The three terms correspond to a downward departure, an upward
excursion that never returns, and an upward excursion that returns.
Only time spent at $n$ contributes to $B_n$.

Thus,
$$
{b}_{n}(s)=\frac{\mu_{n}+(1-\rho_{n+1})\lambda_{n}}
  {s+\mu_{n}+(1-\rho_{n+1})\lambda_{n}},\qquad n=1,2,\ldots
$$
which corresponds to the LST of an exponentially distributed r.v. with parameter $\mu_{n}+(1-\rho_{n+1})\lambda_{n}$, where $\rho_{n+1}={ P}(W_{n+1}<+\infty )$ is determined by (\ref{pnm3}). $~\Box$

%%%%%%%%%%%%%%%%%%%%%%%%%%%%%%%%%%%%%%%%%%%%%%%%%%%%%%%%%%%%%%%%%%%%%%
%%%%%%%%%%%%%%%%%%%%%%%%%%%%%%%%%%%%%%%%%%%%%%%%%%%%%%%%%%%%%%%%%%%%%%

\subsection{Finite Projection and the Dual}  \label{par_2_2}

The successive returns of $X$ to state $0$ before its lifetime form
a possibly stopped renewal process $S=\{S_n\}$. The finite projection
is constructed from the conditional laws of its finite cycles, on the
joint space of the embedded chain and holding times as in
Section~\ref{semi_proj}. It is of interest for transient BDPs, where
$0<P(W_1<+\infty)<1$.

\begin{teor} \label{p2_r1}
The finite projection $X^{*}$ with respect to the renewal process $S$ for a transient BDP $X$ determined by the parameters $\{\lambda _{n}\}_{n=0}^{+\infty}$, $\{\mu_{n}\}_{n=1}^{+\infty}$ is a nonexplosive BDP $X^{*}$ defined by the parameters $\{\lambda^{*}_{n}\}_{n=0}^{+\infty}$, $\{\mu^{*}_{n}\}_{n=1}^{+\infty}$, where
$$
\begin{array}{ll}
\lambda_0^*=\lambda_0, &\\[1mm]
\lambda_n^*=\lambda_n\rho_{n+1},\qquad
\mu_n^*=\displaystyle\frac{\mu_n}{\rho_n}, & n\ge1.
\end{array}
$$
and $\rho_n$ are determined from (\ref{pnm3}).
\end{teor}

{\bf Proof.} Write $P_i$ for the law starting from state $i$, and
let $\rho_{i,0}$ be the probability of ever hitting $0$ under $P_i$,
for $i\geq1$. Set $\rho_{0,0}=1$. Starting from $i\geq1$, let
$\tau$ be the first holding time and $J$ the next state. By the strong
Markov property,
\begin{equation}
P_i\{\tau\leq t,J=j\mid\hbox{eventual return to }0\}
=P_i\{\tau\leq t,J=j\}\frac{\rho_{j,0}}{\rho_{i,0}},
\qquad j=i\pm1.
\label{pnm5}
\end{equation}
From (\ref{pnm4}) and (\ref{pnm3}),
\begin{equation}
\rho_{n,0}=\prod_{k=1}^{n}\rho_k,\qquad n\geq1.
\label{pnm6}
\end{equation}
Thus the conditional jump rates at $i\geq1$ are
$\lambda_i\rho_{i+1}$ and $\mu_i/\rho_i$. Equation~(\ref{pnm2}) at
$s=0$ gives
$$
\lambda_i\rho_{i+1}+\frac{\mu_i}{\rho_i}=\lambda_i+\mu_i,
\qquad i\geq1,
$$
so the exponential holding-time parameter is unchanged. At $0$, the
idle time is exponential with parameter $\lambda_0$ and independent
of the following excursion. Conditioning that excursion to return
therefore leaves the idle law unchanged.

Concatenating independent cycles with these conditioned laws gives
the finite projection defined in Section~\ref{par_1_2}. Each conditioned
return occurs before explosion and uses finitely many jumps. The successive
independent exponential idle periods have parameter $\lambda_0$ and
an almost surely divergent sum, preventing accumulation of cycles in finite time.
Its infinitesimal transition probabilities are consequently
$$
P\{X^*(t+h)=j\mid X^*(t)=i\}=
\left\{\begin{array}{ll}
\lambda_0h+o(h),&i=0,\ j=1,\\
1-\lambda_0h+o(h),&i=j=0,\\
\lambda_i\rho_{i+1}h+o(h),&i\geq1,\ j=i+1,\\
\mu_i h/\rho_i+o(h),&i\geq1,\ j=i-1,\\
1-(\lambda_i\rho_{i+1}+\mu_i/\rho_i)h+o(h),&i=j\geq1,\\
o(h),&|j-i|\geq2.
\end{array}\right.
$$
These are the stated birth and death rates. $\Box$

Let $W^{*}_{i}$ and $W^{*}_{i,j}$ represent random variables defined for the BDP $X^{*}$ in the same way as $W_{i}$, $W_{i,j}$ are for the BDP $X$, $0\leq j<i$. It is expected that the random variables $W^{*}_{i}$, $W^{*}_{i,j}$, $i=1,2,\ldots$, $0\leq j<i$ coincide (in distribution) with the respective finite projections (in the sense of the definition in \cite{part1_4}) of the random variables $W_{i}, W_{i,j}$.
%%%%%%%%%%%%%%%%%%%%%%%%%%%%%%%%%%%%%%%%%%%%%%%%%%%%%%%%%%%%%%%%%%%%%%
\begin{prop}\label{mmn_r4}
\begin{description}
\item[a)] LST of the d.f. for the r.v. $W^{*}_{i}$
\begin{equation}
{w}^{*}_{i}(s)={w}_{i}(s)\rho_{i}^{-1}, i=1,2,\ldots;
\label{pnm7}
\end{equation}
\item[b)] LST of the d.f. for the r.v. $W^{*}_{i,j}$
\begin{equation}
{w}^{*}_{i,j}(s)=\prod_{k=j+1}^i\frac{{w}_{k}(s)}{\rho_k},
\qquad i=1,2,\ldots,\quad i>j.
\label{pnm8}
\end{equation}
\end{description}
\end{prop}

{\bf Proof.} We will prove part (a) using mathematical induction. For $i=1$, from
$$
{P}(W^{*}_{1}<t) =  P(W_{1}<t | \hbox{the respective renewal interval is finite})=
$$
$$
={P}(W_{1}<t| W_{1}<+\infty)=\frac{{P}(W_{1}<t)}{{P}(W_{1}<+\infty)},
$$
we have ${w}^{*}_{1}(s)={w}_{1}(s)/\rho_{1}$.

Assume that formula (\ref{pnm7}) holds for $i=k$. Since $X^{*}$ is also a BDP and formula (\ref{pnm2}) holds for any $n=1,2,\ldots$, we have
$$
{w}^{*}_{n}(s)=\frac{\mu^{*}_{n}}{\lambda^{*}_{n}+\mu^{*}_{n}+s-\lambda^{*}_{n}{w}^{*}_{n+1}(s)}.
$$
Thus, by induction, from Theorem \ref{p2_r1}
\begin{equation}
{w}^{*}_{k}(s)={w}_{k}(s)/\rho_{k}=\frac{\mu_{k}/\rho_{k}}{\lambda_{k}\rho_{k+1}+\mu_{k}/\rho_{k}+s-\lambda_{k}\rho_{k+1}{w}^{*}_{k+1}(s)}.
\label{pnm9}
\end{equation}
From (\ref{pnm2}), we have
$$
\lambda_{k}\rho_{k+1}+ \mu_{k}/\rho_{k}=\lambda _{k}+\mu_{k}.
$$
Then (\ref{pnm9}) implies
\begin{equation}
{w}_{k}(s) =\frac{\mu_{k}}{\lambda_{k}+\mu_{k}+s-\lambda_{k}\rho_{k+1}{w}^{*}_{k+1}(s)}.
\label{pnm10}
\end{equation}
Comparing (\ref{pnm10}) with (\ref{pnm2}), we obtain
$$
{w}^{*}_{k+1}(s)={w}_{k+1}(s)/\rho_{k+1},
$$
thus concluding the proof of part (a).

Evidently, part (b) follows from part (a) and Proposition \ref{mmn_r3}.~$\Box$

It is interesting to clarify whether the finite projection of any BDP is ergodic or not. The answer to this question is negative.

{\bf Example.} Let $X_{k}=\{X_{k}(\omega ,t),t\in  [0,+\infty )\}$ 
be a BDP described by the birth rates $\lambda_{n}=k+n, n=0,1,2,\ldots $,
and death rates $\mu_{n}=n, n=1,2,\ldots$,
where the number $k$ is an integer, $k\ge 1$. Then, for this process
\begin{equation}
{w} _{1}(s)=\frac{k+s}{k}-\frac{1}{k e^s I_{k}(s)},
\label{pnm11}
\end{equation}
where $s \in \{ s\in  C| Re(s)>0 \}$ and 
$$
I_{k}(s) = \int^{+\infty }_{1}\frac{e^{-st}}{t^{k}}\,dt.
$$

{\bf Proof.} From (\ref{pnm1}) we have
$$
{w} _{1}(s)={\frac{1}{(k+1)+1+s-}}\ {\frac{(k+1)2}{(k+2)+2+s-}}\
	    {\frac{(k+2)3}{(k+3)+3+s- \ldots}} .
$$

Introduce the continued fraction
$$
f_{k}(s)=\frac{1}{k+s-k{w}_{1}(s)}=
$$
\begin{equation}
={\frac{1}{k+s-}}\ {\frac{k}{(k+1)+1+s-}}\ {\frac{(k+1)2}{(k+2)+2+s-}}\
 {\frac{(k+2)3}{(k+3)+3+s- \ldots  }}
\label{pnm12}
\end{equation}

It is easy to verify that the continued fraction (\ref{pnm12}) is
the even part of the continued fraction
\begin{equation}
f(s) = {\frac{1}{s+}}\ {\frac{k}{1+}}\ {\frac{1}{s+}}\ {\frac{k+1}{1+}}\
	{\frac{2}{s+}}\ {\frac{k+2}{1+ \ldots}}.
\label{pnm13}
\end{equation}
$$
\hbox{Since the series}\ \ \ \sum\limits^{+\infty}_{n=1}
\prod\limits_{i=1}^{n}\frac{\lambda_{i-1}}{\mu_{i}}=
\sum\limits^{+\infty }_{n=1}\prod\limits_{i=1}^{n} 
\frac{(i+k-1)}{i}\ \ \ \hbox{diverges for}\ \ k=1,2,\ldots,
$$
according to the Stieltjes theorem \cite{l_pnm3},
the fraction (\ref{pnm13}) converges, and the limiting functions
of (\ref{pnm12}) and (\ref{pnm13}) coincide,
i.e., $f_{k}(s)\equiv f(s) $.

Thus,
$$
\frac{1}{k+s-k{w}_{1}(s)}= f(s) .
$$

As we know from \cite{l_pnm3}
$$
e^{-s}f(s) = e^{-s}\left({\frac{1}{s+}}\ {\frac{k}{1+}}\
	 {\frac{1}{s+}}\ {\frac{k+1}{1+}}\ {\frac{2}{s+}}\
         {\frac{k+2}{1+ \ldots}} \right) =
 \int^{+\infty}_{1} \frac{e^{-st}}{t^{k}}\, dt = I_{k}(s),
$$
for $k=0,1,2,\ldots$ and $s\in \{s\in { C}| Re(s)>0\}$. This
means that
$$
\frac{1}{k+s-k{w}_1(s)}= e^{s}I_{k}(s),
$$
i.e.,
$$
{w}_{1}(s) = \frac{k+s}{k}-\frac{1}{ke^{s}I_{k}(s)},
$$
where $s \in  \{ s\in { C}| Re(s)>0\}, k=1,2,\ldots$.~$\Box$

It is easy to verify that, in this example,
$$
\rho _{1}= { P}(W_{1}<+\infty )= {w}_{1}(0)=k^{-1}.
$$

Thus, BDP $X_{k}$ is transient for $k\ge 2.$

For the projected BDP, a stationary distribution exists if and only if
the mean return cycle is finite \cite{part1_4}. Since the idle mean is
$1/\lambda_0$, this is equivalent to ${E}W_1^*<\infty$.
For
the finite projection $X^{*}_{k}$ of the BDP $X_{k}$ we have
$$
{ E} W^{*}_{1}=\left.-\frac{\partial{w}^{*}_{1}(s)}{\partial s}\right|_{s=0}
 = - k\left({w}_{1}(s)\right)'_{s=0},
$$
where
$$
\left({w}_{1}(s)\right)'=\left(\frac{k+s}{k}-\frac{1}{ke^{s}I_{k}(s)}\right)'
= k^{-1}\left(1+\frac{I_{k}(s)+I_{k}'(s)}{e^{s}I^{2}_{k}(s)}\right) =
$$
$$
= k^{-1}\left(1+\frac{I_{k}(s)-I_{k-1}(s)}{e^{s}I^{2}_{k}(s)}\right).
$$

%%%%%%%%%%%%%%%%%%%%%%%%%%%%%%%%%%%%%%%%%%%%%%%%%%%%%%%%%%%%%%%%%%%%%%%
%%%%%%%%%%%%%%%%%%%%%%%%%%%%%%%%%%%%%%%%%%%%%%%%%%%%%%%%%%%%%%%%%%%%%%%
%%%%%%%%%%%%%%%%%%%%%%%%%%%%%%%%%%%%%%%%%%%%%%%%%%%%%%%%%%%%%%%%%%%%%%%
%%%%%%%%%%%%%%%%%%%%%%%%%%%%%%%%%%%%%%%%%%%%%%%%%%%%%%%%%%%%%%%%%%%%%%%
Since
$$
I_{k}(0)=\left\{\begin{array}{ll}
  +\infty & \mbox{for $k=0,1$} \\
  \frac{1}{k-1} & \mbox{for $k=2,3,\ldots$}
 \end{array} \right.,
$$
we have
$$
{ E} W^{*}_{1} = - k\left({w}_{1}(s)\right)'_{s=0}=
  \left\{\begin{array}{ll}
  +\infty & \mbox{for $k=1,2$} \\
  \frac{1}{k-2} & \mbox{for $k=3,4,\ldots$}
  \end{array} \right. .
$$

The BDP $X_{2}$ provides an example of a transient BDP whose finite projection $X^{*}_{2}$ is null recurrent since ${ P}(W^{*}_{1}<+\infty)=1$, while ${ E} W^{*}_{1}=+\infty$. Thus, $X^{*}_{2}$ is not ergodic. At the same time, the finite projections $X^{*}_{k}$ of the transient BDPs $X_{k}$ are ergodic for $k \geq 3$. Hence, we must look for conditions under which the finite projection $X^{*}$ of a transient BDP $X$ is ergodic.

\begin{prop}\label{mmn_r7}
The finite projection with respect to the renewal process $S$ of the transient BDP $X$ described by the birth rates $\{\lambda_{n}\}_{n=0}^{+\infty}$ and the death rates $\{\mu_{n}\}_{n=1}^{+\infty}$ is ergodic if one of the following conditions is satisfied:
\begin{description}
\item[a)]
 $ \sup_{n} \rho_{n}\neq 1 $ and $\ \inf_{n} \lambda_{n}\neq 0 $;
\item[b)]
 $ \sup_{n} \rho_{n}\neq 1 $ and $\ \inf_{n} \mu_{n}\neq 0 $;
\item[c)]
 $\sum\limits^{+\infty}_{n=1}(\lambda_{n}-\lambda_{n}\rho_{n+1})^{-1}<+\infty $;
\item[d)]
 $\sum\limits^{+\infty }_{n=2}(\mu_{n}-\mu_{n}\rho_{n})^{-1} <+\infty $.
\end{description}
\end{prop}

{\bf Proof.} According to Theorem \ref{p2_r1}, the finite projection $X^{*}(t)$ is a BDP described by the parameters
$$
\begin{array}{ll}
\lambda_0^*=\lambda_0, &\\[1mm]
\lambda_n^*=\lambda_n\rho_{n+1},\qquad
\mu_n^*=\displaystyle\frac{\mu_n}{\rho_n}, & n\ge1.
\end{array}
$$
where $\rho_{n}, n=1,2,\ldots$ are determined by (\ref{pnm3}).

For $X^{*}$ to be ergodic, it is sufficient (see \cite{part1_1}) that the series
\begin{equation}
S =\sum^{+\infty }_{n=0} \prod_{i=0}^{n}\frac{\lambda^{*}_{i}}{\mu^{*}_{i+1}}=
\rho_{1}^{-1}\sum^{+\infty}_{n=0} \prod_{i=0}^{n}\frac{\lambda_{i}}{\mu_{i+1}}\rho^{2}_{i+1}
=\rho_{1}^{-1}\sum^{+\infty}_{n=0} \prod_{i=0}^{n}\frac{\lambda_{i}}{\lambda_{i+1}}
\frac{\lambda_{i+1}}{\mu_{i+1}}\rho^{2}_{i+1}
\label{pnm14}
\end{equation}
converges. From (\ref{pnm2}), considering $\rho_{n}={w}_{n}(0)$, we have
\begin{equation}
\frac{\lambda_{i}}{\mu_{i}}=\frac{1-\rho_{i}}{\rho_{i}(1-\rho_{i+1})} ,
\quad i=1,2,\ldots.
\label{pnm15}
\end{equation}

Using this identity in~(\ref{pnm14}) and telescoping gives
\begin{equation}
\begin{array}{rl}
S&\displaystyle=\frac{\lambda_0(1-\rho_1)}{\rho_1}
\sum_{n=1}^{\infty}\frac{\rho_{n,0}}{\lambda_n(1-\rho_{n+1})}\\[3mm]
&\displaystyle=\frac{\lambda_0(1-\rho_1)}{\rho_1}
\sum_{n=1}^{\infty}\frac{\rho_n\rho_{n,0}}{\mu_n(1-\rho_n)}.
\end{array}
\label{pnm16}
\end{equation}
Under (a), $\rho_{n,0}\leq a^n$ with $a=\sup_n\rho_n<1$, and the
first series is bounded by a constant times $\sum_{n\geq1}a^n$.
Under (b), the same argument applies to the second series.
Condition (c) implies convergence of the first series because
$\rho_{n,0}\leq1$. Condition (d) implies convergence of the second
because $\rho_n\rho_{n,0}\leq1$ and its initial term is finite.
Thus each condition gives $S<\infty$ and hence ergodicity. $\Box$

\begin{def1}
The BDP $\overline X$ described by the parameters
$$
\overline\lambda_{n}=\left\{\begin{array}{ll}
    \lambda_{0} & \mbox{for $n=0$} \\
    \mu_{n} & \mbox{for  $n=1,2,\ldots$}
		\end{array}
    \right.,
$$
$$
\overline\mu_{n}=\lambda_{n} , n=1,2,\ldots,
$$
is called the 0-dual BDP of the BDP $X$ described by the parameters $\{\lambda_{n}\}_{n=0}^{+\infty}$, $\{\mu_{n}\}_{n=1}^{+\infty}$.
\end{def1}

The following theorem establishes the relationship between the finite projection and the 0-dual of a transient BDP.

\begin{teor}
The finite projection $X^{*}$ of a transient BDP $X$, described by parameters $\{\lambda_{n}\}_{n=0}^{+\infty}$, $\{\mu_{n}\}_{n=1}^{+\infty}$, and the 0-dual $\overline X$ of $X$ are governed by the same probability distribution if and only if $\mu_{i}/\lambda_{i}=c < 1,$ $i=1,2,\ldots $.
\end{teor}
{\bf Proof. } Assume that the finite projection 
$X^{*}$ of the transient BDP $X$ coincides (in distribution) with the 0-dual 
$\overline X$ of $X$, that is, according to the definition of the 0-dual $\overline X$, the BDP $X^{*}$ is described by the parameters
$$
\lambda ^{*}_{n}=\left\{\begin{array}{ll}
    \lambda_{0} & \mbox{for $n=0$} \\
    \mu_{n} & \mbox{for  $n=1,2,\ldots$}
		\end{array}
    \right.
$$
$$
    \mu^{*}_{n}=\lambda_{n} ,\; n=1,2,\ldots,
$$
On the other hand, according to Theorem~\ref{p2_r1}, $X^{*}$ is 
described by the parameters $ \lambda ^{*}_{0}= \lambda_{0},
  \lambda^{*}_{n}= \lambda _{n}\rho _{n+1}, n=1,2,\ldots ,$
  $\mu^{*}_{n}=\mu _{n}/\rho_{n}, n=1,2,\ldots $. 
Thus, $\rho_{n+1}= \lambda^{*}_{n}/\lambda_{n}= \mu_{n}/\lambda_{n},\,
			n=1,2, \ldots$. Furthermore, $\rho _{n+1}=\mu _{n+1}/\mu^{*}_{n+1}=
			\mu_{n+1}/\lambda _{n+1},\; n=0,1,2,\ldots $.
This implies that
$\mu_{n}/\lambda_{n}= c, \; n=1,2,\ldots $, where $c=\rho_{i}<1,
\; i=1,2,\ldots $ since $X$ is transient, thus proving the necessity.

Conversely, assume that $\mu _{n}/\lambda _{n}= c < 1, n=1,2,\ldots$.
Then, from (\ref{pnm3}) and (\ref{pnm4}), it follows that
$\rho_{n}= c =\mu _{n}/\lambda_{n}, n=1,2,\ldots $.
Thus, $\lambda ^{*}_{n} = \lambda _{n}\rho _{n+1}=
 \lambda _{n}\rho _{n}=\lambda
_{n}\mu _{n}/\lambda _{n}= \mu _{n}, n=1,2,\ldots,$ $\lambda_{0}^{*}=
\lambda _{0},$ $\mu ^{*}_{n}= \mu _{n}/\rho _{n}= \lambda _{n}, n=1,2,\ldots $.
Hence, the sufficiency is also proven.~$\Box$

\subsection{\sloppy Finite Projections of Queueing Systems Described by Birth-and-Death Processes} \label{par_2_3}

\subsubsection{The $(M/M/1)^*$ System}

Let $\cal S$ be an $M/M/1$ queueing system with arrival rate $\lambda$ and service rate $\mu$. We assume that the system $\cal S$ is unstable, i.e., $\lambda/\mu > 1$. In this case, according to the corollary of Theorem \ref{p1_r9}, the finite projection $(M/M/1)^*$ of the $M/M/1$ system is a stable queueing system where services are independent and exponentially distributed with parameter $\lambda$, arrivals into the occupied system occur at independent intervals, exponentially distributed with parameter $\mu$, and idle periods are exponentially distributed with parameter $\lambda$. This result also follows from Theorem~\ref{p2_r1}. The number of requests $L^*(t)$ in the system at time $t$ in such a system is a birth-and-death process with birth parameters $\lambda^*_0 = \lambda$, $\lambda^*_k = \mu$ for $k \ge 1$, and death parameters $\mu^*_k = \lambda$ for $k \ge 1$. The equilibrium distribution of this process is 
$$
\pi^*_k = \frac{\lambda - \mu}{2\lambda - \mu} \left(\frac{\mu}{\lambda} \right)^{k-1}, \ k \ge 1, \quad
\pi^*_0 = \frac{\lambda - \mu}{2\lambda - \mu}.
$$

Since, in the finite projection $(M/M/1)^*$, the service rate for requests is equal to the arrival rate in the original system, and the arrival rate during the busy periods coincides with the service rate in the original system, the characteristics of the $(M/M/1)^*$ system that depend only on the behavior during busy periods, such as, for example, the lengths of busy periods starting with $k$ requests in the queue ($k \ge 1$), coincide in distribution with the respective characteristics in the dual system of the original queueing system.

Next, we will determine the state probabilities and the virtual waiting time in the projected system $(M/M/1)^*$. We assume that at the initial moment, the system is in state $m \in \{0,1,\ldots\}$, i.e., a.s. $L^*(0) = m$. Let 
$$
P_k(t) = P^*\{L(t) = k\}, \qquad p_k(s) = \int_0^\infty e^{-st} P_k(t) dt.
$$
Consider the generating function $p(z, s) = \sum_{k=0}^\infty z^k p_k(s)$. To determine $p(z, s)$, we use the method applied by N.T.J. Bailey in studying the $M/M/1$ system; see \cite{part1_m1}, \cite{part1_11}, or \cite{part1_m2}.

\begin{lema}\label{mm1_r5}
If the initial state $m = 0$, then
$$
p(z, s) = \frac{\gamma_1 \gamma_2 \mu + z \gamma_1 (\lambda - \mu)}{(1 - \gamma_1)(\lambda + \lambda \gamma_1 - \mu \gamma_1)(\gamma_2 - z)\mu},
$$
if the initial state $m = 1$, then
$$
p(z, s) = \frac{\gamma_1^2 \gamma_2 \mu + z \gamma_1 (s + \lambda - \gamma_1 \mu)}{(1 - \gamma_1)(\lambda + \lambda \gamma_1 - \mu \gamma_1)(\gamma_2 - z)\mu},
$$
if the initial state $m \ge 2$, then
$$
p(z, s) = \frac{(s + \lambda - \lambda z) \sum_{i=0}^{m-2} z^{i+1} \gamma_1^{m-i} + z^{m} \gamma_1 (s + \lambda) + \gamma_1^{m+1} \mu (\gamma_2 - z)}{(1 - \gamma_1)(\lambda + \lambda \gamma_1 - \mu \gamma_1)(\gamma_2 - z)\mu},
$$
$$
\mbox{where} \quad \gamma_1 = \gamma_1(s) = \frac{\mu + \lambda + s - \sqrt{(\mu + \lambda + s)^2 - 4\lambda\mu}}{2\mu}, \quad |\gamma_1(s)| < 1 \quad \mbox{for } Re(s) > 0,
$$
and
$$
\gamma_2 = \gamma_2(s) = \frac{\mu + \lambda + s + \sqrt{(\mu + \lambda + s)^2 - 4\lambda\mu}}{2\mu}.
$$
\end{lema}
{\bf Proof.} From the Kolmogorov system of differential equations, we have:
\begin{eqnarray}\label{mm1_r1}
& & sp_0(s)-\delta_{0,m}=-\lambda p_0(s)+\lambda p_1(s) \nonumber\\ 
& & sp_1(s)-\delta_{1,m}=\lambda p_0(s)-(\lambda+\mu)p_1(s)+\lambda p_2(s) \\ 
& & sp_k(s)-\delta_{k,m}=\mu p_{k-1}(s)-(\lambda+\mu)p_k(s)+\lambda p_{k+1}(s), \quad k \geq 2. \nonumber
\end{eqnarray}
From the system (\ref{mm1_r1}), we obtain:
$$
sp(z,s)-z^m=\mu zp(z,s)+(\lambda-\mu)zp_0(s)-(\lambda+\mu)p(z,s)
+\mu p_0(s)-\frac{\lambda}{z}p_0(s)+\frac{\lambda}{z}p(z,s).
$$
From this, we can express:
\begin{equation}\label{mm1_r2}
p(z,s)=\frac{z^{m+1}-(1-z)(\lambda +(\lambda-\mu)z)p_0(s)}
{sz-(1-z)(\lambda-\mu z)}.
\end{equation}
The function $p(z,s)$ is bounded for $|z|\leq 1$, $Re(s)>0$. However, the denominator in (\ref{mm1_r2}) vanishes for:
$$
z=\gamma_1(s)=\frac{\mu+\lambda+s-\sqrt{(\mu+\lambda+s)^2-4\lambda\mu}}{2\mu},
$$
and $|\gamma_1(s)|<1$ for $Re(s)>0$, since $\gamma_1(s)$ is the Laplace-Stieltjes transform of the time spent by the process $L(t)$ in the set $\{1,2,\ldots\}$, i.e., the Laplace-Stieltjes transform of the busy-period distribution starting with one customer in the $(M/M/1)^*$ system; see \cite{part1_10} or \cite{part1_11}. This also follows from Proposition~\ref{mm1_r3}. Consequently, the numerator in (\ref{mm1_r2}) also vanishes for $z=\gamma_1(s)$, i.e.,
$$
(\gamma_1(s))^{m+1}-(1-\gamma_1(s))(\lambda+(\lambda-\mu)\gamma_1(s)) p_0(s)=0,
$$
or equivalently:
$$
p_0(s)=\frac{\gamma_1^{m+1}(s)}{(1-\gamma_1(s))(\lambda+(\lambda-\mu)\gamma_1(s))}.
$$
Substituting this expression into (\ref{mm1_r2}) and bringing it to a common denominator, we obtain:
\begin{eqnarray}
p(z,s)&=&\frac{z^{m+1}(1-\gamma_1)(\lambda+(\lambda-\mu)\gamma_1)-
(1-z)(\lambda +(\lambda-\mu)z)\gamma_1^{m+1}}
{(1-\gamma_1)(\lambda+(\lambda-\mu)\gamma_1)(sz-(1-z)(\lambda-\mu z))} =\nonumber\\
&=&\frac{z^{m+1}\left((1-\gamma_1)(\lambda-\mu\gamma_1)-s\gamma_1\right)+z^{m+1}(\lambda\gamma_1(1-\gamma_1)+s\gamma_1)}
{(1-\gamma_1)(\lambda+(\lambda-\mu)\gamma_1)(sz-(1-z)(\lambda-\mu z))}-\nonumber\\
& &-\frac{(1-z)(\lambda -\mu z)\gamma_1^{m+1}+\gamma_1^{m+1}
\lambda z(1-z)} 
{(1-\gamma_1)(\lambda+(\lambda-\mu)\gamma_1)(sz-(1-z)(\lambda-\mu z))} =\nonumber
\end{eqnarray}
(Since $z=\gamma_1$ cancels the denominator of the fraction in (\ref{mm1_r2}), we have):
\begin{equation}
p(z,s)=\frac{
\begin{array}{l}
z^{m+1}(\lambda\gamma_1(1-\gamma_1)+s\gamma_1)
-\gamma_1^{m+1}z(s+\lambda-\lambda z)\\[1mm]
{}+\gamma_1^{m+1}(z-\gamma_1)(\gamma_2-z)\mu
\end{array}}
{(1-\gamma_1)(\lambda+(\lambda-\mu)\gamma_1)(z-\gamma_1)(\gamma_2-z)\mu},
\label{mm1_r4}
\end{equation}
where:
$$
\gamma_2(s)=\frac{\mu+\lambda+s+\sqrt{(\mu+\lambda+s)^2-4\lambda\mu}}{2\mu}.
$$
Equality (\ref{mm1_r4}) holds because $\gamma_1(s)$ and $\gamma_2(s)$ are the roots of the quadratic polynomial in the denominator of the fraction in (\ref{mm1_r2}). From (\ref{mm1_r4}), we derive the expressions for $p(z,s)$ when $m=0,1$. 

For $m\geq 2$, from (\ref{mm1_r4}), we have:
\begingroup\arraycolsep=3pt
\begin{eqnarray}
p(z,s)=\frac{\gamma_1z\left((s+\lambda)(z^m-\gamma_1^m)-
\lambda z\gamma_1(z^{m-1}-\gamma_1^{m-1})\right) + \gamma_1^{m+1}(z-\gamma_1)(\gamma_2-z)\mu}
{(1-\gamma_1)(\lambda+\lambda\gamma_1-\mu\gamma_1)(z-\gamma_1)
(\gamma_2-z)\mu}=\nonumber\\
=\frac{\gamma_1z(z-\gamma_1)\left((s+\lambda)\sum\limits_{i=0}^{m-1}
z^i\gamma_1^{m-1-i}-\lambda z\gamma_1\sum\limits_{i=0}^{m-2}
z^i\gamma_1^{m-2-i}\right)+\gamma_1^{m+1}(z-\gamma_1)(\gamma_2-z)\mu}
{(1-\gamma_1)(\lambda+\lambda\gamma_1-\mu\gamma_1)(z-\gamma_1)
(\gamma_2-z)\mu}.\nonumber
\end{eqnarray}
\endgroup
Simplifying by $z-\gamma_1$ and grouping terms, we obtain the lemma's statement for $m\geq 2$.~$\Box$

Now, we can determine the Laplace transforms $p_i(s)$, $i\geq 1$, by expanding $(\gamma_2-z)^{-1}$ into a power series and grouping terms of the same power of $z$ in the expressions for $p(z,s)$.
\begin{prop} For the Laplace transforms $p_n(s)$ of the probabilities $P_n(t)$, the following equalities hold:
$$
p_0(s)=\frac{\gamma_1^{m+1}(s)}{(1-\gamma_1(s))(\lambda+(\lambda-\mu)\gamma_1(s))}.
$$
\begin{description}
\item[a)]
If the initial state is $m=0$, then:
$$
p_n(s)=p_0(s)\left(\frac{\mu}{\lambda}\right)^{n-1}\gamma_1^{n},
\quad n\ge 1,
$$
\item[b)]
If the initial state is $m=1$, then:
$$
p_n(s)=p_0(s)\left(\frac{\mu}{\lambda}\right)^{n}\gamma_1^{n-1}\mu^{-1}(s+\lambda),
\quad n\ge 1,
$$
\item[c)]
If the initial state is $m=2$, then:
$$
p_1(s)=p_0(s)\lambda^{-1}(s+\lambda),
$$
$$
p_n(s)=p_0(s)\left(\frac{\mu}{\lambda}\right)^{n}\left(\gamma_1^{n-1}(s+\lambda)+\gamma_1^{n-2}\gamma_2(s+\lambda-\lambda\gamma_1)\right)\mu^{-1},
\quad n\ge 2,
$$
\item[d)]
If the initial state is $m\ge 3$, then:
\end{description}\vspace{-0.15 cm}
\begin{eqnarray}
p_1(s)&=&p_0(s)\frac{s+\lambda}{\lambda}, \nonumber\\
p_n(s)&=&p_0(s)\frac{\gamma_1^{-1}\gamma_2^{1-n}(s+\lambda)+(s+\lambda-\lambda\gamma_1)\sum\limits_{i=2}^n \gamma_1^{-i}\gamma_2^{i-n}}
{\mu\gamma_2}, 
\quad 2\le n\le m-1, \nonumber\\
p_n(s)&=&p_0(s)\frac{\gamma_1^{-1}\gamma_2^{1-n}(s+\lambda)+(s+\lambda-\lambda\gamma_1)\sum\limits_{i=2}^{m}\gamma_1^{-i}\gamma_2^{i-n}}
{\mu\gamma_2},
\quad n\ge m. \nonumber
\end{eqnarray}
Where $\gamma_1$ and $\gamma_2$ are defined as in Lemma \ref{mm1_r5}.
\end{prop}
{\bf Proof.} The Laplace transform $p_0(s)$ was determined in the proof of Lemma~\ref{mm1_r5}.

\noindent{\bf a)} From Lemma~\ref{mm1_r5} for $m=0$, we have:
$$
p(z,s)=\frac{\mu\gamma_1\gamma_2+z\gamma_1(\lambda-\mu)} 
{(1-\gamma_1)(\lambda+\lambda\gamma_1-\mu\gamma_1)\mu\gamma_2}
\sum_{n=0}^\infty\left(\frac{z}{\gamma_2}\right)^n.
$$
From here, after grouping terms by powers of $z$ and considering the equality $\gamma_1\gamma_2=\lambda/\mu$, we obtain:
\begin{eqnarray}
p_1(s)&=&\frac{\mu\gamma_1+\gamma_1(\lambda-\mu)}
{(1-\gamma_1)(\lambda+\lambda\gamma_1-\mu\gamma_1)\mu\gamma_2}=
\frac{\gamma_1^{2}}{(1-\gamma_1)(\lambda+\lambda\gamma_1-\mu\gamma_1)},
\nonumber\\
p_n(s)&=&\frac{\mu\gamma_1\gamma_2^{-n+1}+\gamma_1(\lambda-\mu)
\gamma_2^{-n+1}}{(1-\gamma_1)(\lambda+\lambda\gamma_1-\mu\gamma_1)\mu\gamma_2}=
\frac{\lambda\gamma_1\gamma_2^{-n}}
{(1-\gamma_1)(\lambda+\lambda\gamma_1-\mu\gamma_1)\mu}= \nonumber\\
&=&\left(\frac{\mu}{\lambda}\right)^{n-1}
\frac{\gamma_1^{n+1}}{(1-\gamma_1)(\lambda+\lambda\gamma_1-\mu\gamma_1)},
\quad \mbox{where} \ \ n\ge 2.\nonumber
\end{eqnarray}

\noindent{\bf b)} From Lemma~\ref{mm1_r5} for $m=1$, we have:
$$
p(z,s)=\frac{\mu\gamma_1^2\gamma_2+\sum\limits_{n=1}^\infty z^n
\gamma_1\gamma_2^{-n+1}(s+\lambda)}
{(1-\gamma_1)(\lambda+\lambda\gamma_1-\mu\gamma_1)\mu\gamma_2}.
$$
Therefore:
$$
p_n(s)=\frac{\gamma_1\gamma_2^{-n+1}(s+\lambda)}
{(1-\gamma_1)(\lambda+\lambda\gamma_1-\mu\gamma_1)\mu\gamma_2}=
\left(\frac{\mu}{\lambda}\right)^{n}
\frac{\gamma_1^{n+1}(s+\lambda)}{\mu(1-\gamma_1)(\lambda+\lambda
\gamma_1-\mu\gamma_1)}, \ n\ge 1.
$$  

\noindent{\bf c)} From Lemma~\ref{mm1_r5} for $m=2$, we have:
\begin{eqnarray}
p(z,s)=& &\frac{\mu\gamma_1^3\gamma_2+z\gamma_1\big((s+\lambda) 
\gamma_1-\gamma_1^2\mu\big)+z^2\gamma_1(s+\lambda-\lambda\gamma_1)} 
{(1-\gamma_1)(\lambda+\lambda\gamma_1-\mu\gamma_1)\mu\gamma_2}
\sum_{n=0}^\infty\left(\frac{z}{\gamma_2}\right)^n=\nonumber\\
=& &\frac{\mu\gamma_1^3\gamma_2+z\Big(\mu\gamma_1^3\gamma_2/
\gamma_2+\gamma_1\big((s+\lambda)\gamma_1-\gamma_1^2\mu\big)\Big)}
{(1-\gamma_1)(\lambda+\lambda\gamma_1-\mu\gamma_1)\mu\gamma_2}+\nonumber\\
&+&\frac{\sum\limits_{n=2}^\infty z^n\Big(\mu\gamma_1^3\gamma_2^{1-n}+
\gamma_1\big((s+\lambda)\gamma_1-\gamma_1^2\mu\big)\gamma_2^{1-n}+
\gamma_1(s+\lambda-\lambda\gamma_1)\gamma_2^{2-n} \Big)}
{(1-\gamma_1)(\lambda+\lambda\gamma_1-\mu\gamma_1)\mu\gamma_2}.
\nonumber
\end{eqnarray}
By performing elementary calculations and using the equality $\gamma_1\gamma_2=\lambda/\mu$, we obtain the statement of the proposition for $m=2$.

\noindent{\bf d)} From Lemma~\ref{mm1_r5} for $m\ge 3$, we have:
\begin{eqnarray}
p(z,s)=& p_0(s)&+\frac{(s+\lambda-\lambda z) 
\sum\limits_{i=0}^{m-2}z^{i+1}\gamma_1^{m-i}+z^{m}\gamma_1(s+\lambda)}
{(1-\gamma_1)(\lambda+\lambda\gamma_1-\mu\gamma_1)(\gamma_2-z)\mu}
=\nonumber
\end{eqnarray}
\begin{eqnarray}
=& p_0(s)&+\frac{z\gamma_1^m(s+\lambda)+\sum\limits_{i=2}^{m-1}z^i
\gamma_1^{m-i+1}\left(s+\lambda-\lambda\gamma_1\right)+
z^m\gamma_1\left(s+\lambda-\lambda\gamma_1\right)}
{(1-\gamma_1)(\lambda+\lambda\gamma_1-\mu\gamma_1)\mu\gamma_2}
\sum\limits_{n=0}^\infty\frac{z^n}{\gamma_2^n}
\nonumber\\
=& p_0(s)&+\frac{z\gamma_1^m(s+\lambda)+
\sum\limits_{n=1}^{\infty}z^{n+1}\gamma_1^m\gamma_2^{-n}(s+\lambda)}
{(1-\gamma_1)(\lambda+\lambda\gamma_1-\mu\gamma_1)\mu\gamma_2}+\nonumber\\
&+ &\frac{\left(s+\lambda-\lambda\gamma_1\right)
\sum\limits_{n=2}^{\infty}z^n \sum\limits_{i=2}^{\min\{m-1,n\}}
\gamma_1^{m-i+1}\gamma_2^{i-n}+\left(s+\lambda-\lambda\gamma_1\right)
\sum\limits_{n=m}^{\infty} z^{n}\gamma_1\gamma_2^{m-n}}
{(1-\gamma_1)(\lambda+\lambda\gamma_1-\mu\gamma_1)\mu\gamma_2}.
\nonumber
\end{eqnarray}
Grouping the terms around $z^n$, we obtain the expressions for the Laplace transforms $p_n(s)$:
\begin{eqnarray}
p_1(s)&=&\frac{\gamma_1^{m}(s+\lambda)}{(1-\gamma_1)(\lambda+
\lambda\gamma_1-\mu\gamma_1)\mu\gamma_2}
\nonumber\\
p_n(s)&=&\frac{\gamma_1^{m}\gamma_2^{1-n}(s+\lambda)+
(s+\lambda-\lambda\gamma_1)\sum\limits_{i=2}^n \gamma_1^{m-i+1}
\gamma_2^{i-n}}
{(1-\gamma_1)(\lambda+\lambda\gamma_1-\mu\gamma_1)\mu\gamma_2}, 
\quad \ 2\le n\le m-1, \nonumber\\
p_n(s)&=&\frac{\gamma_1^{m}\gamma_2^{1-n}(s+\lambda)+
(s+\lambda-\lambda\gamma_1)\sum\limits_{i=2}^{m}\gamma_1^{m-i+1}
\gamma_2^{i-n}}
{(1-\gamma_1)(\lambda+\lambda\gamma_1-\mu\gamma_1)\mu\gamma_2},
\quad \  n\ge m, \nonumber
\end{eqnarray}
highlighting the corresponding expressions for $p_0(s)$, we obtain the statement of the proposition. $\Box$
\bigskip

The Laplace transforms $p_n(s)$ can be inverted, resulting in probabilities $P_n(t)$ expressed in terms of modified Bessel functions of the first kind, similar to $M/M/1$ systems. To this end, $p_0(s)$ is represented as a sum of simple fractions. These fractions are then written as power series in $\gamma_2^{-1}$. The obtained expressions for $p_n(s)$, after substituting $\gamma_1$ with $\lambda/(\mu\gamma_2)$, can be inverted easily. We will illustrate this procedure for the case $m=0$, as the resulting formulas are simpler in this case.

\begin{prop} 
If the initial state of the process $L(t)$ is zero, $L(0)=0$ a.s., then, for $t>0$,
\begin{eqnarray}
P_0(t)&=&\frac{\exp\left(-(\lambda+\mu)t\right)}{(2\lambda-\mu)t}
\sum_{k=1}^\infty\frac{\lambda^{k/2}-(\mu-\lambda)^k\lambda^{-k/2}}
{\mu^{k/2}}kI_k\left(2t\sqrt{\lambda\mu}\right)\nonumber\\
P_n(t)&=&\frac{\lambda\exp\left(-(\lambda+\mu)t\right)}{\mu(2\lambda-\mu)t}
\sum_{k=n+1}^\infty\frac{\lambda^{k/2-n}-(\mu-\lambda)^{k-n}\lambda^{-k/2}}
{\mu^{k/2-n}}kI_k\left(2t\sqrt{\lambda\mu}\right),\quad n\ge 1,\nonumber
\end{eqnarray}
At $t=0$, the probabilities are $P_0(0)=1$ and $P_n(0)=0$ for $n\geq1$, obtained by continuity. Here $I_k(z)$ is the modified Bessel function of the first kind:
$$
I_k(z)=\sum_{m=0}^\infty\frac{(z/2)^{2m+k}}{m!\Gamma(m+k+1)}.
$$
\end{prop}
{\bf Proof.} 
It is easy to verify that 
$$
p_0(s)=\frac{\gamma_1}{(1-\gamma_1)(\lambda+\lambda\gamma_1-\mu\gamma_1)}
=\frac{\gamma_1}{2\lambda-\mu}\left(\frac{1}{1-\gamma_1}+\frac{\lambda-
\mu}{\lambda+\lambda\gamma_1-\mu\gamma_1}\right).
$$
Thus,
$$
p_0(s)=\frac{\gamma_1}{2\lambda-\mu}\sum_{k=0}^\infty\gamma_1^k
\left(1-\left(\frac{\mu-\lambda}{\lambda}\right)^{k+1}\right).
$$
Since $\gamma_1=\lambda/(\mu\gamma_2)$, we have 
\begin{equation}\label{mm1_r6}
p_0(s)=\frac{1}{2\lambda-\mu}\sum_{k=1}^\infty
\frac{\lambda^k-(\mu-\lambda)^k}{\mu^k}\gamma_2^{-k}.
\end{equation}
From the previous proposition and the equality \
$\lambda/\mu=\gamma_1\gamma_2$, using the equality
(\ref{mm1_r6}), we have
$$
p_n(s)=p_0(s)\lambda\gamma_2^{-n}\mu^{-1}=
\frac{\lambda}{\mu(2\lambda-\mu)}\sum_{k=1}^\infty
\frac{\lambda^k-(\mu-\lambda)^k}{\mu^k}\gamma_2^{-(k+n)}.
$$
For $\mathop{\rm Re}s>0$, $\gamma_2(s)^{-k}$ is the Laplace transform of
$$
\frac{k}{t}(\mu/\lambda)^{k/2}e^{-(\lambda+\mu)t}
I_k\left(2t\sqrt{\lambda\mu}\right).
$$
Using this identity in the last equality and in (\ref{mm1_r6}) gives the proposition's assertion. $\Box$
\bigskip

Let us note that in the case $m\ge 1$, when inverting the Laplace transforms $p_n(s)$, in addition to modified Bessel functions of the first kind, their derivatives also appear, which themselves are expressed through Bessel functions of higher orders. This leads to more complex formulas for $P_n(t)$ in this case.

To determine the distribution function $W(x,t)$ of the virtual waiting time
$V(t)$, i.e., the waiting time duration for a request arriving at the system at time $t$ until its service begins, we assume first-come first-served service and observe that 
$$
V(t)=\left\{\begin{array}{ll}
    {0} & \mbox{if $L(t)=0$,} \\
    \sum_{i=1}^k \xi_i & \mbox{if $L(t)=k$,} 
\end{array}
\right.
$$
where $\xi_1$ is the remaining service time of the customer in service at time $t$, and $\xi_2,\ldots,\xi_k$ are the service times of the customers waiting ahead.
Thus, 
\begin{equation}\label{mm1_r7}
W(x,t)=P^*\{V(t)\le x\}=P_0(t)+\sum_{k=1}^\infty P_k(t)P^*\left\{
\sum_{i=1}^k \xi_i\le x\right\},
\end{equation}
where the random variables $\xi_i$, \ $i=1,\ldots,k$, are i.i.d. with the distribution function
$1-\exp(-\lambda x)$, for $x\ge 0$.

Let \ \ $w(x,s)=\int_0^\infty e^{-st}W(x,t)\,dt$. \ Then,
from (\ref{mm1_r7}), we have
$$
w(x,s)=p_0(s)+\sum_{k=1}^\infty p_k(s)P^*\left\{
\sum_{i=1}^k \xi_i\le x\right\}.
$$
For $m=0$,
$$
P^*\left\{\sum_{i=1}^k \xi_i\le x\right\}=\int_0^{\lambda x}\frac{u^{k-1}}
{(k-1)!} e^{-u}\,du, \qquad p_n(s)=p_0(s)\lambda\gamma_2^{-n}\mu^{-1}.
$$
Thus,
\begin{eqnarray}
w(x,s)&=& p_0(s)\left(1+\frac{\lambda}{\mu}\sum_{k=1}^\infty
\gamma_2^{-k}\int_0^{\lambda x}\frac{u^{k-1}}{(k-1)!} e^{-u}\,du
\right)=\nonumber\\
&=&p_0(s)\left(1+\frac{\lambda}{\mu(\gamma_2-1)}
\left[1-\exp\left\{{-\lambda
x\left(\frac{\gamma_2-1}{\gamma_2}\right)}\right\}\right]
\right).\nonumber 
\end{eqnarray}

Since the $(M/M/1)^*$ system is stable, from (\ref{mm1_r7}), we see that 
$\lim\limits_{t\to\infty} W(x,t)=\overline{W}(x)$
and 
$$
\overline{W}(x)=\pi_0^*+\sum_{k=1}^\infty \pi_0^*\left(\frac{\mu}{\lambda}
\right)^{k-1}\int_0^{\lambda x}\frac{u^{k-1}}{(k-1)!}
e^{-u}\,du. 
$$
Performing the calculations, we obtain $\overline{W}(x)=1-\frac{\lambda}{2\lambda-\mu}\exp[-(\lambda-\mu)x]$, $x\geq0$. Note that the distribution
$\overline{W}(x)$ of the virtual waiting time does not depend
on the initial distribution of the process $L^*(t)$.

\subsubsection{The $(M/M/n)^*$ System}

Let us consider a queuing system of type $M/M/n$ with arrival rate $\lambda$ and service rate $\mu$ for each of the $n$ servers in the system. We assume that the system is unstable, i.e., $\lambda/(n\mu) > 1$. The number of requests $L(t)$ in this system follows a birth-and-death process with birth rates $\lambda_i = \lambda$, $\ (\forall)\ i\ge 0$, and death rates
$$
\mu_i=\left \{ \begin{array}{ll}
i\mu, & \mbox{if} \ 1\le i\le n,\\
n\mu, & \mbox{if} \ i\ge n.
\end{array} \right.
$$
According to Theorem \ref{p2_r1}, the finite projection $L^*(t)$ of the Markov process $L(t)$ is also a Markov process, characterized by the rates $\lambda_0^*=\lambda$, $\lambda_i^*=\lambda\rho_{i+1}$ for $i\geq1$, and $\mu_i^*=\mu_i/\rho_i$ for $i\geq1$. From equality (\ref{pnm2}), it follows that $\rho_k = n\mu/\lambda$ for $k \ge n$, since in this case $\rho_k = \rho_{k+1}$. Otherwise, for $1 \le k < n$, the probabilities $\rho_k$ can be computed using the recurrence relation (\ref{pnm2}). To explicitly determine the probabilities $\rho_k$, the representation (\ref{pnm3}) and known results for continued fractions can be used (see \cite{l_pnm3}). These probabilities can also be more easily obtained from the difference equations examined by Karlin S. and McGregor J. (see, for example, \cite{l_frc1}). According to these results,
$$
\rho_k = \eta_k(1+\eta_k)^{-1}, \quad \mbox{where} \quad \eta_k=\sum_{i=k}^\infty\prod_{j=k}^i\frac{\mu_j}{\lambda_j}.
$$
Thus, for the $M/M/n$ system, when $k < n$, we have
$$
\eta_k = \sum_{i=k}^{n-1}\prod_{j=k}^i\frac{j\mu}{\lambda} + \sum_{i=n}^\infty\prod_{j=k}^n\frac{j\mu}{\lambda}\prod_{j=n+1}^i\frac{n\mu}{\lambda} = \sum_{i=k}^{n-1}\frac{i!}{(k-1)!}\left(\frac{\mu}{\lambda}\right)^{i-k+1} + \frac{n!}{(k-1)!}\frac{\left(\mu/\lambda\right)^{n-k+1}}{1-n\mu/\lambda}.
$$
From which, for $k < n$,
\begin{equation}\label{mmn_r2}
\rho_k = \eta_k(1+\eta_k)^{-1}, \ \ \mbox{where} \ \ 
\eta_k = \sum_{i=k}^{n-1}\frac{i!}{(k-1)!}\left(\frac{\mu}{\lambda}\right)^{i-k+1} + \frac{n!}{(k-1)!}\frac{\left(\mu/\lambda\right)^{n-k+1}}{1-n\mu/\lambda}.
\end{equation}
Thus, with $q=n\mu/\lambda<1$, the projected process $L^*(t)$
has the rates
\begin{equation}\label{queue_projected_rates}
\lambda_0^*=\lambda,\qquad
\lambda_k^*=\lambda\rho_{k+1}\ (k\geq1),\qquad
\mu_k^*=\frac{\min(k,n)\mu}{\rho_k}\ (k\geq1),
\end{equation}
where $\rho_k=q$ for $k\geq n$ and
$\rho_k=k\mu/(\lambda+k\mu-\lambda\rho_{k+1})$ for $1\leq k<n$.
In particular, $\lambda_k^*=n\mu$ for $k\geq n-1$ only when
$k\geq1$; the boundary rate $\lambda_0^*=\lambda$ applies also
when $n=1$.

We observe that from the equality $\lambda^*_k+\mu^*_k = \lambda_k+\mu_k$, $k\ge 1$, it follows that the distribution of the time the system spends in state $k$ is the same in both the initial system and its finite projection and is exponential with parameter $\lambda_k+\mu_k$ for $k\ge 1$ and $\lambda$ for $k=0$. This equality, used together with the equality $\lambda_k^*\mu_{k+1}^* = \lambda_k\mu_{k+1}$, $1 \le k$, simplifies the calculation of the transition rates $\lambda_k^*$ and $\mu_{k+1}^*$ for $k\le n-1$.

In the finite projection $(M/M/n)^*$, the arrival rate depends on the number of customers in the system: it equals $n\mu$ for $k\geq\max(1,n-1)$ and $\lambda$ at $k=0$. The service rate of each busy server also depends on the number of customers in the system. Since the servers are identical, the service rate per busy server is $\mu^*_k/k$ for $1\leq k<n$, and $\lambda/n$ for $k\geq n$.

From Proposition \ref{mmn_r7}, the projection is positive recurrent. For its stationary weights (see \cite{part1_1}), write $\omega_0=1$ and
\begin{equation}\label{queue_weights}
\omega_k=\frac{\lambda^k}{\mu^k k!}\rho_1
\prod_{j=2}^k\rho_j^2\quad(1\leq k\leq n),\qquad
\omega_k=\omega_n q^{k-n}\quad(k\geq n).
\end{equation}
The empty product equals one. The balance equation
$\omega_k/\omega_{k-1}=\lambda_{k-1}^*/\mu_k^*$ gives
$\omega_1=\lambda\rho_1/\mu$: the first numerator is
$\lambda_0^*=\lambda$, rather than $\lambda\rho_1$. Subsequent
ratios contribute $\rho_k^2$. Hence
\begin{equation}\label{queue_stationary}
Z=\sum_{k=0}^{n-1}\omega_k+\frac{\omega_n}{1-q},
\qquad \pi_k^*=\frac{\omega_k}{Z}.
\end{equation}

There are two waiting quantities here. First, let $W$ be the
time until the stationary process enters $\{0,\ldots,n-1\}$,
while arrivals continue, with $W=0$ if the initial state is below $n$. Write $W(x)=P(W\leq x)$. Put $c=n\mu$,
$A=\sum_{k=0}^{n-1}\pi_k^*$, and
\begin{equation}\label{mmn_r5}
\theta(s)=\frac{\lambda+c+s-
\sqrt{(\lambda+c+s)^2-4\lambda c}}{2c}
=\frac{2\lambda}{\lambda+c+s+
\sqrt{(\lambda+c+s)^2-4\lambda c}}.
\end{equation}
The nonnegative real square root is used for real $s\geq0$;
on $\mathop{\rm Re}s>0$ use the analytic branch selected by
$\theta(s)\to0$ as positive real $s\to\infty$. This is the
transform of one downward passage while the queue remains at or
above $n$, and $\theta(0)=1$. Therefore
\begin{equation}\label{mmn_r6}
E[e^{-sW}]=A+\sum_{j\geq1}\pi_n^*q^{j-1}\theta(s)^j
=A+\frac{\pi_n^*\theta(s)}{1-q\theta(s)}.
\end{equation}
At $s=0$ its value is $A+\pi_n^*/(1-q)=1$. For $x>0$,
\begin{equation}\label{queue_idle_density}
W'(x)=\frac{\pi_n^*}{q}\frac{e^{-(\lambda+c)x}}{x}
\sum_{j\geq1}j q^{j/2}I_j(2x\sqrt{\lambda c}),
\qquad P(W=0)=A.
\end{equation}
Indeed, $\theta(s)^j$ has density
$\frac jx(\lambda/c)^{j/2}e^{-(\lambda+c)x}
I_j(2x\sqrt{\lambda c})$; multiplying by the stationary
weights and summing nonnegative terms proves the formula and its
mass $1-A$. Here $I_j$ is the modified Bessel function defined above.

Second, under first-come first-served service, the stationary
virtual waiting time $V$ for a customer is zero in states below $n$.
When $k\geq n$ customers are ahead, $k-n+1$ completions are
needed; the aggregate rate throughout that wait is $\lambda$,
and later arrivals join behind the customer. Thus
\begin{equation}\label{queue_virtual_wait}
P(V=0)=A,\quad E[e^{-sV}]=A+
\frac{\pi_n^*\lambda}{s+\lambda-c},\quad
P(V\leq x)=1-\frac{\pi_n^*}{1-q}e^{-(\lambda-c)x},\quad x\geq0.
\end{equation}
For $n=1$ this agrees with the formula above.
The density (\ref{queue_idle_density}) concerns $W$, because future
arrivals can delay an idle server, whereas those arrivals do not pass
the tagged customer in the FCFS queue.

This concludes the section dedicated to the study of finite projections of Markov processes.

\bigskip

\centerline{\large \bf Conclusions and Observations}

The finite projection procedure does not alter certain local characteristics of the initial process. In particular, the time spent in a given state before first leaving it is the same in both the original Markov process and its finite projection. This is a special case of the semi-Markov invariance property discussed in Sections \ref{invar_pr} and \ref{semi_proj}: here the holding time is independent of the next state.

As shown in Section 2.1, the two passage-time distributions in the complementary sets $\{0,\ldots,n\}$ and $\{n+1,n+2,\ldots\}$, starting from $n$ and $n+1$ respectively, together with the initial distribution, uniquely determine the law of the minimal BDP considered here.

From the analysis of finite projections of $M/M/n$ queuing systems conducted in this section, we conclude that with some modifications, the methods developed for the initial $M/M/n$ systems are applicable to the study of the projected $(M/M/n)^*$ systems. For multi-server systems, the projected rates lead to the normalized stationary weights (\ref{queue_stationary}); the stationary time to an idle server and the FCFS virtual waiting time have distinct transforms (\ref{mmn_r6}) and (\ref{queue_virtual_wait}).
  
\newpage  
%%%%%%%%%%%%%%%%%%%%%%%%%%%%%%%%%%%%%%%%%%%%%%%%%%%%%%%%%%%%%%%%%%%%%%%%
%%%%%%%%%%%%%%%%%%%%%%%%%%%%%%%%%%%%%%%%%%%%%%%%%%%%%%%%%%%%%%%%%%%%%%%%%%%%%%
 \section{Finite Projections of Markov Chains} 
\label{par_3_} 
\setcounter{equation}{0}

Let $X=\{X_n\}_{n\ge 0}$ be a Markov chain taking values in the set $E=\{0,1,2,\ldots\}$, with one-step transition probabilities $(p_{i,j},\ i,j\in E)$ and initial distribution $\gamma=(\gamma(i),\ i\in E)$. Assume that the chain $X$ is defined on the measurable  space $(\Omega,{\cal F})$, where 
$$
\Omega =E^{\times \infty}, \qquad X_n(\omega)=\omega_n, \ \
(\forall)\ \omega=(\omega_0,\omega_1,\ldots,\omega_i,\ldots) \in
\Omega, \ n=0,1,\ldots \, 
$$
and the $\sigma$-algebra $\cal F$ is generated by the sets $[a_0,a_1,\ldots,a_n]=\{\omega:X_0=a_0, X_1=a_1,\ldots,
X_n=a_n\}$, where $a_0,a_1,\ldots,a_n\in E$, $n\ge 0$. Denote ${\cal F}_n =\sigma (X_0,\ldots,X_n)$, the $\sigma$-algebra generated by the random variables $X_0,\ldots,X_n$, for $n\ge 0$. It is clear that ${\cal F}=\sigma (X_0,X_1,\ldots)$.

If the transition matrix on $E$ is substochastic, we adjoin an
absorbing state outside $E$ to account for the missing row mass
and use the corresponding enlarged canonical path space.

In this chapter, we assume that the Markov probability $P$ defined on $(\Omega,{\cal F})$ by the one-step transition probabilities $(p_{i,j},\ i,j\in E)$ and the initial distribution $\gamma=(\gamma(i),\ i\in E)$ corresponds to a transient Markov chain, meaning that all recurrent states of $X$ are absorbing (the concept of transient Markov chains in this form is used in \cite{lm_kem}). We use the following standard notations \cite{c3}, \cite{lm_1}:
\begin{eqnarray}
  p_{i,j}^{(n)}&=& P\{X_n=j|X_{0}=i\}, p_{i,j}^{(0)}=
  \delta(i,j)\ (\delta(i,j)\ \mbox{is the Kronecker symbol});\nonumber\\ 
  f_{i,j}^{(n)}&=&P\{X_n=j,X_m\not =j,\
  m=\overline{1,n-1},\ |X_{0}=i\},\ f_{i,j}^{(0)}=0;\nonumber\\
  {}_{k}p_{i,j}^{(n)}&=& P\{X_n=j,X_m\not =k,\
  m=\overline{1,n-1},\ |X_{0}=i\},\
  ({}_{k}p_{i,j}^{(0)}=\delta(i,j)\ \mbox{for}\ k\not =j);\nonumber\\ 
  {}_{C}p_{i,j}^{(n)}&=& P\{X_n=j,X_m\not\in C,\
  m=\overline{1,n-1},\ |X_{0}=i\},\
  ({}_{C}p_{i,j}^{(0)}=\delta(i,j)\ \mbox{for}\ j\not\in C);\nonumber\\ 
  f_{i,j}&=&\sum_{n=1}^\infty f_{i,j}^{(n)},\
  {}_{k}p_{i,j}=\sum_{n=1}^\infty\ {}_{k}p_{i,j}^{(n)},\nonumber\\
  m_{i,j}&=&\sum_{n=1}^\infty n f_{i,j}^{(n)},\ \qquad
  k,i,j\in E.  \nonumber  
\end{eqnarray}
Note that $f_{i,j}^{(n)}=\ _{j}p_{i,j}^{(n)}$. If $f_{i,j}>0$, then $m_{i,j}/f_{i,j}$ is the mean of the conditional distribution $(f_{i,j}^{(n)}/f_{i,j},\ n\ge 1)$, $i,j\in E$ of the first entrance time. To avoid trivial cases, we assume that the Markov probability $P$ is such that the state set $E$ consists of $P$-possible states, i.e., 
\begin{equation}\label{etalon}
\sum_{i\in E} \gamma(i)f_{i,j}>0 \qquad (\forall ) j\in E.
\end{equation}

The goal now is to describe the finite projections of Markov chains. For convenience, we redefine the concept of finite projection.

\subsection{Definition}
Let $S_0$ be the first entrance time into state $0$, allowing time
zero, and let $S_j$, $j\geq1$, be the successive return times:
\begin{equation}\label{markov_initial_hit}
S_0=\inf\{n\geq0:X_n=0\},\qquad
S_{j+1}=\inf\{n>S_j:X_n=0\},\qquad \inf\emptyset=\infty.
\end{equation}
We assume $0<f_{0,0}<1$. The possible-state condition
(\ref{etalon}) ensures $P\{S_0<\infty\}>0$.
The sequence $S=\{S_j\}$ is a renewal process, ordinary if
$X_0=0$ and delayed otherwise. For $i\ne0$ with $f_{i,0}>0$,
$$
P\{S_0=n\mid X_0=i\}=f_{i,0}^{(n)},\qquad
{\bf E}(S_0\mid X_0=i,S_0<\infty)=m_{i,0}/f_{i,0}.
$$
If $X_0=0$, then $S_0=0$ and $S_1$ is the first positive return.

The following proposition describes the finite projection of $P$
with respect to the renewal process $S$.
\begin{prop}\label{lm_r3}
Let $P_n(\cdot)=P(\cdot\mid S_n<\infty)$, $n\geq1$. Then:
\begin{description}
\item[a)] There is a unique probability $P^*$ on $(\Omega,{\cal F})$
such that, for every $k\geq0$ and $A\in{\cal F}_k$,
$$
\lim_{n\to\infty}P_n(A)=P^*(A).
$$
\item[b)] Under $P^*$, $X$ is a Markov chain with minimal state
space of possible states $E^*=\{i\in E:f_{i,0}>0\}$,
transition probabilities and initial distribution
\begin{equation}\label{markov_projected_rates}
\begin{array}{l}
p_{i,j}^*=\displaystyle\frac{p_{i,j}f_{j,0}^{1-\delta(0,j)}}{f_{i,0}},
\qquad i,j\in E^*,\\[3mm]
\gamma^*(i)=\displaystyle\frac{\gamma(i)f_{i,0}^{1-\delta(0,i)}}
{\gamma(0)+\sum_{j\ne0}\gamma(j)f_{j,0}},\qquad i\in E^*.
\end{array}
\end{equation}
\item[c)] State $0$ is recurrent under $P^*$ and is positive
recurrent if and only if $m_{0,0}<\infty$. Moreover,
$$
f_{i,0}^{*(n)}=f_{i,0}^{(n)}/f_{i,0},\qquad i\in E^*,\quad n\geq1.
$$
For $i=0$, this is the distribution of the first positive return.
\end{description}
\end{prop}

As before, quantities associated with $P^*$ are marked by an
asterisk, for example $p_{i,j}^*$, $m_{i,j}^*$ and $X^*$.

{\bf Proof.} The assumptions imply $P\{S_n<\infty\}>0$ for
every $n$, so $P_n$ is well defined. For a cylinder
$C=[i_0,\ldots,i_k]$, let $N(i_0,\ldots,i_k)$ count the zeros
at indices $0,\ldots,k-1$. For $n>k$, the Markov property gives
$$
P\{S_0<\infty\}=\gamma(0)+\sum_{j\ne0}\gamma(j)f_{j,0},\qquad
P\{S_n<\infty\mid X_0=i\}=f_{i,0}^{1-\delta(i,0)}f_{0,0}^n,
$$
and hence
$$
P_n(C)=\frac{P(C)f_{i_k,0}^{1-\delta(i_k,0)}
f_{0,0}^{\,n-N(i_0,\ldots,i_k)}}
{\left(\gamma(0)+\sum_{j\ne0}\gamma(j)f_{j,0}\right)f_{0,0}^n}.
$$
For $i_0,\ldots,i_k\in E^*$, multiplying and dividing by the
intermediate entrance probabilities and regrouping yields
\begin{equation}\label{lm_r2}\label{markov_cylinder_early}
P_n(C)=\gamma^*(i_0)p_{i_0,i_1}^*\cdots p_{i_{k-1},i_k}^*.
\end{equation}
A cylinder containing a state outside $E^*$ has $P_n$-probability
zero when $n>k$. Every event in ${\cal F}_k$ is a countable union
of these cylinders, so $P_n$ is constant on ${\cal F}_k$ for $n>k$.

The initial distribution $\gamma^*$ sums to one. The first-step
identity
$$
f_{i,0}=p_{i,0}+\sum_{j\ne0}p_{i,j}f_{j,0}
$$
shows that every row of $(p_{i,j}^*)$ sums to one, including the
row $i=0$. Thus the consistent cylinder probabilities define a
unique Markov probability $P^*$ on ${\cal F}$, proving a) and b).
The chain remains in $E^*$ almost surely. By (\ref{etalon}),
every state of $E^*$ can be reached with positive probability
under $P^*$. Consequently, for every proper subset
$\widetilde E\subset E^*$,
$$
P^*\{\exists n\geq0:X_n\in E^*\setminus\widetilde E\}>0,
$$
so $E^*$ is the minimal state space of possible states.

To prove c), sum the products of projected transition probabilities
along paths avoiding zero until their last step. The factors in
(\ref{markov_projected_rates}) telescope, giving
$f_{i,0}^{*(n)}=f_{i,0}^{(n)}/f_{i,0}$. For $i\ne0$ this is
$P^*\{S_0=n\mid X_0=i\}$; for $i=0$ it is
$P^*\{S_1=n\mid X_0=0\}$. Their sums equal one, so zero is
recurrent. Finally,
$$
m_{0,0}^*=m_{0,0}/f_{0,0},
$$
which is finite exactly when $m_{0,0}<\infty$. $\Box$

\begin{def1}
The Markov probability $P^*$ is called the finite projection of
$P$ with respect to state $0$.
\end{def1}
This is the finite projection with respect to the renewal process
$S$, in the sense defined for regenerative processes in
Section~\ref{par_1_2}.

The procedure of taking the finite projection induces a partition of the state space of the Markov chain $X$: $\Pi =\{E^*_0, E^*_{tr}, E\setminus E^*\}$, where:
\begin{itemize}
\item $E\setminus E^*$ is the maximal set of states that $P^*$-a.s. are not visited by the trajectories of the Markov chain $X$, i.e., $E\setminus E^*=\{i\in E:P^*\{\exists n\ge 0: \ X_n=i\}=0\}$,
\item $E_0^*$ is the set of states in $E^*$ that can be reached from state zero (under either $P$ or $P^*$), i.e., $E^*_0=\{i\in E^*:\ f_{0,i}>0\}$. $E^*_0$ is the unique class of recurrent states in the finite projection $P^*$, i.e., $f^*_{i,i}=1 \ \Leftrightarrow \ i\in E^*_0$.
\item $E^*_{tr}=E^*\setminus E^*_0$ --- the set of all transient states in the finite projection $P^*$, i.e., $E^*_{tr}=E^*\setminus E^*_0=\left \{i\in E: \ \sum\limits_{j\in E}\gamma(j)f_{j,i} f_{i,0}>0, \ f_{0,i}=0\right \}$.
\end{itemize}

Thus, from the results established, we conclude that in the finite projection $P^*$:
\begin{description}
\item[a)] All absorbing states of the chain $X$ belong to the element $E\setminus E^*$ of the partition $\Pi$.
\item[b)] If the initial distribution $\gamma$ is concentrated on $E_0^*$, all possible states of $X^*$ are recurrent, $E^*=E^*_0$, and the chain $X^*$ is irreducible.
\item[c)] The chain $X^*$ does not disappear, i.e., $P^*\{X_\cdot\in E^*\}=1$.
\end{description}

Hence, taking the finite projection is a method of transforming a Markov chain $X$, possibly with a substochastic transition matrix, absorbing, and transient states, into a Markov chain $X^*$ with a stochastic transition probability matrix where all states are recurrent (if $E^*=E^*_0$, i.e., $\gamma^*\{E^*_0\}=1)$ and form a single class.

The problem of determining the characteristics of the initial Markov chain $X$ that are invariant under finite projection is considered one of the fundamental problems in the study of finite projections of Markov chains. This problem will be partially solved in this chapter.

\subsection{Relationships between the characteristics of the probabilities $P$ and $P^*$}

Our aim is to express the main characteristics of the finite projection $P^*$ in terms of the characteristics of the initial probability $P$.

The following proposition describes the invariant measures of the projected chain $X^*$, i.e., the nonnegative solutions to the system of equations $x_j=\sum_i x_i p_{i,j}^*,\ \ i,j\in E^*$.

\begin{prop}\label{lm_r4}
\begin{description}
\item[a)] 
If the Markov chain $X^*$ is irreducible and aperiodic, then there exists a unique (up to a nonnegative constant $c$) invariant measure $u^*$ of the form:
$$
u_j^*=c {\,}_0p_{0,j}^{1-\delta(j,0)} f_{j,0},\ \ j\in E^*.
$$
\item[b)]
If the Markov chain $X^*$ is ergodic, then its stationary distribution is:
$$
\pi_j^*={}_0p_{0,j}^{1-\delta(j,0)} f_{j,0}/m_{0,0},\ \ j\in E^*.
$$
\end{description}
\end{prop}

{\bf Proof.}
Part a) follows from Derman's theorem on the invariant measures of an irreducible and recurrent Markov chain, which states that $u_j^*=c {}_{0} p_{0,j}^{*}$, see \cite{c3} or \cite{lm_1}. Indeed, $X^*$ is recurrent (Proposition~\ref{lm_r3}), aperiodic, and ${}_{0}p_{0,j}^{*}=
{}_{0}p_{0,j}^{1-\delta(j,0)} f_{j,0}/f_{0,0},\ \ j\in E^*$.

Part b). It is known that $\pi_j^*=1/m_{j,j}^*,\ \ j\in E^*$. Moreover, from the recurrence of $X^*$, we have $m_{0,0}^*/m_{j,j}^*={}_0p_{0,j}^*$, see \cite{lm_1}. Using Proposition \ref{lm_r3}, the proof of part b) is complete.~$\Box$

The next proposition is an adaptation of a theorem by Chung.
\begin{prop}
If the Markov chain $X^*$ is ergodic, then:
$$
\sum_{n=1}^\infty \left (p_{i,k}^{*(n)}-p_{j,k}^{*(n)}
\right )=
$$
$$
= \ \left\{\begin{array}{ll}
(m_{j,0}/f_{j,0}-m_{i,0}/f_{i,0}) \pi_0^*,&\ \ \mbox{if}\ k=0, \\ 
({}_{0}p_{i,k} f_{k,0} - m_{i,0}\pi_k^*)/f_{i,0} - 
({}_{0}p_{j,k} f_{k,0} - m_{j,0}\pi_k^*)/f_{j,0},
& \  \mbox{if}\ k\not =0.
\end{array} \right.
$$
\end{prop}
 {\bf Proof.}
In the case when $k=0$, the statement follows from Propositions~\ref{lm_r3} and \ref{lm_r4} b), according to which the equality
$$
S=\sum_{n=1}^\infty \left (p_{i,k}^{*(n)}-p_{j,k}^{*(n)}\right )= (m_{j,k}^*-m_{i,k}^*)/m_{k,k}^*,
$$
see \cite{lm_1}, transforms into
\begin{equation}\label{lm_r5}
S=(m_{j,0}/f_{j,0}-m_{i,0}/f_{i,0}) \pi_0^*.
\end{equation}

To prove the proposition for $k\not =0$, we use the decomposition
$$
p_{i,k}^{*(n)}={}_{0}p_{i,k}^{*(n)}+ \sum_{\nu=1}^{n-1}
p_{i,0}^{*(\nu)} {}_{0}p_{0,k}^{*(n-\nu)}.
$$
Form the difference
$$
p_{i,k}^{*(n)}-p_{j,k}^{*(n)}={}_{0}p_{i,k}^{*(n)}-
{}_{0}p_{j,k}^{*(n)} + \sum_{\nu=1}^{n-1}
{}_{0}p_{0,k}^{*(n-\nu)}(p_{i,0}^{*(\nu)} - p_{j,0}^{*(\nu)}).
$$
Summing over $n$, we obtain
$$
S={}_{0}p_{i,k}^{*}- {}_{0}p_{j,k}^{*} +
{}_{0}p_{0,k}^{*}\sum_{n=1}^{\infty} (p_{i,0}^{*(n)}
-p_{j,0}^{*(n)}). 
$$
Now, from equality (\ref{lm_r5}) and Proposition~\ref{lm_r3}, it follows that
$$
S={}_{0}p_{i,k}f_{k,0}/f_{i,0}- {}_{0}p_{j,k}f_{k,0}/f_{j,0}
+  {}_{0}p_{0,k}f_{k,0}(m_{j,0}/f_{j,0}-m_{i,0}/f_{i,0})/
m_{0,0}.
$$
Since $\pi_k^*= {}_0p_{0,k}^{1-\delta(k,0)}
f_{k,0}/m_{0,0}$, regrouping terms yields
$$
S=({}_{0}p_{i,k} f_{k,0} - m_{i,0}\pi_k^*)/f_{i,0} -
({}_{0}p_{j,k} f_{k,0} - m_{j,0}\pi_k^*)/f_{j,0}.\ \qquad
\Box
$$
\bigskip

Let us examine the generating functions associated with the Markov probability $P$:
$$
\overline{f}_{i,j}(z)=\sum_{n=1}^\infty f_{i,j}^{(n)}z^n,\ \
\overline{p}_{i,j}(z)=\sum_{n=0}^\infty p_{i,j}^{(n)}z^n,\ \
{}_{k}\overline{p}_{i,j}(z)=\sum_{n=0}^\infty
{}_{k}p_{i,j}^{(n)}z^n,
$$
and denote by $\overline{f}_{i,j}^{\, *},\ \overline{p}_{i,j}^{\, *},\
{}_{k}\overline{p}_{i,j}^{\, *}$
the same generating functions but associated with the Markov probability $P^*$.

\begin{prop}
For $|z|<1$, the following relations hold:
\begin{description}
\item[a)]
$\overline{p}_{i,0}^{\, *}=\overline{p}_{i,0}f_{0,0}f_{i,0}^{
\delta(i,0)-1}\left (1-\overline{p}_{0,0}(1-f_{0,0})\right )^{-1},
\ i\in E^*$;
\item[b)]
$\overline{p}_{0,j}^{\, *}={}_{0}\overline{p}_{0,j}f_{j,0}
\left (f_{0,0}-\overline{f}_{0,0}\right )^{-1},
\ \ j\not =0$;
\item[c)]
$\overline{p}_{i,j}^{\, *}=
f_{j,0}/f_{i,0}\left ({}_{0}\overline{p}_{i,j}+
\overline{f}_{i,0}{}_{0}\overline{p}_{0,j}(f_{0,0}-
\overline{f}_{0,0})^{-1}\right ),\ \ i,j\in E^*,\ i,j\not =0.$
\end{description}
\end{prop}

{\bf Proof.}
Since
$\overline{p}_{i,j}-\delta(i,j)=\overline{f}_{i,j}\overline{p}_{j,j}$,
see \cite{lm_1}, and
$\overline{f}_{i,0}=\overline{f}_{i,0}^{\, *}f_{i,0}$, we have:
for $i=0$,
\begin{eqnarray}
(\overline{p}_{0,0}-1)/\overline{p}_{0,0}&=&
(\overline{p}_{0,0}^{\, *}-1)f_{0,0}/\overline{p}_{0,0}^{\,
*},\nonumber \\ 
\overline{p}_{0,0}^{\, *}&=&\overline{p}_{0,0}f_{0,0}
\left (1-\overline{p}_{0,0}(1-f_{0,0})\right )^{-1},\nonumber
\end{eqnarray}
for $i\not = 0$,
$$
\overline{p}_{i,0}^{\, *}=\overline{f}_{i,0}/f_{i,0}
\overline{p}_{0,0}^{\, *}=\overline{p}_{i,0}
f_{0,0}/f_{i,0}\left (1-\overline{p}_{0,0}(1-f_{0,0})
\right)^{-1}.
$$
Point a) is proven.

If $j\not =0$, the equality 
$$
p_{i,j}^{(n)}={}_{0}p_{i,j}^{(n)}+\sum_{\nu=1}^n f_{i,0}^{(\nu)}
p_{0,j}^{(n-\nu)}
$$
implies
$$
\overline{p}_{i,j}={}_{0}\overline{p}_{i,j}+
\overline{f}_{i,0} \overline{p}_{0,j}.
$$
Using Proposition~\ref{lm_r3} b), for $i=0$, we have
$$
\overline{p}_{0,j}^{\, *}={}_{0}\overline{p}_{0,j}^{\, *}
(1-\overline{f}_{0,0}^{\, *})^{-1}={}_{0}\overline{p}_{0,j}
f_{j,0}(f_{0,0}-\overline{f}_{0,0})^{-1}, 
$$
and for $i\not =0$,
$$
\overline{p}_{i,j}^{\, *}=
{}_{0}\overline{p}_{i,j}f_{j,0}/f_{i,0}+
\overline{f}_{i,0}/f_{i,0} \overline{p}_{0,j}^{\, *}.
$$
Thus,
$$
\overline{p}_{i,j}^{\, *}=
f_{j,0}/f_{i,0}\left ({}_{0}\overline{p}_{i,j}+
\overline{f}_{i,0} {}_{0}\overline{p}_{0,j}(f_{0,0}-
\overline{f}_{0,0})^{-1}\right ).
$$
This completes the proof.$\Box$

\subsection{Finite projection with respect to an arbitrary state. Examples.} \label{proiect_arb}
So far, we have studied the finite projection with respect to state $0$. In the same manner, finite projections $P^{*l}$ are defined with respect to any other state $l$ for which $0<f_{l,l}<1$. It is clear that all results concerning $P^*$ also hold for $P^{*l}$, provided we replace $P^*$ and $0$ with $P^{*l}$ and $l$, respectively.

It is worth noting that if the initial Markov chain $X$ is irreducible, then its finite projections $P^{*l}, \ l\in E$ are either all positive recurrent or all null recurrent.

{\bf Example 1.}  
{\it Markov chains with simple paths.}  
Let us recall the following well-known definitions.
\begin{def1}
A path from state $i$ to $j$ is any sequence of states $i_1, i_2, \ldots, i_n$, $n\ge 2$, not necessarily distinct, such that $i_1=i$, $i_n=j$ and $p_{i_k,i_{k+1}}>0$, $k=1,2,\ldots,n-1$.
\end{def1}
\begin{def1}
A simple path from $i$ to $j$ is any path from $i$ to $j$ where all the states are distinct.
\end{def1}

Let $X$ be a Markov chain where the Markov probability $P$ is such that for any $i\in E$, $i\not =l$, $l\in E$, there exists a unique simple path from $i$ to $l$. Then the finite projection $P^{*l}$ with respect to state $l$ is determined by the one-step transition probabilities
$$
p_{i,j}^{*l}=\left\{\begin{array}{ll}
p_{i,j}f_{j,i}^{1-\delta(i,j)}, & \mbox{if}\ j\not=\phi (i,l)\ \mbox{and}\ i\not=l, \\ 
p_{i,j}/f_{i,j}, & \mbox{if}\ j=\phi (i,l)\ \mbox{and}\ i\not=l, \\ 
p_{l,j}f_{j,l}^{1-\delta(j,l)}/f_{l,l}, & \mbox{if}\ i=l, \\ 
0, & \ \mbox{otherwise.}
\end{array}\right.
$$
Here, $\phi (i,l)$ denotes the second state of the simple path from $i$ to $l$.

If $i\ne l$ and $p_{i,j}>0$, the uniqueness of the simple paths implies
$$
\begin{array}{ll}
f_{i,l}=f_{i,j} f_{j,l}^{1-\delta(j,l)}, & \mbox{when}\ j=\phi (i,l), \\ 
f_{j,l}=f_{j,i}^{1-\delta(j,i)} f_{i,l}, & \mbox{otherwise.}
\end{array}
$$
From this, we obtain both the above formula for $p^{*l}_{i,j}$ and an analogous formula for the initial distribution $\{\gamma^{*l}(i)\}$.

\bigskip

{\bf Example 2.}  
{\it Backward recurrence time Markov chain.}  
We consider the transient case in which all the projections $P^{*l}$, $l\ge0$, are defined. The Markov chain $X$ has transition probabilities
$$
p_{i-1,i}=p_i,\qquad  p_{i-1,0}=q_i=1-p_i, \ \  i=1,2,\ldots.
$$
$X_n$ is interpreted as the time elapsed since the last renewal. In this case, the entry probabilities are
$$
f_{i,j}=\left\{\begin{array}{ll}
1, & \mbox{if}\ i<j, \\ 
1-\prod_{n=i+1}^\infty p_n, &  \mbox{if}\ i\ge j,
\end{array}
\right.
\qquad
f_{0,0}^{(n)}=q_n\prod_{k=1}^{n-1}p_k, \  n=1,2,\ldots.
$$
Thus, the transition probabilities of the finite projection $P^{*l}$, $l\ge 0$, are
$$
p_{i-1,i}^{*l}=p_i^{*l}=\left\{\begin{array}{ll}
p_i(1-\prod_{n=i+1}^\infty p_n)/(1-\prod_{n=i}^\infty p_n), &\ \mbox{if}\ l<i, \\ 
p_i, &\ \mbox{otherwise,}
\end{array}
\right.
$$
$$
p_{i-1,0}^{*l}=q_i^{*l}=\left\{\begin{array}{ll}
q_i/(1-\prod_{n=i}^\infty p_n),
& \ \mbox{if}\ l<i, \\ 
q_i, & \ \mbox{otherwise.}
\end{array}
\right.
$$
Clearly, $p_i^{*l}+q_i^{*l}=1$. Put $u_j=u_0 \prod_{i=1}^j p_i$, $j=0,1,2,\ldots$, where $u_0>0$. The balance equations for the projected transitions give the invariant measures of $P^{*l}$:
$$
u_j^{*l}=\left\{\begin{array}{ll}
u_j(1-\prod_{n=j+1}^\infty p_n)/(1-\prod_{n=l+1}^\infty p_n),
& \ \mbox{if}\ j>l, \\ 
u_j, & \ \mbox{otherwise.}
\end{array}
\right.
$$

\bigskip

The proposition below provides a simple criterion for verifying the coincidence of finite projections of the same Markov probability.

\begin{prop} For any $k,l\in E$, such that $k\not=l$ and the finite projections $P^{*l}$, $P^{*k}$ are well-defined (i.e., $0<f_{l,l}<1$, $0<f_{k,k}<1$), and the initial distribution $\gamma$ is concentrated in a single state, the following statements hold.
\begin{description}
\item[a)] If the finite projections $P^{*l}$ and $P^{*k}$ coincide, then
\begin{eqnarray}
f_{i,l}=f_{i,k}^{1-\delta (i,k)} f_{k,l}, \qquad (\forall ) i\in
A^l_k=\{j\in E: \ {}_{l}p_{j,k}>0\}, \nonumber \\ 
f_{i,k}=f_{i,l}^{1-\delta (i,l)} f_{l,k}, \qquad (\forall ) i\in
A^k_l=\{j\in E: \ {}_{k}p_{j,l}>0\}. \nonumber
\end{eqnarray}
The state sets $A^k_l$ and $A^l_k$ are disjoint, and $l\in A^l_k$, while $k\in A^k_l$. In particular, $f_{l,l}=f_{k,k}= f_{l,k}f_{k,l}$.
\item[b)] If $f_{l,l}=f_{k,k}= f_{l,k}f_{k,l}>0$ and 
$B^l_k=\{j\in E: {}_{\{l,k\}} p_{k,j}>0\}$,
$B^k_l=\{j\in E: {}_{\{l,k\}} p_{l,j}>0\}$,
then $E^{*l}=E^{*k}$, and the transition matrices of $P^{*l}$ and $P^{*k}$ coincide on $(B^l_k\cup B^k_l)\cap E^{*l}$. If the sets $A^l_k$ and $A^k_l$ are disjoint, then the finite projections $P^{*l}$ and $P^{*k}$ themselves coincide.
\end{description}
\end{prop}
{\bf Proof.}  
Let $l,k\in E$, $l\not= k$, and suppose that $P^{*l}=P^{*k}$. Consequently, $E^{*l}=E^{*k}$, ${}_{\{l,k\}}p^{*k}_{i,k}=
{}_{\{l,k\}}p^{*l}_{i,k}$ and  ${}_{\{l,k\}}p^{*k}_{i,l}=
{}_{\{l,k\}}p^{*l}_{i,l}$ \ $(\forall )\ i\in E^{*l}$. 

First consider a return to $l$ with no intermediate visit to $l$ or $k$.
Its probability under $P^{*l}$ is ${}_{\{l,k\}}p_{l,l}/f_{l,l}$,
whereas under $P^{*k}$ it is ${}_{\{l,k\}}p_{l,l}$. Equality and
$f_{l,l}<1$ imply ${}_{\{l,k\}}p_{l,l}=0$. Similarly,
${}_{\{l,k\}}p_{k,k}=0$. Thus every return to either target passes
through the other target. In particular,
$f_{l,l}=f_{k,k}=f_{l,k}f_{k,l}$.

For $i\notin\{l,k\}$, assume temporarily that the probabilities in the last two equalities are both nonzero:
${}_{\{l,k\}}p^{*k}_{i,k},{}_{\{l,k\}}p^{*k}_{i,l}\not=0$. From the definition of finite projection and the coincidence of $P^{*l}$ and $P^{*k}$, we obtain 
$$
{}_{\{l,k\}}p^{*k}_{i,k}= {}_{\{l,k\}}p_{i,k}/f_{i,k}
={}_{\{l,k\}}p^{*l}_{i,k}=
{}_{\{l,k\}}p_{i,k}f_{k,l}/f_{i,l},
$$
and hence $f_{i,l}=f_{i,k}f_{k,l}$. Similarly, we have $f_{i,k}=f_{i,l}f_{l,k}$. From these last two equalities, we obtain $f_{l,k}f_{k,l}=1$, which contradicts the assumption that the Markov chain $X$ is transient.

Thus, for a transient Markov chain $X$, the coincidence of the finite projections $P^{*l}$ and $P^{*k}$ implies that at least one of the probabilities 
${}_{\{l,k\}}p^{*k}_{i,k},{}_{\{l,k\}}p^{*k}_{i,l}$ is zero, and evidently, at least one of the probabilities 
${}_{\{l,k\}}p_{i,k},{}_{\{l,k\}}p_{i,l}$ is zero for $(\forall)\ i\in E^{*l}\setminus\{l,k\}$. This means that any state $i\in E^{*l}$ outside $\{l,k\}$ cannot simultaneously connect to both states $l$ and $k$ through simple paths that do not contain $l$ or $k$ internally.

Therefore, the state set $E$ can be partitioned into three subsets: $A^l_k$, $A^k_l$, $E\setminus E^{*l}$, such that $A_k^l$ contains all states from which state $k$ can be reached with entry into state $l$ forbidden, $A_l^k$ contains all states from which state $l$ can be reached with entry into state $k$ forbidden, and $E\setminus E^{*l}$ contains all states from which neither $l$ nor $k$ can be reached with positive probability.

From the coincidence of the finite projections $P^{*l}$ and $P^{*k}$, we also obtain ${}_{l}p_{k,k}=0$, ${}_{k}p_{l,l}=0$, meaning that $l\in A_k^l$, while $k\in A_l^k$. Thus, we have 
$$
f_{l,l}=f_{l,k}f_{k,l}, \qquad \hbox{and}\qquad 
f_{k,k}=f_{k,l}f_{l,k}.
$$
Moreover, any simple path of positive probability connecting state $i\in A^l_k$ to state $l$ passes through state $k$, and any simple path of positive probability connecting state $i\in A^k_l$ to state $k$ passes through state $l$. Hence,
\begin{eqnarray}
f_{i,l}=f_{i,k}^{1-\delta (i,k)} f_{k,l}, \qquad (\forall ) i\in
A^l_k=\{j\in E: \ {}_{l}p_{j,k}>0\}, \nonumber \\ 
f_{i,k}=f_{i,l}^{1-\delta (i,l)} f_{l,k}, \qquad (\forall ) i\in
A^k_l=\{j\in E: \ {}_{k}p_{j,l}>0\}. \nonumber
\end{eqnarray}
Point a) is proven.

Let us now prove the statement of point b). Write $b={}_{\{k,l\}}p_{l,l}$. First-return decomposition gives
$$
f_{l,l}=b+{}_{\{k,l\}}p_{l,k}f_{k,l},\qquad
{}_{\{k,l\}}p_{l,k}=(1-b)f_{l,k}.
$$
Since $f_{l,l}=f_{l,k}f_{k,l}<1$, these identities imply $b=0$.
Consequently there does not exist $i\in E\setminus\{l,k\}$ such that ${}_{\{k,l\}} p_{l,i}>0$ and ${}_{\{k,l\}} p_{i,l}>0$. Hence, if ${}_{\{k,l\}} p_{l,i}>0$, then all simple paths starting at $i$ and ending at $l$ with positive probability pass through state $k$. Consequently, $(\forall)\ i\in B^k_l=\{j\in E: {}_{\{l,k\}}
p_{l,j}>0\}$, the equality $f_{i,l}=f_{i,k}^{1-\delta
(i,k)} f_{k,l}$ holds. Similarly, $(\forall)\ i\in
B^l_k=\{j\in E: {}_{\{l,k\}} p_{k,j}>0\}$, the equality $f_{i,k}=f_{i,l}^{1-\delta (i,l)} f_{l,k}$ holds.

In the following transition calculations, restrict both $B$-sets to the common state space $E^{*l}=E^{*k}$. To verify the coincidence of the transition matrices of the finite projections $P^{*l}$ and $P^{*k}$ on the state set $B^l_k\cup
B^k_l$, it suffices to show that the one-step transition probabilities $p^{*l}_{i,j}, p^{*k}_{i,j}$ coincide for $(\forall ) i,j \in B^l_k\cup
B^k_l$.

Let $i,j\in B^l_k$. Then $p_{l,j}=0$, $j\not=k$, $f_{i,k}=f_{i,l}^{1-\delta (i,l)} f_{l,k}$. Consequently, $p^{*k}_{l,j}=p^{*l}_{l,j}=0$, and for $i\not=l$, we have
$$
p^{*k}_{i,j}=p_{i,j}\frac{f_{j,k}^{1-\delta(j,k)}}{f_{i,k}}=
p_{i,j}\frac{f_{j,k}}{f_{i,k}}=p_{i,j}\frac{f_{j,l}^{1-
\delta(j,l)}f_{l,k}}{f_{i,l}^{1-\delta(i,l)}f_{l,k}}= 
p_{i,j}\frac{f_{j,l}^{1-\delta(j,l)}}{f_{i,l}}=  p^{*l}_{i,j}.
$$
Thus, $p^{*k}_{i,j}=p^{*l}_{i,j}$ \quad $(\forall)\ i,j\in B^l_k$.

Let $i\in B^l_k$, $j\in B^k_l$. Note that $p_{i,j}^{*l}>0$ can only hold for $i=l$. In this case,  
$$
p^{*l}_{l,j}=p_{l,j}\frac{f_{j,l}}{f_{l,l}}=
p_{l,j}\frac{f_{j,k}^{1-\delta(j,k)}f_{k,l}}{f_{l,k}f_{k,l}}= 
p^{*k}_{l,j}.
$$
The cases when $i,j\in B^k_l$ or $i\in B^k_l$, $j\in B^l_k$ are analogous to those examined, and thus the first part of point b) is proven. If $A^l_k$ and $A^k_l$ are disjoint, every state in the common
projected space lies on exactly one of these two sides. The same
hitting-probability factorizations hold on the whole sides, so
substitution in the projected transition formula gives equality of
the remaining rows as well. Here the symmetric difference
$A^l_k\triangle B^k_l$, apart from $l,k$, contains only transient
states of $X^{*l}$. The initial distribution is concentrated
at the same single state under both projections. Hence the full
projected laws coincide.$\Box$

  \bigskip
From the proved proposition, it follows that the finite projections $P^{*l}$ and $P^{*k}$, with initial distribution concentrated at a single point, coincide if and only if the following conditions hold:
\begin{itemize}
\item $f_{l,l}=f_{k,k}= f_{l,k}f_{k,l}>0$,
\item the sets $A^l_k$ and $A^k_l$ are disjoint.
\end{itemize}

We introduce the notations
$$
R_{l,i}(j)=\sum_{r\ge j,\ r\ne l} {}_{l}p_{l,r} p_{r,i}+p_{l,i}
\ \ \ \mbox{and} \ \ R_{l,i}=R_{l,i}(0) \qquad \mbox{for}\ \
j<l,\ \ l\ge1,\ \ i,j,l\in E.
$$
It is evident that $R_{l,i}={}_{l}p_{l,i}$. The following result establishes the relationship between the invariant measures of the initial Markov probability $P$ and the invariant measures of the finite projections $P^{*l},\ l\in E$.

\begin{prop} \label{inv_meas1}
Let the initial Markov chain $X$ be irreducible, and let the measures \ $u^{*l}=\{u^{*l}_j\}$ \ be invariant for the finite projections \ $P^{*l}$, with \ $u^{*l}_0=f_{0,l}$ for $l\ge1$, then:
\begin{description}
\item[a)] If there exists $K\subseteq E$, $ \hbox{card}\, K=\infty$, such that
\begin{equation}\label{lm_r7}
\lim_{j,l\to \infty,\ j<l,\ l\in K} R_{l,i}(j)/R_{l,0} = 0,\qquad
\mbox{for any}\ \ i\in E,
\end{equation}
then there exists $K'\subseteq K$, $ \hbox{card}\, K' =\infty$, such that the limits exist:
$$
u_j=\lim_{l\to \infty,\ l\in K'} u^{*l}_j/f_{j,l}, \qquad j\in E,
$$
and $\{u_j\}$ is an invariant measure of the initial chain $X$;
\item[b)] If there exists $K\subseteq E$, $ \hbox{card}\, K=\infty$, such that the limits exist:
$$
u_j=\lim_{l\to \infty,\ l\in K} u^{*l}_j/f_{j,l},\qquad j\in E,
$$
and $\{u_j\}$ is an invariant measure of the initial Markov chain $X$, then
$$
\lim_{j,l\to \infty,\ j<l,\ l\in K} R_{l,i}(j)/R_{l,0} =
0,\qquad \mbox{for any}\ \ i\in E.
$$
\end{description}
\end{prop}

{\bf Proof.} Since $u^{*l}$ are the invariant measures of the projections $P^{*l}$, from Proposition~\ref{lm_r3}, for $l\ge1$:
\begin{equation}\label{lm_r6}
u^{*l}_j/f_{j,l}=\sum_i u^{*l}_i/f_{i,l} p_{i,j},\qquad
j\not =l,  
\end{equation}
$$
u^{*l}_l=\sum_i u^{*l}_i/f_{i,l} p_{i,l}.
$$
Since $u^{*l}_0=f_{0,l}$, from Proposition~\ref{lm_r4} a), it follows that:
$$
u^{*l}_i/f_{i,l}={}_{l}p_{l,i}^{1-\delta(i,l)}/{}_{l}p_{l,0}.
$$
Therefore:
$$
u^{*l}_j/f_{j,l}=\sum_{i<r} u^{*l}_i/f_{i,l} p_{i,j}+
\sum_{i\ge r} {}_{l}p_{l,i}^{1-\delta(i,l)}
p_{i,j}/{}_{l}p_{l,0}.
$$
This follows from (\ref{lm_r6}). For brevity write
$v_l(i)=u_i^{*l}/f_{i,l}$. Then $v_l(0)=1$ and $v_lP\le v_l$.
Fixed positive-probability paths from $0$ to $i$ and from $i$ to $0$
therefore give positive lower and finite upper bounds for each $v_l(i)$,
uniformly in $l$. A diagonal selection gives an infinite $K'\subseteq K$
along which every coordinate has a finite positive limit $u_i$.
For fixed $j$, split (\ref{lm_r6}) at a finite index $r<l$.
The finite sum converges coordinatewise, and its remaining tail is
$R_{l,j}(r)/R_{l,0}$. Condition (\ref{lm_r7}) permits this tail to
tend to zero as $r\to\infty$. Thus $u_j=\sum_i u_i p_{i,j}$ for
every $j$, with $u_0=1$, proving a).

For b), fix $i$ and $j$. For $l>j$ and $l\ne i$, the exact tail identity is
$$
\frac{R_{l,i}(j)}{R_{l,0}}
=v_l(i)-\sum_{r<j}v_l(r)p_{r,i}.
$$
Along $K$, its limit for fixed $j$ is
$u_i-\sum_{r<j}u_rp_{r,i}$, which tends to zero as $j\to\infty$
by invariance of $u$. Since the tail decreases as $j$ increases
within $j<l$, this proves the joint limit in b). $\Box$

Condition (\ref{lm_r7}) is analogous to the necessary and sufficient condition for the existence of invariant measures for an irreducible, aperiodic, and transient Markov chain proposed by Harris. The necessity of Harris's condition was demonstrated by Veech \cite{Veech_1}.

\subsection{The space of finite projections of a Markov chain}
\label{spat_proiect}
Let $\mathcal M$ be the space of laws of possibly killed Markov chains,
using the absorbing cemetery state introduced at the beginning of this
chapter. We equip $\mathcal M$ with the $F$-topology: $Q_n\to Q$ if
and only if
\begin{equation}\label{live_prefix}
Q_n\{X_0=i_0,\ldots,X_k=i_k\}\longrightarrow
Q\{X_0=i_0,\ldots,X_k=i_k\}
\end{equation}
for every $k\ge0$ and $i_0,\ldots,i_k\in E$.
These cylinder probabilities determine the law, including the missing
mass due to initial or subsequent killing, so this topology is Hausdorff.
It need not coincide with ordinary weak convergence of the
finite-dimensional distributions on the augmented state space.

We study the finite projections as a topological subspace of
$\mathcal M$. First, we recall the required facts from boundary theory.
A finite nonnegative function $h$ is excessive if
$$
h(i)\ge\sum_jp_{ij}h(j),\qquad i\in E.
$$
For such an $h$ with $\sum_i\gamma(i)h(i)\le1$, its $h$-transform
has initial distribution and transition probabilities
\begin{eqnarray}\label{killed_transform}
\gamma^h(i)&=&\gamma(i)h(i),\quad
\gamma^h(\partial)=1-\sum_i\gamma(i)h(i),\nonumber\\
p^h_{ij}&=&p_{ij}\frac{h(j)}{h(i)},\quad
p^h_{i\partial}=1-\frac{\sum_jp_{ij}h(j)}{h(i)}
\quad(h(i)>0).
\end{eqnarray}
States with $h(i)=0$ are inaccessible and may be assigned immediate
killing; the cemetery state is absorbing. If $h\equiv0$, all the
cylinder probabilities in (\ref{live_prefix}) are zero. Thus nonnegative
excessive functions, including the zero function, are allowed.

In this section assume that the Green function
$$
G(i,j)=\sum_{n=0}^{\infty}p_{ij}^{(n)}
$$
is finite for every $i,j\in E$. Thus the chain has no recurrent states.
Its Martin kernel is
\begin{equation}\label{martin_kernel}
K(i,j)=\frac{G(i,j)}{\sum_a\gamma(a)G(a,j)}.
\end{equation}
The standing condition (\ref{etalon}) ensures that the denominator is
positive: $\gamma$ is a reference distribution. We also assume
$0<f_{ll}<1$ for every $l\in E$, so that all the finite projections
appearing below are defined.

The Martin compactification $E^+$ carries the $M^+$-topology.
A sequence $\{l_n\}$ in $E$ converges in this topology if and only if
the functions $K(i,l_n)$ converge for every $i\in E$ and either
$l_n\to\infty$ or $l_n$ is eventually constant. The finite states are
isolated, and $E$ is open in $E^+$.

For fixed $l$, $K(\cdot,l)$ is a finite excessive function. From
$G(i,l)=f_{il}^{1-\delta(i,l)}G(l,l)$ we obtain, on $E^{*l}$,
$$
p_{ij}^{*l}=p_{ij}^{K(\cdot,l)}f_{il}^{-\delta(i,l)}.
$$
Thus the finite projection and the $K(\cdot,l)$-transform have the
same transition probabilities away from the target $l$.
Boundary theory for countable-state Markov chains is presented, for
example, in \cite{lm_dyn} and \cite{lm_kem}. The following results
describe its connection with finite projections.

\Needspace{10\baselineskip}
\begin{samepage}
\begin{prop}
Let $l_n\to\infty$.\label{lm_sir} The following statements are equivalent:
\begin{description}
\item[a)] $l_n$ converges in the $M^+$-topology to a point
$l^+\in E^+\setminus E$;
\item[b)] $K(i,l_n)$ converges for every $i\in E$, with finite limit
$K(i,l^+)$;
\item[c)] $P^{*l_n}$ converges in the $F$-topology.
\end{description}
The limit in c) is the possibly killed $K(\cdot,l^+)$-transform, with
initial live distribution $\gamma(i)K(i,l^+)$.
\end{prop}
\par
\end{samepage}
{\bf Proof.} The equivalence of a) and b) follows from the definition
of the $M^+$-topology and the bound
\begin{equation}\label{martin_coordinate_bound}
K(i,j)\le\left(\sum_a\gamma(a)f_{a,i}\right)^{-1}<\infty.
\end{equation}
Indeed, $G(a,j)\ge f_{a,i}G(i,j)$ by counting visits after the first
positive hit of $i$. The denominator in the bound is positive by
(\ref{etalon}).

For finite $l$, multiplication of the initial and transition
probabilities in Proposition~\ref{lm_r3}(b), with target $l$ in place
of $0$, gives
\begin{equation}\label{lm_*}
P^{*l}[i_0,\ldots,i_k]=P[i_0,\ldots,i_k]K(i_k,l)
\prod_{m=0}^{k-1}f_{i_m,l}^{-\delta(i_m,l)}.
\end{equation}
The product includes a departure from $l$ at time zero.
For a fixed cylinder, $l_n$ eventually differs from all its states,
so the product is one. Therefore b) implies
$$
P^{*l_n}[i_0,\ldots,i_k]\longrightarrow
P[i_0,\ldots,i_k]K(i_k,l^+).
$$
Fatou's lemma gives $PK(\cdot,l^+)\le K(\cdot,l^+)$ and
$\sum_i\gamma(i)K(i,l^+)\le1$. The displayed limits are precisely
the cylinder probabilities of the transform (\ref{killed_transform}),
proving c). Conversely, the reference-distribution condition supplies
a positive-probability finite path ending at each $i$. Dividing its
convergent cylinder probability by the fixed positive original
probability in (\ref{lm_*}) recovers convergence of $K(i,l_n)$.
This proves b). $\Box$

If $\gamma(i)>0$, the initial distribution gives the relation
$$
K(i,l)=\frac{P^{*l}\{X_0=i\}}{\gamma(i)}.
$$
The proposition interprets transforms associated with Martin boundary
kernels as limits of finite projections.

Define an equivalence relation $R$ on $E^+$ by identifying finite
states $i,j$ exactly when $P^{*i}=P^{*j}$, and leaving boundary
points in singleton classes. Give $E^+/R$ the quotient topology
$M^+/R$. Write
$$
\mathcal P=\{P^{*l}:l\in E\},
$$
and let $\overline{\mathcal P}$ be its closure in $\mathcal M$.
\begin{prop}\label{homeo}
The space $(E^+/R,M^+/R)$ is homeomorphic to
$\overline{\mathcal P}$ with the $F$-topology.
\end{prop}
{\bf Proof.} Set $\varphi(l)=P^{*l}$ for $l\in E$ and
$\varphi(l^+)=P^{K(\cdot,l^+)}$ for $l^+\in E^+\setminus E$.
Proposition~\ref{lm_sir} and the same cylinder argument give continuity
at boundary points; finite states are isolated. Its image is compact
and hence closed in the Hausdorff space $\mathcal M$. Since $E$ is
dense in $E^+$, this image is $\overline{\mathcal P}$.
Reference paths recover every kernel coordinate, so distinct boundary
points have distinct images. A boundary transform has finite expected
live visits to each fixed $j$, namely
$K(j,l^+)\sum_i\gamma(i)G(i,j)$, whereas a finite projection visits
its target infinitely often. Thus a boundary image cannot equal a
finite projection. The induced continuous bijection from the compact
quotient to its Hausdorff image is a homeomorphism. $\Box$

In particular, if $P^{*i}\ne P^{*j}$ for every pair $i\ne j$, then
$E^+$ itself is homeomorphic to $\overline{\mathcal P}$ with the
$F$-topology.

\begin{lema}\label{martin_as_limit} \mbox{\cite{lm_dyn}}
If $P\{X_n\in E\mbox{ for all }n\ge0\}=1$, then $X_n$ converges
almost surely in the $M^+$-topology to a point $X^+\in E^+\setminus E$.
In particular, for every $i\in E$,
$$
K(i,X_n)\longrightarrow K(i,X^+)\quad P\hbox{-a.s.}
$$
\end{lema}
This lemma and the equivalence of b) and c) in
Proposition~\ref{lm_sir} immediately give the following result.
\begin{prop}\label{projection_as_limit}
If $P\{X_n\in E\mbox{ for all }n\ge0\}=1$, then
$$
P^{*X_n}\longrightarrow P^{K(\cdot,X^+)}
\quad P\hbox{-a.s. in the }F\hbox{-topology}.
$$
\end{prop}

\subsection{$\tau$-Equivalent Markov Probabilities} \label{invar_pr}

In this section, we will briefly discuss a property that is invariant for the initial Markov chain and its finite projections. Let $\tau$ be a Markov time for the chain $X$, i.e., $\tau$ is a random variable defined on the measurable space $(\Omega, {\cal F})$ with values in $[0, \infty]$ such that for $(\forall)\ n \geq 0$, $\{\tau \leq n\} \in {\cal F}_n$. We denote by $\Gamma(\tau, P)$ the generalized phase space (state space) of the chain $X$ stopped at time $\tau$:
$$
\Gamma(\tau, P) = \{(n, i): \ n \geq 0,\ i \in E,\ P\{X_n = i,\ n < \tau\} > 0\}.
$$

\begin{def1}
Markov probabilities $P$ and $Q$ are called $\tau$-equivalent if they satisfy the conditions:
\begin{description}
\item[a)] $\Gamma(\tau, P) = \Gamma(\tau, Q)$,
\item[b)] $P\{B\ |\ \tau > n,\ X_n = i\} = Q\{B\ |\ \tau > n,\ X_n = i\}$,
for $(\forall)\ B \in {\cal F}_n$ and $(\forall)\ (n, i) \in \Gamma(\tau, P)$.
\end{description}
\end{def1}

In other words, two Markov probabilities $P$ and $Q$ are $\tau$-equivalent if, in reverse time $n, n-1, \ldots, 0$, the Markov chain $X$ behaves in the same way with respect to both $P$ and $Q$, given that $\tau > n$ and the state of $X$ at time $n$ is known.

Clearly, the $\tau$-equivalence of two arbitrary processes (probabilities) can also be defined in the same way. When $\tau = \infty$, the above definition reduces to the definition of equivalence of two Markov chains, used in \cite{Kerov_O} and expressed in terms of cotransition probabilities. Equivalent Markov processes have been studied in a number of works, including \cite{Kerov_O} and \cite{Versik_K}.

Write $T_l^0=\inf\{n\geq0:X_n=l\}$, including a hit at time
zero. For an arbitrary initial distribution the following assertion
uses $T_l^0$, rather than the first positive return.
\begin{prop}\label{tau_equiv_corrected}
If $X$ is irreducible and $\tau\leq T_l^0$ is a Markov time,
then $P$ and $P^{*l}$ are $\tau$-equivalent.
\end{prop}
{\bf Proof.} On $\{\tau>n,X_n=i\}$ there is no visit to $l$
at any of the indices $0,\ldots,n$. The likelihood ratio for
each live prefix in (\ref{lm_*}) is therefore $K(i,l)$, which
depends only on its endpoint. Conditional on that endpoint it
cancels from the law of the past. Irreducibility gives matching
positive phase-space supports. This proves both requirements of
$\tau$-equivalence. If the initial state is fixed at $l$ and
the first positive return is used instead, the additional factor
$f_{l,l}^{-1}$ is constant on the histories before that return and
also cancels. This latter convention must not be mixed with an
arbitrary initial distribution. $\Box$

From the demonstrated proposition, it follows in particular that if the initial Markov chain is irreducible, then the finite projections $P^{*l}$ and $P^{*m}$ of the Markov probability $P$ are $\tau$-equivalent for any Markov time $\tau \leq \inf\{n \geq 0:\ X_n \in \{l, m\}\}$ $P$-a.s. As another consequence of the demonstrated proposition, we obtain:

\begin{teor}
Let $f(X) = f(X_0, \ldots, X_{n-1})$ be a functional that depends only on the first $n$ values of the Markov chain $X$. Then for any Markov time $\tau \leq T_l^0$ and for any $(n, i) \in \Gamma(\tau, P^{*l})$:
$$
\hbox{\bf E} (f(X)|\ X_n = i,\ \tau > n) = \hbox{\bf E}^{*l} (f(X)|\ X_n = i,\ \tau > n),
$$
provided at least one of the specified expectations exists.
\end{teor}

\subsection{Markov Chains Embedded in the Queueing Systems
  $ M/G_{E}/1$,\ $GI/M/1 $\ and Their Finite Projections}

      In this section, we will complete and extend some results presented in \cite{l_pnm7}. Like the results of the previous section, they concern the behavior of Markov chains in the reverse direction of time. The following is true:
  \begin{prop}\label{lm_que}
  Let $ X=\{X_n\}_{n=0}^{+\infty} $ be an irreducible, aperiodic, transient Markov chain, and its transition probabilities $ (p_{i,j},\ i,j\in E) $ satisfy the conditions:
  \begin{itemize}
  \item[a)] 
   $\sum\limits_i p_{i,j}=1$, for $j\in E\setminus\{l\}$;
\item[b)] 
  the probability $f_{l,l}$ of returning to state $l$ is \ $ f_{l,l} =
   \sum\limits_i p_{i,l} < 1$;
\end{itemize}
Then the measure $\{f_{i,l}\}_{i=0}^{+\infty}$ is 
  invariant for the Markov probability $P^{*l}$. The transition matrix of the reversed chain with respect to this invariant measure is $ (q_{i,j},\
  i,j\in E) $, \ where $q_{i,j}=p_{j,i}f_{i,l}^{-\delta(i,l)}$.
\end{prop}

{\bf Proof.} Let $f_{i,l} = P\{ \exists n > 0: X_n = l \ |\ X_0 = i\}$, i.e., $f_{i,l}$ is the probability of reaching state $l$ from state $i$. If conditions a) and b) are satisfied, then by direct verification, we obtain that the measure $\{f_{i,l}\}_{i=0}^{+\infty}$ is invariant for the Markov probability $P^{*l}$. Therefore, the transition probabilities of the reversed chain with respect to this invariant measure are given by
$$
q_{i,j} = p_{j,i}^{*l}f_{j,l}/f_{i,l} = p_{j,i}f_{i,l}^{-\delta(i,l)}.
$$
Thus, the proposition is proved. $\Box$

Let us consider an $M/G_E/1$ queueing system, i.e., a system where requests arrive at a single service station following a Poisson process with parameter $a > 0$. There is unlimited waiting space for requests arriving at a busy system. The service times for successive requests are independent and identically distributed random variables (except for requests entering an unoccupied system) with distribution function $B(t)$ and finite mean $b_1$. The duration of the first service in each busy period has the distribution function $E(x) = b_1^{-1} \int\limits_0^x (1 - B(u))\, du$, rather than $B(x)$.

With each $M/G_E/1$ queueing system, we associate a dual $GI/M/1$ queueing system, where the interarrival times have distribution $B(x)$, and the service times are exponentially distributed with parameter $a > 0$.

Assume that $0<b_1<\infty$ and $ab_1<1$.

Let $X = \{X_n\}_{n=0}^{+\infty}$ and $\overline{X} = \{\overline{X}_n\}_{n=0}^{+\infty}$ denote the numbers of customers in the systems observed just before the arrival times of requests in the $GI/M/1$ system and immediately after the departure times of requests from the $M/G_E/1$ system, respectively. Let $b_2 = \int\limits_0^\infty x^2\,dB(x)$. The main result of this section is:

\begin{prop}
\begin{itemize}
\item[a)]
The finite projection $X^{*0}$ with respect to state $0$ of the transient Markov chain $X$ embedded in the $GI/M/1$ system is ergodic if and only if $b_2 < +\infty$.
\item[b)]
If $b_2<\infty$, the chain $\overline{X}$ and the chain with transition matrix $(p^{*0}_{i,j})$, both started in their common stationary distribution, are time-reversals of one another.
\end{itemize}
\end{prop}

{\bf Proof.}
Let $p_{i,j}$ and $\overline{p}_{i,j}$, $i, j \in E$, denote the corresponding probabilities of $X$ and $\overline{X}$, respectively. Then:
$$
p_{i,j} = \left\{\begin{array}{ll}
    \int\limits_0^\infty \left[1 - \sum\limits_{k=0}^i \frac{(ax)^k}{k!} e^{-ax} \right] \, dB(x), & \hbox{for } j = 0, \, i \geq j; \\
    \int\limits_0^\infty \frac{(ax)^{i+1-j}}{(i+1-j)!} e^{-ax} \, dB(x), & \hbox{for } j \geq 1, \, i \geq j-1; \\
    0, & \hbox{for } j > 1, \, i < j-1.
\end{array} \right.
$$

The transition probabilities $(p_{i,j},\ i, j \in E)$ satisfy conditions a) and b) of Proposition \ref{lm_que}. Indeed, condition a) is evidently satisfied, and
$$
\sum_{i=0}^\infty p_{i,0} = \sum_{i=0}^\infty \int_0^\infty \frac{a(ax)^i}{i!} e^{-ax} \left[1 - B(x)\right] \, dx = ab_1.
$$

On the other hand, it is easy to verify using the Laplace-Stieltjes transform of the busy period in the $GI/M/1$ system that for the Markov chain $X$, $f_{0,0} = ab_1$, i.e., condition b) of Proposition \ref{lm_que} is satisfied. Since:
$$
\overline{p}_{i,j} = \left\{\begin{array}{ll}
    \int\limits_0^\infty \frac{(ax)^j}{j!} e^{-ax} \, dE(x), & \hbox{for } i = 0, \, j \geq i; \\
    \int\limits_0^\infty \frac{(ax)^{j+1-i}}{(j+1-i)!} e^{-ax} \, dB(x), & \hbox{for } i \geq 1, \, j \geq i-1; \\
    0, & \hbox{for } i > 1, \, j < i-1.
\end{array} \right.,
$$
and
$$
p_{i,0} = ab_1 \int\limits_0^\infty \frac{(ax)^i}{i!} e^{-ax} \, dE(x),
$$
Proposition~\ref{lm_que} identifies $(\overline p_{i,j})$ as the reversed transition matrix of $(p^{*0}_{i,j})$ with respect to the invariant measure $\{f_{i,0}\}$. The two recurrent chains therefore have the same positive or null recurrence type. As shown in \cite{l_pnm7}, the busy period $\Pi$ in the $M/G_E/1$ system has a distribution given by the Laplace-Stieltjes transform:
$$
\pi(s) = \frac{1 - \pi_1(s)}{b_1 \left(s + a - a\pi_1(s)\right)},
$$
where $\pi_1(s)$ is the Laplace-Stieltjes transform of the busy period in the $M/G/1$ queueing system with a Poisson input process of parameter $a > 0$ and service times with distribution function $B(x)$. The mean busy period and the mean positive return time of the departure chain to zero are
$$
{\bf E}\Pi=\frac{b_2}{2b_1(1-ab_1)},\qquad
\overline m_{0,0}=1+a{\bf E}\Pi
=1+\frac{ab_2}{2b_1(1-ab_1)}.
$$
Indeed, the return time counts the initial customer and the customers arriving during the busy period. These formulas also hold with value infinity when $b_2=\infty$: the first service has mean $b_2/(2b_1)$, and the subsequent ordinary busy periods have mean $b_1/(1-ab_1)$. Thus the mean return time is finite exactly when $b_2<\infty$. Since the departure chain has a positive self-loop, this proves a). When $b_2<\infty$, normalizing the common invariant measure and starting both chains in that distribution proves b). $\Box$

\subsection{Finite projection of semi-Markov processes}\label{semi_proj}
In this section, we examine some properties of finite projections taken from semi-Markov processes with a discrete state space $E=\{0,1,2,\ldots\}$. Let $(X,Y)=\{X_{n+1},Y_n\}_{n\geq0}$ be a Markov renewal sequence on $(\Omega,{\cal F},P)$: conditional on its history through $X_n$, the law of the next state and holding time depends only on $X_n$:
$$
F_{i,j}(t) = P\{X_{n+1} = j, Y_n \leq t \ | \ X_n = i\} = P\{X_1 = j; Y_0 \leq t \ | \ X_0 = i\}
$$
Such a process $(X,Y)$ is called a Markov renewal process, and $F(t) = \{F_{i,j}(t)\}$ is called the semi-Markov kernel associated with the process $\{X_{n+1}, Y_n\}_{n \geq 0}$. We assume that the kernel $F(t)$ is substochastic, i.e., $\lim\limits_{t \to \infty} \sum\limits_j F_{i,j}(t) \leq 1$ for $\forall \ i \in E$. The missing row mass represents termination of the embedded sequence: no next state is assigned, and the holding time at the last state is infinite. On a finite transition the next state belongs to $E$ and the holding time is finite. In the kernel formula, the event includes existence of that transition. We also assume that $\forall \ i, j \in E$, $F_{i,j}(0) = 0$.

The one-step transition probabilities of the embedded chain are
$$p_{i,j}=\lim_{t\to\infty}F_{i,j}(t).$$
For $p_{i,j}>0$, the conditional holding-time distribution is $G_{i,j}(t)=F_{i,j}(t)/p_{i,j}$.

Let $T_n = Y_0 + Y_1 + \ldots + Y_n$, $n = 0,1,2,\ldots$, and $I\{T_n \leq t\}$ be the indicators of the events $\{T_n \leq t\}$, with subsequent $T_n=\infty$ after termination. Then $N(t) = \sum\limits_{n=0}^\infty I\{T_n \leq t\}$ is the Markov renewal counting process associated with $\{X_{n+1}, Y_n\}_{n \geq 0}$. Once the counting process $N(t)$ is defined, the semi-Markov process corresponding to the Markov chain $\{X_{n+1}, Y_n\}_{n \geq 0}$ is easily defined. Put $\zeta=\sum_{n\geq0}Y_n$, with $\zeta=\infty$ when the embedded sequence terminates with an infinite holding time. The minimal semi-Markov process is
$$
Z(t)=X_{N(t)},\qquad 0\leq t<\zeta.
$$
The path is considered only up to its lifetime $\zeta$; no additional state is introduced at explosion. The process is regular if $P\{\zeta<\infty\}=0$ and irregular otherwise. If $\forall \ i \in E$, $p_{i,i} = 0$, then the counting process $N(t)$ indicates the number of jumps of the trajectory of process $Z$ up to and including time $t$.

Let $S=\{S(n)\}_{n\geq0}$ be the successive renewal epochs at a fixed state $l\in E$:
$$
S(0)=\inf\{0\leq t<\zeta:Z(t)=l\},
$$
$$
S(n+1)=\inf\{T_{N(S(n))}\leq t<\zeta:Z(t)=l\},\qquad n\geq0.
$$
The second formula is used when $S(n)$ and its next renewal epoch are finite; otherwise $S(n+1)=\infty$. Here $\inf\emptyset=\infty$. These are marked renewal visits, including self-transitions. The Markov-renewal property makes $S$ a delayed renewal sequence. For its transient finite projection assume $P\{S(0)<\infty\}>0$ and $0<f_{l,l}<1$.

The process $Z$ is a regenerative process with the renewal process $S$ included, so if the renewal process $S$ is transient, the notion of finite projection of regenerative processes can also be applied to the semi-Markov process $Z$.

The finite projection $Z^{*l}$ of the semi-Markov process $Z$ with respect to the renewal process $S$ will be called the finite projection with respect to state $l$. Construct the projection on the joint space of the embedded sequence and holding times, so that renewal visits are recorded even when the state does not change. Denote its law, and the induced law of $Z$ with its lifetime, by $P^{*l}$.

\begin{prop}\label{pr_sem}
The finite projection $Z^{*l}$ of the semi-Markov process $Z$ with respect to state $l$ is a regular semi-Markov process described by the semi-Markov kernel
$$
F^{*l}_{i,j}(t) = G_{i,j}(t) p_{i,j}^{*l},
$$
where $p_{i,j}^{*l}$ are the transition probabilities in the finite projection $X^{*l}$ of the embedded Markov chain $X$. Moreover, the initial distribution of the projected process $Z^{*l}$ coincides with the initial distribution of the projection $X^{*l}$:
$$
P^{*l}\{Z(0) = i\} = P^{*l}\{X_0 = i\} = \gamma^{*l}(i).
$$
\end{prop}
Before proving the proposition, we will demonstrate an elementary lemma that, in fact, we have already tacitly proven earlier.
\begin{lema}
Let $A,B,C$ be events with $P(B)>0$ for which the following hold:
\begin{description}
\item[a)] Event $B$ implies $A$: $A \supseteq B$;
\item[b)] Events $B$ and $C$ are independent conditional on the occurrence of $A$: 
$$
P\{B \cap C|A\} = P\{B|A\}P\{C|A\};
$$
then $P\{C|B\} = P\{C|A\}$.
\end{description}
\end{lema}
{\bf Proof.} 
$$
P\{C|B\} = P\{C|A \cap B\} = \frac{P\{C \cap B|A\}}{P\{B|A\}} = 
\frac{P\{B|A\}P\{C|A\}}{P\{B|A\}} = P\{C|A\}. \ \ \Box
$$

{\bf Proof of Proposition \ref{pr_sem}.}
Conditional on $X_n=i$ and $X_{n+1}=j$, the holding time $Y_n$ is independent of the future embedded path. Conditioning on further visits to $l$ therefore leaves its law $G_{i,j}$ unchanged and gives the embedded transition probabilities $p_{i,j}^{*l}$. More explicitly, on the projected state space,
$$
F_{i,j}^{*l}(t)=F_{i,j}(t)
\frac{f_{j,l}^{1-\delta(j,l)}}{f_{i,l}}
=G_{i,j}(t)p_{i,j}^{*l}.
$$
The denominator is positive there; zero transition entries are assigned zero kernel entries. At $i=l$, $f_{l,l}$ is the positive-return probability. The same calculation holds conditional on the full history through $X_n$, so the projected sequence is Markov renewal. The preceding lemma expresses the conditional independence used here.

The projected embedded chain reaches the recurrent class of $l$ almost surely (Proposition~\ref{lm_r3}). Its initial hit uses finitely many finite holding times. Thereafter the return-cycle lengths are independent, identically distributed and strictly positive; their sum diverges almost surely. Hence $Z^{*l}$ is regular, in agreement with the standard recurrence criterion in~\cite{part1_1}.

The initial distributions coincide because $Z(0)=X_0$. $\Box$

The demonstrated proposition shows that determining the finite projection of a semi-Markov process reduces to determining the finite projection of the Markov chain embedded in it.

Suppose that the projected embedded chain admits an invariant probability distribution $u^{*l}$ with $u_l^{*l}>0$. The mean-return formula (see Proposition 2.2, page 228 in \cite{part1_1}) gives
$$
\mu^{*l}={\bf E}^{*l}[S(n)-S(n-1)]=\frac{\mu}{u_l^{*l}},\qquad n\geq1,
$$
$$
\mu=\sum_{i,j\in E}u_i^{*l}\mu_{i,j}^{*l},\qquad
\mu_{i,j}^{*l}=\int_0^\infty t\,dF_{i,j}^{*l}(t).
$$
Here $0<\mu<\infty$ is assumed for the following estimates.

Combining Theorem \ref{t_prob_dif}, proven in the first chapter, with the average return time $\mu^{*l}$ to state $l$, we obtain an asymptotic upper bound for the difference between the initial probability $P$ and its finite projection $P^{*l}$, on events occurring before time $t$, valid for all sufficiently large $t$: 
$$
\sup_{A\in {\cal F}_t} \vert P^{*l}\{A\} -P\{A\} \vert 
\le 1-r_0 + (1-f_{l,l})\big(\varepsilon (t)+1\big)
t\frac{u_l^{*l}}{\mu},
$$
where $\mu$ and $u_l^{*l}$ are specified above, ${\cal F}_t$ records the path before $t$ within its lifetime (it is generated by $\{x<\zeta,Z(x)\in A\}$ for $x<t$ and $A\subseteq E$), $r_0= P \{S(0) < \infty \}$, \
$f_{l,l}= P \{S(1) < \infty | S({0})< \infty \}$, and
$\varepsilon(t)$ is nonnegative, nonincreasing and tends to zero as $t\to\infty$. 

From the last estimation, for a discrete-time Markov chain with unit holding times and invariant probability distribution $u^{*l}$, so that $\mu=1$, we obtain the estimation of the approximation of the Markov probability $P$ by its finite projection $P^{*l}$:
$$
\sup_{A\in {\cal F}_n} \vert P^{*l}\{A\} -P\{A\} \vert 
\le 1-r_0 + (1-f_{l,l})\big(\varepsilon (n)+1\big)
n{u_l^{*l}},
$$
for all sufficiently large $n$, where $\varepsilon(n)$ is nonnegative and nonincreasing and $\lim\limits_{n\to\infty} \varepsilon (n)=0$. 

Applying the clarifications of Theorem \ref{t_prob_dif}, the difference between $P$ and $P^{*l}$ can be estimated even more accurately, naturally under more restrictive conditions.

It is noteworthy that, in accordance with the proven proposition, taking the finite projection of an irregular semi-Markov process transforms it into a regular process.   
\par
\begin{samepage}
\bigskip
  
\centerline{\large \bf Conclusions and Observations}
  
For a transient Markov chain, multiple finite projections can be defined, one projection for each self-communicating state of the chain. There are situations where the finite projections taken with respect to different states coincide. For an irreducible Markov chain, its finite projections are either all null recurrent or all positive recurrent. 
\par
\end{samepage}

Since finite projections of Markov chains have invariant measures, the notion of finite projection can serve as a means of demonstrating the criterion for the existence of invariant measures for transient Markov chains (Proposition \ref{inv_meas1}). That is, under certain conditions, the existence of invariant measures is a property common to both finite projections and the initial Markov chain. The separate connection with the Martin boundary in Proposition \ref{homeo} uses the $F$-topology on possibly killed laws. 

Determining the finite projection of semi-Markov processes reduces to determining the projection of the Markov chain embedded in the initial process. It is significant that the finite projection of a semi-Markov process is a regular process.
  
\newpage      
\section{Skip-Free Semi-Markov Processes}
\setcounter{equation}{0}
\subsection{Definition of the Process}
Many random phenomena are described by semi-Markov processes with a discrete state space, the transition graph from one state to another being a chain. Such semi-Markov processes will be called skip-free semi-Markov processes (PSSF). Skip-free semi-Markov processes with values in the set of integers $ Z $ represent a generalization of birth-and-death processes (see \cite{l_frc1}, \cite{c2}), their applications being similar to those of birth-and-death processes. 

It is well known that there is a connection between birth-and-death processes and continued fractions (see \cite{l_pnm3},
\cite{l_2:k1}, \cite{l_frac6}, \cite{l_frac7}, \cite{l_frac8}, \cite{l_frac9}).
In the current section, it is shown that this connection is caused solely by the structure of the transition matrix of the Markov chain embedded in the birth-and-death process, that is, only by the property of the process being skip-free.

The main goal of the section is to present the Laplace-Stieltjes Transforms (LST) of the sojourn time of the PSSF in a given state set and the ordinary Laplace transforms of the transition probabilities of the PSSF from one state to another, expressed through continued fractions, generalizing some results obtained in the section dedicated to birth-and-death processes. By appealing to some probabilistic reasoning, we obtain an exact estimation of the relative approximation error of continued fractions by their truncations. First, we will provide the definition of PSSF and introduce the basic notations.
      
The PSSF is defined constructively as follows. 
Consider a family of probability spaces $(\Omega,{\cal F},
P_m) $, $ m\in Z $. Let $x=\{x_0(\omega),x_1(\omega),\ldots \} $ be a random walk defined on $ (\Omega,{\cal F},P_m) $ with values in $ Z $ and transition probabilities $ p_{i,i+1}=p_i,
p_{i,i-1}=q_i, p_i+q_i=1 $, $ i\in Z $. We assume that
$ P_m(x_0(\omega)=m)=1 $. Let $ \xi_1(\omega),\xi_2(\omega),\ldots
$ be a sequence of i.i.d. random variables, each uniformly distributed on the interval $[0,1]$, entirely independent of $\{x_n(\omega), n=0,1,\ldots\}$. We assume that for each 
$i\in Z$, the distribution functions $A_i(t)$ and 
$B_i(t)$ are given.
Define the functions $\phi_{i,i-1}(t)$ and $\phi_{i,i+1}(t)$ such that $\phi_{i,i-1}(\xi)$ has the distribution function $B_i(t)$, and $\phi_{i,i+1}(\xi)$ has the distribution function $A_i(t)$, where $\xi$ is uniformly distributed on $[0,1]$. For the passage formulas below, holding times are proper and strictly positive. Now we set
$$
\tau_k=\phi_{x_{k-1},x_k}(\xi_k),
$$
$$
X(t)=x_{k-1}, \; \;  \hbox{if} \sum_{j=1}^{k-1}\tau_j \le
t<\sum_{j=1}^k\tau_j, \; \; (\sum_{j=1}^0\tau_j=0)
$$
Put $\zeta=\sum_{k\geq1}\tau_k$ and consider this path only for $0\leq t<\zeta$. The process $X(t)$ defined in this way will be called a {\it skip-free semi-Markov process.}

Certain problems in queueing theory as well as in reliability theory can be reduced to determining the distribution of the sojourn time of a semi-Markov process in a fixed set of states. In \cite{l_frac5}, this problem is reduced to solving a system of linear equations for the Laplace-Stieltjes transforms of the sojourn times. Therefore, for PSSF, the problem of determining the distribution of the random variable $W_S=\inf\{0\leq t<\zeta:X(t)\notin S\}, \  S \subset Z $, known as the sojourn time of the PSSF in the set of states $S$, is of interest. We will limit ourselves to cases $S=\{m,m+1,\ldots,
n-1,n\}, \ m<n$, $S=\{m,m+1,\ldots\} $ and $S=\{ \ldots,m-1, m\}$,
the latter two being symmetric. Other cases, evidently, reduce to the ones listed. Since PSSF is a relatively simple class, we will not rely on the systems of equations presented in \cite{l_frac5}, but instead deduce simpler relations.

Here and below $\inf\emptyset=\infty$: exit means reaching a state outside $S$ before the lifetime ends. For a half-line this is a boundary-passage time; the elapsed residence before the lifetime ends is $W_S\wedge\zeta$. Let $V_n=\inf\{0\leq t<\zeta:X(t)=n\}$, then
$$
V_{n}= \left\{
\begin{array} {ll}
             W_{\{n+1,n+2,\ldots \}} & \mbox{if $m>n$} \\
             W_{\{\ldots,n-2,n-1\}} & \mbox{if $m<n$}.
\end{array}  \right.
$$
Let $F_{m,n}(t)$ be the distribution function of the random variable $V_n$ under the condition that $X(0)=m$, i.e. $F_{m,n}(t)=P_m\{V_n\leq t\}$.
$F_{m,n}(t)$ can also be viewed as the distribution function of the absorption time in state $n$, if state $n$ is absorbing, i.e.
$P_m\{X(t_2)=n | X(t_1)=n\}=1$ for $t_2>t_1$.

\subsection{ Laplace-Stieltjes Transforms of Sojourn Times}

For the passage and continued-fraction formulas in this section, assume $0<p_i,q_i<1$. We introduce notation for the Laplace-Stieltjes transforms (LST)
$$
w_{m,n}(s)=\int^{+\infty}_{0}e^{-st}\,dF_{m,n}(t), \ \ m,n\in Z, \  m\ne n,
$$
$$
a_i(s)=\int^{+\infty}_{0}e^{-st}\,dA_i(t), \ \
b_i(s)=\int^{+\infty}_{0}e^{-st}\,dB_i(t), \ i\in Z.
$$

Passage laws may be defective: $w_{m,n}(s)={\bf E}_m[e^{-sV_n};V_n<\infty]$ for $\mathop{\rm Re}s\geq0$, and $w_{m,n}(0)$ is the hitting probability. At each finite hitting epoch, the remaining passage has the corresponding fresh law by the Markov-renewal property.

\begin{lema}\label{l1}
\begin{description}
\item[a)]
\begin{equation}
{w}_{m,m-1}(s)=\frac{q_mb_m(s)}{1-p_ma_m(s){w}_{m+1,m}(s)}
\label{frcon1}
\end{equation}
\item[b)]
$
{w}_{m,n}(s)=\prod\limits_{i=n}^{m-1}{w}_{i+1,i}(s), \quad \; n<m,
$
\end{description}
\end{lema}
{\bf Proof.} The equalities in the lemma are a simple consequence of the fact that at each jump the value of the process $X$ changes by only one unit, that is, $X(T_k)=X(T_{k-1}) \pm 1, \  k=1,2,\ldots$ where $T_k=\tau_1+\cdots+\tau_k$ and $T_0=0$.

Since $X(0)=m$, the following relations hold, implying the respective equalities of the lemma:
\begin{description}
\item[a)] $W_{\{m,m+1,\ldots\}}=\tau_1+W_{m+1,m}'+W_{m,m-1}'$,\
where the random variables $W_{m+1,m}'=W_{m,m-1}'=0$ \ if $X(\tau_1)=m-1$,\ and in the case where $X(\tau_1)=m+1$, \ $W_{m+1,m}'$, $W_{m,m-1}'$ \ have the distributions $F_{m+1,m}$\  and $F_{m,m-1}$ respectively;
\item[b)] if $m>n$, then $W_{\{n+1,n+2,\ldots\}}=W_{\{m,m+1,\ldots\}}+
W_{m-1,m-2}'+\ldots+W_{n+1,n}'$,\newline
 where \ $W_{m-1,m-2}',\allowbreak\ldots,\allowbreak W_{n+1,n}'$\hfil\break
have the distributions $F_{m-1,m-2},\allowbreak\ldots,\allowbreak F_{n+1,n}$ respectively.$\Box$
\end{description}

For approximate calculation of $w_{m,m-1}(s)$,
the following expansion is useful.
\begin{prop}\label{prp1}
For $s\in \{p\in {\cal C} : Re(p)\ge 0\}$ there is a continued fraction expansion
\begin{equation}
 w_{m,m-1}(s)=\frac{q_mb_m(s)}{1-}
\ \frac{p_ma_m(s)q_{m+1}b_{m+1}(s)}{1-}
\ \frac{p_{m+1}a_{m+1}(s)q_{m+2}b_{m+2}(s)}{1-\ldots}
\label{frcon2}
\end{equation}
\end{prop}
{\bf Proof.}
For classical continued-fraction convergence results, see Van Vleck's theorem in~\cite{l_pnm3}. Here a direct probabilistic argument is available. Under $0<p_i,q_i<1$, the first-step recurrence (\ref{frcon1}) identifies every finite
truncation with a taboo first-passage transform. Proposition~\ref{pr}
below proves that these transforms converge uniformly on the closed
right half-plane to $w_{m,m-1}(s)$, including at transform zeros.
This proves the expansion. $\Box$

In the next section, we will demonstrate that the continued fraction (\ref{frcon2}) converges uniformly
for $s\in \{p\in {\cal C} : Re(p)\ge 0\}$.

By symmetry considerations, analogous relations for
${w}_{m,n}(s), \ n>m$, are obtained.
\begin{lema}\label{l2}
\begin{description}
\item[a)]
$\displaystyle
{w}_{m,m+1}(s)=\frac{p_ma_m(s)}{1-q_mb_m(s){w}_{m-1,m}(s)}
$
\item[b)]
$
{w}_{m,n}(s)=\prod\limits_{i=m+1}^{n}{w}_{i-1,i}(s), \quad \; n>m,
$
\end{description}
\end{lema}
\begin{prop}\label{prp2}
For $s\in \{p\in {\cal C} : Re(p)\ge 0\}$ there is a continued fraction expansion
$$
 w_{m,m+1}(s)=\frac{p_ma_m(s)}{1-}
\ \frac{q_mb_m(s)p_{m-1}a_{m-1}(s)}{1-}
\ \frac{q_{m-1}b_{m-1}(s)p_{m-2}a_{m-2}(s)}{1-\ldots}
$$
\end{prop}

For PSSF, the first-passage transforms $\{w_{m,n}(s)\}$ also express the ordinary Laplace transforms of time-dependent state probabilities. Let us denote 
$$
P_{m,n}(t)=P_m\{t<\zeta,X(t)=n\},\qquad p_{m,n}(s)=\int_0^\infty 
e^{-st}\; P_{m,n}(t)\,dt,\quad
\mathop{\rm Re}s>0.
$$
  
From the total probability formula, we have
$$
p_{m,m}(s)=p_m a_m(s) w_{m+1,m}(s)p_{m,m}(s) + 
q_m b_m(s) w_{m-1,m}(s)p_{m,m}(s) +
$$
$$
+[1-p_ma_m(s)-q_mb_m(s)]/s,\ \ m\in Z
$$
The expected discounted occupation during the initial visit to $m$
is the transform of its holding-time survival function,
$[1-p_ma_m(s)-q_mb_m(s)]/s$. Summing the discounted return
excursions therefore gives
\begin{equation}\label{occupation_resolvent}
p_{m,m}(s)=\frac{[1-p_ma_m(s)-q_mb_m(s)]/s}
{1-p_m a_m(s) w_{m+1,m}(s)-q_m b_m(s) w_{m-1,m}(s)},\quad m\in Z.
\end{equation}
From the skip-free property, we have
$$
p_{m,n}(s)=w_{m,n}(s)p_{n,n}(s),\ \ \ m\not =n.
$$
The last two equalities, Lemmas \ref{l1}, \ref{l2}, and Propositions \ref{prp1}, \ref{prp2}, \ref{prp3} permit calculation with prescribed precision, provided passage-transform errors are propagated through the denominator. For real $x>0$, that denominator is bounded below by $1-p_ma_m(x)-q_mb_m(x)>0$.

In the cases examined so far, the set ${\cal S}$
was infinite. It remains to consider
the case $S=\{l,l+1,\ldots,n-1,n\}, \ l<n$.

Let $F_m^{[l,n]}(t)=P_m\{W_{\{l,\ldots,n\}}\leq t\}$,
$F_{m,k}^{[l,n]}(t)=P_m\{V_k\leq t,\ V_k\le V_{n+1},\ V_k\le V_{l-1}\}$, 
$l-1\le k\le n+1$, the latter function being a taboo probability with
states $n+1$ and $l-1$ forbidden.
Let us denote by
$w_m^{[l,n]}(s)$ and $w_{m,k}^{[l,n]}(s)$ the corresponding LSTs.
We chose these notations because, although redundant
(e.g., $w_{m,k}^{[l,n]}(s)=w_{m,k}^{[j,n]}(s)$, for $j\le l\le k<m$),
they seem clearer to us.
The upper indices indicate from which segment of integers the process takes values until absorption, and the lower ones:
the first---initial state of $X(t)$, the second---state where absorption occurs.
\begin{lema}\label{l3}
The following relations hold:
\begin{description}
\item[a)]
$\displaystyle w_{m,m-1}^{[l,n]}(s)=
\frac{q_mb_m(s)}{1-p_ma_m(s)w_{m+1,m}^{[l,n]}(s)},
\qquad m<n$,
\item[b)]
$w_{n,n-1}^{[l,n]}(s)=q_nb_n(s)$,
\item[c)]
$\displaystyle w_{m,m+1}^{[l,n]}(s)=
\frac{p_ma_m(s)}{1-q_mb_m(s)w_{m-1,m}^{[l,n]}(s)},
\qquad m>l$,
\item[d)]
$w_{l,l+1}^{[l,n]}(s)=p_la_l(s)$,
\item[e)]
$
{w}_{m,l-1}^{[l,n]}(s)=\prod\limits_{i=l}^{m}{w}_{i,i-1}^{[l,n]}(s),
$
\item[f)]
$
{w}_{m,n+1}^{[l,n]}(s)=\prod\limits_{i=m}^{n}{w}_{i,i+1}^{[l,n]}(s).
$
\end{description}
\end{lema}
{\bf Proof.}
Point a) is obtained in exactly the same way as 
(\ref{frcon1}), point b) is contained in the definition of PSSF.
Relations c) and d) are symmetric to relations a) and b) respectively,
and e) and f) are analogous to point b) of Lemma \ref{l1} and Lemma 
\ref{l2}
respectively.$\Box$

Taking into account that $w_{m}^{[l,n]}(s)=w_{m,l-1}^{[l,n]}(s)
+w_{m,n+1}^{[l,n]}(s)$, from Lemma \ref{l3} we obtain
\begin{prop}
The sojourn time of the PSSF in the set of states 
$\{l,\ldots,n\}$,
with the condition that $X(0)=m$, $l\le m\le n$, has the LST
$$
{w}_{m}^{[l,n]}(s)=\prod_{i=l}^{m}{w}_{i,i-1}^{[l,n]}(s)+
\prod_{i=m}^{n}{w}_{i,i+1}^{[l,n]}(s),
$$
where
$$
 w_{k,k-1}^{[l,n]}(s)=\frac{q_kb_k(s)}{1- }\ \
\frac{p_ka_k(s)q_{k+1}b_{k+1}(s)}{1-\ldots}\ \ \ \
\frac{p_{n-1}a_{n-1}(s)q_{n}b_{n}(s)}{1}
,$$
 and
\begin{equation}
 w_{k,k+1}^{[l,n]}(s)=\frac{p_ka_k(s)}{1-}\ \
 \frac{q_kb_k(s)p_{k-1}a_{k-1}(s)}{1-\ldots}\ \ \ \
 \frac{q_{l+1}b_{l+1}(s)p_{l}a_{l}(s)}{1}
\label{frcon3}
\end{equation}
\end{prop}
As can be seen, the LSTs $w_{k,k-1}^{[l,n]}(s), \ w_{k,k+1}^{[l,n]}(s)$
represent the $(n-k+1)$th and $(k-l+1)$th truncations of the continued
fractions corresponding to the LSTs $w_{k,k-1}(s)$ and $w_{k,k+1}(s)$,
respectively.
This allows us to use probabilistic arguments to estimate the relative
error of approximating the continued fraction corresponding to the
LST $w_{k,k-1}(s)$ (or $w_{k,k+1}(s)$).

{\bf Example.} The results obtained for the LSTs of the sojourn times
of the PSSF in a set of states allow us to determine the finite projection
of the PSSF. From Proposition \ref{pr_sem}, the following particular
results are obtained.

{\it
\noindent If the sequence of successive time moments $S=\{S_n\}$ at which
the PSSF $X(t)$ enters state $0$ is a transient renewal process, then the finite projection $X^*(t)$ of the PSSF $X(t)$ with respect to the renewal process $S$ is a regular PSSF with transition probabilities 
$$
p^*_i=  \left\{
\begin{array} {ll}
             p_i\rho_{i+1,i} & \mbox{if $i>0$} \\
             p_0\rho_{1,0}/(p_0\rho_{1,0}+q_0\rho_{-1,0})
                        & \mbox{if $i=0$} \\
             p_i/\rho_{i,i+1} & \mbox{if $i<0$},
\end{array}  \right.
$$
$$
q^*_i=  \left\{
\begin{array} {ll}
             q_i/\rho_{i,i-1} & \mbox{if $i>0$} \\
             q_0\rho_{-1,0}/(p_0\rho_{1,0}+q_0\rho_{-1,0})
                        & \mbox{if $i=0$} \\
             q_i\rho_{i-1,i} & \mbox{if $i<0$},
\end{array}  \right.
$$
where the probabilities $\rho_{i,j}$ are determined from the formulas:
$$
  \rho_{i,i-1}={P_i}(W_{\{i,i+1,\ldots\}}<+\infty)={w}_{i,i-1}(0)=
  \frac{q_i}{1-}
\quad \frac{p_iq_{i+1}}{1-}
\quad \frac{p_{i+1}q_{i+2}}{1-\ldots},
$$
$$
   \rho_{i,i+1}={P_i}(W_{\{i,i-1,\ldots\}}<+\infty)={w}_{i,i+1}(0)=
 \frac{p_i}{1-}\quad \frac{q_ip_{i-1}}{1-}
\quad \frac{q_{i-1}p_{i-2}}{1-\ldots}
$$
and the distribution functions of the transition durations between states in the projection remain unchanged:
$$A_i^*(t)=A_i(t), \qquad B_i^*(t)=B_i(t).$$

}

\subsection{\sloppy Relative Error of Approximation of Continued Fractions
by Their Truncations}
The truncation $w_{m,m-1}^{[m,k]}$ has order $k-m+1$, $k\geq m$. Here assume $0<p_i,q_i<1$ and proper, strictly positive holding times. We estimate the relative error using the finite-passage interpretation.
\begin{prop}\label{pr}
For $\mathop{\rm Re}s\geq0$ and $w_{m,m-1}(s)\ne0$, write the relative error as
$$
R_k(s)=\frac{w_{m,m-1}(s)-w_{m,m-1}^{[m,k]}(s)}{w_{m,m-1}(s)}.
$$
It has the product representation
$$
R_k(s)=w_{m,k+1}^{[m,k]}(s)w_{k+1,m}(s),
$$
which also defines $R_k(s)$ at transform zeros.
\begin{description}
\item[a)] The error converges uniformly to zero:
\begin{equation}\label{cf_uniform}
\sup_{\mathop{\rm Re}s\geq0}|R_k(s)|\leq
1-\frac{w_{m,m-1}^{[m,k]}(0)}{w_{m,m-1}(0)}\longrightarrow0.
\end{equation}
\item[b)] The factorization is
\begin{equation}\label{cf_factorization}
w_{m,m-1}^{[m,k]}(s)=w_{m,m-1}(s)[1-R_k(s)].
\end{equation}

\end{description}
\end{prop}
{\bf Proof.} A first passage from $m$ to $m-1$ either avoids $k+1$, or first reaches $k+1$ and then returns through $m$. The Markov-renewal property gives
$$
w_{m,m-1}=w_{m,m-1}^{[m,k]}+
w_{m,k+1}^{[m,k]}w_{k+1,m}w_{m,m-1},
$$
proving b) without division at a zero. For a transform of a positive measure, its modulus on the right half-plane is at most its value at zero. Consequently
$$
|R_k(s)|\leq w_{m,k+1}^{[m,k]}(0)w_{k+1,m}(0)
=1-\frac{w_{m,m-1}^{[m,k]}(0)}{w_{m,m-1}(0)}.
$$
Here $w_{m,m-1}(0)>0$ since $q_m>0$. Every successful finite-step passage has a finite maximum before hitting $m-1$, so the taboo hitting probabilities increase to $w_{m,m-1}(0)$. This proves a). The first-step recursions identify the taboo transforms with the continued-fraction truncations. $\Box$

Part a) gives more than absolute uniform convergence: the factorized relative error itself tends uniformly to zero.

To obtain a computable bound using coefficients only through $k$, introduce the finite fractions at real $x\geq0$:
\begin{equation}\label{frcon4}
h_{i,i-1}^{[i,k]}(x)=\frac{q_i b_i(x)}{1-p_i a_i(x)h_{i+1,i}^{[i+1,k]}(x)},
\qquad h_{k+1,k}^{[k+1,k]}(x)=1.
\end{equation}
Their denominators are positive and $0<h_{i,i-1}^{[i,k]}(x)\leq1$.
\begin{prop}\label{prp3}
For $\mathop{\rm Re}s\geq0$, with $x=\mathop{\rm Re}s$,
\begin{equation}\label{frcon5}
|R_k(s)|\leq\prod_{i=m}^k
w_{i,i+1}^{[m,i]}(x)h_{i+1,i}^{[i+1,k]}(x).
\end{equation}
The factors are determined by (\ref{frcon3}) and (\ref{frcon4}), using only coefficients through $k$.
\end{prop}
{\bf Proof.} The first-step recursion for the true downward transforms has the same form as (\ref{frcon4}). Starting with $w_{k+1,k}(x)\leq1$, backward monotonicity gives $w_{i,i-1}(x)\leq h_{i,i-1}^{[i,k]}(x)$ for $m+1\leq i\leq k$. Factor the downward passage into adjacent passages and use
$$
|R_k(s)|\leq w_{m,k+1}^{[m,k]}(x)w_{k+1,m}(x),\qquad
w_{m,k+1}^{[m,k]}(x)=\prod_{i=m}^k w_{i,i+1}^{[m,i]}(x).
$$
The terminal factor is bounded by one. This proves (\ref{frcon5}). $\Box$

At $x=0$ all the $h$-factors equal one, and the bound is $P_m\{V_{k+1}<V_{m-1}\}$. It need not tend to zero when upward escape has positive probability, although the actual error tends uniformly to zero by (\ref{cf_uniform}). If infinitely many upward passages cannot occur in finite time, the bound tends to zero for every $x>0$, uniformly on $\mathop{\rm Re}s\geq x$. For each fixed real $x\geq0$ it is the sharp supremum over admissible tails beyond $k$: return from $k+1$ can be made arbitrarily likely and fast. This real-frequency statement does not assert pointwise sharpness for complex moduli.

\bigskip
\centerline{\large \bf Conclusions and Observations}

On finite exit, the time spent in a given set is a sum with a random number of terms, each equal to the time spent in a state before the next jump. The same construction associates signed summands with the random walk: each new summand has a fresh distribution determined by the states before and after the corresponding jump, and the terms are summed until the first exit from the fixed set. For possibly infinite exit indices, the characteristic transform is restricted to finite exit. First-step factorization and bounds by unweighted hitting probabilities extend to these characteristic transforms; the positive-real Laplace comparison does not.

The relative-error estimates obtained here may be useful in determining rates of convergence for particular classes of continued fractions.

\newpage
\section*{ Conclusion}
\addcontentsline{toc}{section}{ Conclusion}

\thispagestyle{myheadings}  
\markright{\hspace{1cm} \it CONCLUSION}

The subject of study in this thesis is transient regenerative processes and, in particular, specific classes of regenerative processes, such as semi-Markov processes, homogeneous Markov chains with a countable set of states, birth-and-death processes, and queueing systems described by such processes. Additionally, certain types of continued fractions associated with the problems studied in the work are examined.

The purpose of the thesis is to develop and apply the finite projection method for the aforementioned processes. As a result of the research, using the methods of finite projection, Laplace transforms, renewal process theory, and continued fractions, new results have been obtained concerning the approximation of transient regenerative processes, skip-free semi-Markov processes, queueing systems of types $M/G/1$, $GI/M/1$, $M/M/1$, $M/M/n$. The structure of the finite projections of semi-Markov processes, Markov chains, and queueing systems of types $M/G/1$, $M/M/n$, etc., has been described.

Finite projection associates a recurrent regenerative process with a transient one and provides bounds for its approximation. Its best-approximation property concerns the metrics and comparison classes specified in Section~\ref{par_1_3}. Since the class of transient regenerative processes is very broad, it is not possible to explicitly determine the structure of their finite projection. This also applies to the finite projection of queueing systems. However, by examining narrower classes of processes, we can precisely describe the structure of finite projections, see, for example, Theorem \ref{p1_r9}.

The main results can be applied to the analysis of non-stationary queueing systems, the design of information processing systems, the study of convergence rates of continued fractions, etc.

The thesis presents new results; at the time of writing, some had no counterparts in the international literature known to the author.

\clearpage
\section*{Author's publications related to the thesis}
\addcontentsline{toc}{section}{Author's publications related to the thesis}
\markboth{Author's publications related to the thesis}{Author's publications related to the thesis}

The following list reproduces the author's publications related to the thesis, as listed in the original Romanian document.

\begin{enumerate}
\item Hlibiciuc A., Leahu A. Proiec\c{t}ia finit\u{a}, sisteme de a\c{s}teptare \c{s}i aproximarea lor [Finite projection, queueing systems and their approximation]. \textit{Buletinul \c{s}tiin\c{t}ific al Universit\u{a}\c{t}ii ``Politehnica'' din Timi\c{s}oara, matematic\u{a} fizic\u{a}}, UTT, Timi\c{s}oara, 1996, Tom 41(55), No.~1, pp.~38--43.

\item Hlibiciuc A., Leahu A. Proiec\c{t}ia finit\u{a} \c{s}i aproximarea proceselor de regenerare transiente [Finite projection and approximation of transient regenerative processes]. In: \textit{Scripta Scientiarum, Simposium septimum tiraspolense, Generalis topologiae et suae applicationum}. ``Tehnica'' UTM, Chi\c{s}in\u{a}u, 1996, pp.~120--123.

\item Hlibiciuc A. Eroarea de aproximare a proceselor de regenerare prin proiec\c{t}iile lor finite [The error in approximating regenerative processes by their finite projections]. In: \textit{Scripta Scientiarum, Simposium septimum tiraspolense, Generalis topologiae et suae applicationum}. ``Tehnica'' UTM, Chi\c{s}in\u{a}u, 1996, pp.~124--126.

\item Hlibiciuc A. Finite projection of Markov chains and M/G/1, GI/M/1 queueing systems. \textit{Scripta Scientiarum Mathematicarum}, Societas Topologiae Sistematum Subtilium et Appliationum Moldaviae, Chi\c{s}in\u{a}u, 1995, Tomus~1, fasciculus~1, pp.~43--51.

\item Hlibiciuc A., Leahu A. Markov Chains' Finite Projection. \textit{Analele \c{s}tiin\c{t}ifice ale Universit\u{a}\c{t}ii ``Ovidius'', Ser. Matematica}, Universitatea ``Ovidius'', Constan\c{t}a, 1995, Vol.~3, Fasc.~1, pp.~83--92.

\item Leahu A., Hlibiciuc A. Finite projection and dual of Birth-Death processes. \textit{Bulet. Acad. Sci. of Moldova, ser. Math.}, 1994, No.~3, pp.~5--14.

\item Leahu A.Ch., Hlibiciuc A.A. On the finite projection of nonstationary M/G/1 queueing systems. \textit{Bulet. Acad. Sci. Rep. Moldova, ser. Math.}, 1992, No.~1, pp.~81--83.

\item Hlibiciuc A., Leahu A. On the busy period of some computer science queueing systems. In: \textit{Proceedings of the International Conference on Microelectronics and Computer Science}. Chi\c{s}in\u{a}u, 1992, pp.~57--59.

\item Leahu A., Hlibiciuc A. Proiec\c{t}ia finit\u{a} a proceselor nesta\c{t}ionare de na\c{s}tere \c{s}i de moarte [Finite projection of nonstationary birth-and-death processes]. In: \textit{Lucr\u{a}rile Conferin\c{t}ei preg\u{a}titoare pentru Congresul matematicienilor rom\^ani de pretutindeni}. Bucure\c{s}ti, 19--21 November 1992, Vol.~2, pp.~171--173.

\item Hlibiciuc A. A relation between the moments of sojourn time of birth-death processes in the state set $\{N,N+1,N+2,\ldots\}$. In: \textit{XXIIIth Congress of the Rumanian American Academy of Science and Arts}\footnote{The original list prints XXIIIth; the 1993 Chi\c{s}in\u{a}u meeting was the eighteenth congress.} (July 13--16, 1993), ``Moldova: Opening of its culture and science for the West'', Section~III Mathematics and Physics. Chi\c{s}in\u{a}u, 1993, Vol.~2, p.~22.
\end{enumerate}

\clearpage
\thispagestyle{plain}
\section*{Disclaimer}
\noindent AI tools were used to assist translation from Romanian into English, language editing, simplification of the presentation, and mathematical checking. I reviewed the resulting text and checked the mathematical statements and arguments. Nevertheless, unintended changes introduced by these tools may have escaped my review, and some passages may not accurately reflect the claims of the original Romanian thesis. The Romanian text remains the authoritative record of the work presented in 1998; corrections and additions belong to this English edition. Responsibility for the present text and any remaining errors rests with me.
\par


\newpage
\begin{thebibliography}{99}
\addcontentsline{toc}{section}{Bibliography}  

\bibitem{l_pnm1} Allen A.O. Probability,  Statistics,  Queueing  Theory 
with Computer Science Applications. (New York: Academic Press, 1978)
  
\bibitem{part1_1} Asmussen S. Applied Probability and Queues.
(Wiley, New York,1987)

\bibitem{Asmussen_1} Asmussen, Soren. The heavy traffic limit of a
class of Markovian queueing models. Oper. Res. Lett. 6 (1987), no.
6, 301--306.
  
\bibitem{Asmussen_4} Asmussen, Soren. Conditioned limit theorems
relating a random walk to its associate, with applications to risk
reserve processes and the $GI/G/1$ queue. Adv. in Appl. Probab. 14
(1982), no. 1, 143--170. 

\bibitem{Asmussen_2} Asmussen, Soren. Queueing simulation in heavy
traffic. Math. Oper. Res. 17 (1992), no. 1, 84--111.

\bibitem{Asmussen_3} Asmussen, Soren. Busy period analysis, rare
events and transient behavior in fluid flow models. J. Appl. Math.
Stochastic Anal. 7 (1994), no. 3, 269--299.


\bibitem{part1_m1} Bailey N.T.J. A continuous time treatment of a
  simple queue using generating functions. J. Roy. Stat.  Soc., B16,
  p. 288--291 (1954). 


\bibitem{Banuelos_1} Banuelos, Rodrigo; Davis, Burgess.  Heat kernel,
  eigenfunctions, and conditioned Brownian motion in planar domains.
  (English) J. Funct. Anal. 84, No.1, 188-200 (1989).

\bibitem{Bass_1} Bass, Richard F. Burdzy, Krzysztof.  Conditioned
  Brownian motion in planar domains. (English) [J] Probab. Theory
  Relat. Fields 101, No.4, 479-493 (1995).

  
\bibitem{Bertoin_Doney} Bertoin, J., Doney, R.A. On conditioning a
  random walk to stay nonnegative. Ann. Probab. 22 (1994), no. 4,
  2152--2167.  

\bibitem{Bertoin_Doney2} Bertoin, J., Doney, R.A. On conditioning
  random walks in an exponential family to stay nonnegative.
  Seminaire de Probabilites, XXVIII, 116--121, Lecture Notes in
  Math., 1583, Springer, Berlin,   1994.  

\bibitem{Bertoin_1} Bertoin, Jean. Conditioning a reflected
  one-dimensional diffusion via its canonical decomposition. J.
  Theoret. Probab. 6 (1993), no. 2, 303--321. 


\bibitem{part1_2} Borovkov A.A. Stochastic Processes in Queueing Theory . (Nauka, Moscow, 1972) (in Russian) 

\bibitem{Borovkov_2} Borovkov A.A. Asymptotic Methods in Queueing Theory. (Nauka, Moscow, 1980) (in Russian) 
    
\bibitem{Chaumont_1} Chaumont, L.
  Sur certains processus de Levy conditionnes a rester positifs. (On
  certain Levy processes conditioned to stay positive). (French)
  Stochastics Stochastics Rep. 47, No.1-2, 1-20 (1994).

\bibitem{part1_14} Chernoff H. A measure of asymptotic efficiency
  for tests of a hypothesis based on the sum of observations.
  Ann. Math. Statist. v. 23 (1952), p. 493-507.

\bibitem{lm_1}  
  Chung K.L. Markov Chains with Stationary Transition
  Probabilities. (Springer-Verlag, New York, Berlin, etc., 1967)
  
\bibitem{part1_3} Cox D.R., Smith W.L. Renewal theory. ( Sov. radio, Moscow, 1967) (in Russian)

\bibitem{Cranston_1} Cranston, M., Zhao, Z.
  Conditional transformation of drift formula and potential theory
  for $\frac12\Delta +b(\cdot)\cdot \nabla$. (English)
  Commun. Math. Phys. 112, 613-625 (1987).


\bibitem{Cranston_2} Cranston, M. Conditional Brownian motion,
  Whitney squares and the conditional gauge theorem. (English)
  Stochastic processes, Proc. 8th Semin., Gainesville/Florida 
  1988, Prog. Probab. 17, 109-119 (1989).


\bibitem{Day_1} Day, Martin V.  Conditional exits for small noise
  diffusions with characteristic boundary. (English) Ann. Probab.
  20, No.3, 1385-1419 (1992).

\bibitem{Doob_1} Doob, J.L. Conditional Brownian motion and the
  boundary limits of harmonic functions. Bull. Soc. Math. France {\bf
  85 } (1957), 431--458.

\bibitem{Doney_1} Doney, R. A. A note on conditioned random walk. J.
  Appl. Probab.   20 (1983), no. 2, 409--412. 
  
\bibitem{lm_dyn} Dynkin E.B. Boundary theory of Markov processes (the
  discrete case), Russian Math. Surveys v. 24, (1969), p. 1-42.

\bibitem{risc_1} Embrechts P., Kl\"{u}ppelberg C., Some aspects
  of insurance mathematics. Theory Probab. Appl. vol. 38, N 2 (1994), pp. 262--295.

\bibitem{Fakinos_1} Fakinos D. Duality relations for certain single server
  queues. Queueing Systems, 4 (1989), p. 77-83.
  
\bibitem{c3} Feller W. An introduction to probability theory
  and its applications. vol. 1. (John Wiley \& Sons, New-York,
  1971) 
  
\bibitem{c2} Feller W. An introduction to probability theory
  and its applications. vol. 2. (John Wiley \& Sons, New-York,
  1971) 

\bibitem{Kesten_2} Ferrari P.A., Kesten H., Martinez S., Picco P.
  Existence of quasi-stationary distributions. A renewal dynamical
  approach. Ann. Probab. 23, No.2, 501-521 (1995).
  

\bibitem{Griffin_1} Griffin P.S., Verchota G.C., Vogel A.L.
  Distortion of area and conditioned Brownian motion. (English)
  Probab. Theory Relat. Fields 96, No.3, 385-413 (1993).
  
\bibitem{l_frac9} Jones W.B. Magnus A.
           Application of Stieltjes fractions to birth-death processes.
           In: Pade and rational approximation. Ed. by E.B. Saff, 
           R.S. Varga. New-York, Academic Press, 1977, pp. 173--179.
    
\bibitem{l_pnm3} Jones W.B., Thron W.J. Continued  Fractions.
          (London: Addison-Wesley Publishing Company, 1980)
  
\bibitem{hardy} Hardy G.H. Divergent series. (Oxford, 1949)

\bibitem{part1_7} Kalashnikov V.V. Qualitative Analysis of Complex Systems Behaviour by Test Functions Method.
  ( Nauka, Moscow, 1978) (in Russian)

\bibitem{kalas_2} Kalashnikov V.V., Vsekhsvyatsky S.Yu.  Estimates of the distribution of the time of the first failure for regenerating models / In the book: Mathematical theory of systems.  Nauka, Moscow, 1986, p. 59-75. (in Russian)

\bibitem{Kalashnikov_1} Kalashnikov V.V., Racev S.T. Mathematical methods for construction of queueing models. 
  ( Nauka, Moscow, 1988) (in Russian)

\bibitem{part1_13} Kantorovici L.V., Akilov G.P. Functional Analysis. (Moscow: Nauka, 1977) (in Russian)

\bibitem{l_frc1} Karlin S.
           A first course in stochastic processes.
           (Academic Press,
           New York,
           1968)

\bibitem{l_frac6} Karlin S., McGregor J.L.
           The differential equations of birth-and-death processes,
           and the Stieltjes moment problem.
           Trans. Amer. Math. Soc., 1957, v. 85, pp. 489--546.

\bibitem{l_frac7} Karlin S., McGregor J.L.
           Random walks. Illinois J. Math., 1959, vol. 3, pp. 66--81.
            
\bibitem{Karr}
{\raggedright
Karr A.F. Weak convergence of a sequence of Markov
  chains // Z.Wahrscheinlichkeitstheorie verw.  Gebiete.-1975. -
  B.33,H.1.- S. 41-48. 
\par}

\bibitem{lm_kem} Kemeny J.G., Snell J.L., Knapp A.W.
  Denumerable Markov Chains. (Springer-Verlag, New York, Berlin,
  etc., 1976)  

\bibitem{Keener_1} Keener, Robert W.
  Limit theorems for random walks conditioned to stay positive.
  Ann. Probab. 20, No.2, 801-824 (1992).

\bibitem{Kesten_1} Kesten, Harry A ratio limit theorem for (sub) Markov
  chains on $\{1,2,\cdots\}$ with bounded jumps. Adv. in Appl.
  Probab. 27 (1995), no. 3, 652--691.
  
\bibitem{Kerov_O} Kerov S.V., Orevkova O.A. Random processes with common cotransition probabilities, 
in: Studies in Mathematical Statistics IX, Zapiski Nauchnyh Seminarov LOMI 184 (1990), pp. 169--181 (in Russian).

%Random processes with common cotransition probabilities
%SV Kerov, OA Orevkova - Journal of Mathematical Sciences 68 (1994), 516-525.

\bibitem{Kolmogorov_base} Kolmogorov A.N. Grundbegriffe der
  Wahrscheinlichkeitsrechnung. (Berlin, 1933)
  
\bibitem{part1_10} Klimov G.P. Stochastic service systems.   ( Nauka, Moscow, 1966) (in Russian)

  
\bibitem{l_frac5} Korolyuk V.S., Sojourn time of a semi-Markov process in a fixed set of states.
           Ukr. Mat. Zh. 17.3 (1965): pp. 123-128. (in Russian)

\bibitem{l_pnm5} Korolyuk V.S., Turbin A.F. Semi-Markov processes and their applications. (Kiev: Naukova Dumka, 1976) (in Russian)

\bibitem{l_pnm7} Leahu A.Ch.  Finite Projection, Reversibility and Queueing Systems. Bulet. Acad. Sci. Rep. Moldova, ser. 
  Math., 2, 1992, pp. 89-93.
  
\bibitem{l_2:k1} Leahu A., Hlibiciuc A. 
        Finite projection and dual of Birth-Death processes.
        Bulet. Acad. Sci. of Moldova, ser. Math., nr. 3, 1994, pp.5-14. 
  
\bibitem{part1_5} Leahu A.Ch., Hlibiciuc A.A. On finite
  projection of  nonstationary M/G/1 queueing systems. Bulet.
  Acad. Sci. Rep. Moldova, ser. Math., 1, (1992) pp. 81-83.
  
\bibitem{part1_4} Leahu A.Ch., Matveev V.F. On finite projection of regenerative processes. Bulet.  Acad. \c{S}t. Rep.
  Moldova, ser. Mat., 3, 1991, pp. 74-77. (in Russian)

\bibitem{Liao_1} Liao, Ming, Pinsky, Mark.  Stochastic parallel
  translation for Riemannian Brownian motion conditioned to hit a
  fixed point of a sphere. (English) Geometry of random motion,
  Proc. AMS-IMS-SIAM Jt. Summer Res.  Conf., Ithaca/NY 1987, Contemp.
  Math. 73, 203-210 (1988). 
  
\bibitem{part1_11} Matveev V.F., Ushakov V.G. Queueing Systems  (Moscow: MGU, 1984) (in Russian)

\bibitem{McKean_1} McKean H.P. Excursions of a non-singular
  diffusion. Z. Wahrsch. Verw. Gebiete {\bf 1} (1963), 230--239. 

\bibitem{l_frac8} Murphy J.A., O'Donohoe M.R.
           Some properties of continued fractions  with applications
           in Markov processes.
           J. Inst. Math. Applics., 1975, v. 16, pp. 57--71.

\bibitem{Pakes_1} Pakes, Anthony G. On the age distribution of a Markov chain. 
  J. Appl. Probability 15 (1978), no. 1, 65--77. 


\bibitem{part1_8} Petrov V.V. Limit Theorems for Sums of Independent Random Variables.  (Nauka, Moscow, 1987) (in Russian)

\bibitem{part1_9} Petrov V.V. Sums of Independent Random Variables.
  (Springer--Verlag, Berlin, etc. 1975)

\bibitem{Pinsky_1} Pinsky R.G. A new approach to the Martin boundary
  via diffusions conditioned to hit a compact set. // The Annals of
  Probability, 1993, vol. 21, no 1, 453--481.


\bibitem{Poleshchuk_1} Poleshchuk, O.M.
  The convergence of a conditional random walk to a Brownian bridge.
  (English. Russian original) Russ. Math. Surv. 45, No.1, 225-226
  (1990); translation from Usp. Mat. Nauk 45, No.1(271), 187-188
  (1990). 

\bibitem{Poleshchuk_2} Poleshchuk, O.M.
  Limit theorems for a conditional random walk. (English. Russian
  original) Mosc. Univ. Math. Bull. 48, No.6, 40-43 (1993);
  translation from Vestn. Mosk. Univ., Ser. I 1993, No.6, 44-49
  (1993). 

\bibitem{Pollett_1} Pollett P. K. Reversibility, invariance and
  $\mu$-invariance. Adv. in Appl. Probab. 20 (1988), no. 3, 600--621


\bibitem{part1_m2} Prabhu N.U. Queues and Inventories. (John Wiley
  \& Sons, New-York, 1965)

\bibitem{Rachev} Rachev, S.T. Probability Metrics and Stability of
  Stochastic Models. (Wiley, New-York, 1991)

\bibitem{Renyi_1} Renyi, Alfred.  On a new axiomatic theory of
  probability. (English) Acta Math. Acad. Sci. Hung. 6, 285-335
  (1955). 

\bibitem{Renyi_2} Renyi, Alfred.  Ein neuer axiomatischer Aufbau der
  Wahrscheinlichkeitsrechnung.  (Hungarian) Magyar Tud. Akad., Mat.
  Fiz. Tud. Oszt.   Koezlemenyei 4, 369-428 (1954). 

\bibitem{Renyi_3} Renyi, Alfred.  On conditional probability spaces
  generated by a dimensionally ordered set of measures. (English.
  Russian summary)   Teor. Veroyatn. Primen. 1, 61-71 (1956). 
  
\bibitem{part1_12} Rusu Gh.I. Functional Analysis 1.  (USM, Chi\c{s}in\u{a}u, 1991) (in Romanian)

\bibitem{Vere-Jones_1} Seneta E., Vere-Jones D. On quasi-stationary
  distributions in discrete-time Markov chains with a denumerable
  infinity of states. J. Appl. Probab. {\bf 3} (1966), 403--434.

\bibitem{Seneta_1} Seneta E. Quasi-stationary behaviour in the
  random walk with continuous time.  Austral. J. Statist. {\bf 8} (1966), 92--98.

\bibitem{part1_6} Shiryayev A.N. Probability. (Springer--Verlag,
  Berlin,  etc., 1984)

\bibitem{reg_r2} Sigman K., Wolff R.W. A review of regenerative processes.
  SIAM Review 35(2) (1993), p. 269-288.

\bibitem{reg_r1} Smith W.L. Regenerative stochastic processes.
  Proc. R. Soc. London, Ser. A 232 (1955), p. 6-31.

\bibitem{Tavare_1} Tavare, S. Dual diffusions, killed diffusions, and the age
  distribution problem in population genetics. Theoret.  Population
  Biol. 16 (1979), no. 3, 253--266. 

\bibitem{Tavare_2} Tavare, S. Time reversal and age distributions. I.
  Discrete-time Markov chains. J. Appl. Probab. 17 (1980), no. 1,
  33--46. 

\bibitem{Veech_1} Veech W. The necessity of Harris' condition for the
  existence of a stationary measure.
  Proc. Amer. Math. Soc. {\bf 14} (1963), 856--860.

\bibitem{Versik_K} Ver\c{s}ik A.M., Kerov S.V. Locally semisimple algebras. Combinatorial theory and the $K_0$-functor, 
Itogi nauki i tekhn. ser. Sovrem. probl. mat. Nov. dostizh., 26, VINITI, Moscow, 1985, 3--56; J. Soviet Math., 38:2 (1987), 1701--1733.

\bibitem{Williams_1} Williams D. Path decomposition and continuity of
  local time for one-dimensional diffusions, I. Proc. London Math. Soc. (3) {\bf
  28} (1974), 738--768. 

\bibitem{Zhao_1} Zhao, Xuelei
  Conditioned Brownian motion in an angular domain. I: Lifetime.
  (English) [J] J. Beijing Norm. Univ., Nat. Sci. 29, No.4, 427-434
  (1993).
      
\end{thebibliography}
\enddocument